\pdfoutput=1
\documentclass[12pt, oneside, notitlepage]{amsart}
\usepackage[margin=3cm]{geometry}
\usepackage{amsmath,amssymb,amsthm,graphicx,mathrsfs,bbm,url}
\usepackage[mathcal]{euscript}
\usepackage{amsthm}
\usepackage{wrapfig}
\usepackage{enumitem}
\usepackage{mathtools}
\usepackage{comment}
\usepackage{stmaryrd}
\usepackage[utf8]{inputenc}
 \usepackage[T1]{fontenc}
\usepackage[usenames,dvipsnames]{color}
\usepackage[colorlinks=true,linkcolor=Red,citecolor=Green]{hyperref}
\usepackage{amsmath, amsthm} 
\usepackage{tikz}
\usetikzlibrary{arrows.meta, patterns}
\usepackage[noabbrev]{cleveref}
\usepackage[super]{nth}
\usepackage[open, openlevel=2, depth=3, atend]{bookmark}
\hypersetup{pdfstartview=XYZ}
\usepackage[font=footnotesize]{caption}

\usepackage{epstopdf}
 
\usepackage{hyperref}

\theoremstyle{plain}
\newtheorem{theorem}{Theorem}[section]
\newtheorem*{theorem*}{Theorem}

\newtheorem{lemma}[theorem]{Lemma}

\newtheorem{proposition}[theorem]{Proposition}
\newtheorem{corollary}[theorem]{Corollary}

\theoremstyle{definition}
\newtheorem{definition}[theorem]{Definition}

\theoremstyle{remark}
\newtheorem{remark}[theorem]{Remark}

\numberwithin{equation}{section}

\newcommand{\C}{\mathbb{C}}
\newcommand{\R}{\mathbb{R}}

\newcommand{\Z}{\mathbb{Z}}

\newcommand{\N}{\mathbb{N}}
\newcommand{\K}{\mathbf{k}}
\newcommand{\X}{\mathbf{X}}
\newcommand{\el}{\mathbf{l}}
\newcommand{\El}{\mathbf{L}}

\DeclareMathOperator{\Ell}{ell}

\DeclareMathOperator{\Op}{Op}
\DeclareMathOperator{\WF}{WF}

\DeclareMathOperator{\supp}{supp}

\newcommand{\be}{\begin{equation}}
\newcommand{\ee}{\end{equation}}

\newcommand{\diag}{\operatorname{diag}}

% ...

% ...

\title
[Spectral theory for frame flows on hyperbolic manifolds]
{Spectral theory of frame flows on closed hyperbolic manifolds}

\author{Julien Namazi}

\address{Université de Paris-Est Créteil , LAMA, 94010, Créteil, France.}

\email{julien.namazi@cnrs.fr}

\begin{document}

\begin{abstract}
We prove a resolvent estimate for the generator of the frame flow on hyperbolic manifolds away from vertical lines of resonances. A byproduct of the proof is an optimal essential spectral gap property for the generator, hence giving another proof of exponential mixing of frame flows with respect to the volume measure of the frame bundle. This extends the result of \cite{DGH2021} in dimension $3$ to any dimension. We make extensive use of the Borel-Weil calculus developed in \cite{CL24} to overcome difficulties of this higher-dimensional case. 
\end{abstract}

\maketitle
\setcounter{tocdepth}{1}
\tableofcontents

\section{Introduction}\label{section:introandmain}
Let $(M^n,g)$ be a closed manifold whose sectional curvature is strictly negative everywhere. The geodesic flow $(\varphi_t)$ on $SM:=\{(x,v)\in TM\,|\, g_x(v,v)=1\}$ is known to be \textit{Anosov} \cite{Anosov1967}. Namely, writing $X_{SM}$ the vector field generating $(\varphi_t)$, there exists a $(d\varphi_t)$ invariant splitting 
\begin{align}
    &T(SM)=E^u_{SM}\oplus E^s_{SM}\oplus \mathbb{R}X_{SM},
\end{align}
such that there exists $C>0$ and $\nu>0$ satisfying 
\begin{align}\label{eq:contractionintro}
    &\forall x\in SM,\forall v\in E^s_{SM}, \|d_x\varphi_t(v)\|_{\varphi_t(x)}\leq Ce^{-\nu t}\|v\|_x, \,\text{ if $t>0$},\\
    &\forall x\in SM,\forall v\in E^u_{SM}, \|d_x\varphi_{t}(v)\|_{\varphi_t(x)}\leq Ce^{\nu t}\|v\|_x, \,\,\,\,\text{ if $t<0$},\notag
\end{align}
where $C,\nu>0$ depend on the choice of metric on $SM$.

The Anosov property is of central importance in the realm of dynamical systems: the direction of the generator of the flow, not present for Anosov diffeomorphisms, makes the study of their \textit{speed of mixing} very challenging. In the present paper, we will deal with extensions of Anosov flows to $G$-principal bundles: such dynamical systems are called \textit{partially hyperbolic}.

\begin{definition}\label{def:partiallyhypcompact}
    Let $G$ be a compact Lie group and let $\pi:P\to SM$ be a $G$-principal bundle whose right action by $g\in G$ is written $R_g$. A flow $(\Phi_t)$ is said to be an extension of $(\varphi_t)$ to $P$ if the following conditions are fulfilled
    \begin{align}\label{eq:partiallyhypcompact}
        &\pi\circ \Phi_t=\varphi_t\circ \pi,\\
        & R_g\Phi_t(\cdot)=\Phi_t(R_g\cdot),\,\,\forall g\in G.\notag
    \end{align}
\end{definition}
As a consequence of the $G$-equivariance property, $\Phi_t$  restricts to an isometry  on each fiber for all $t\in \R$ when $G$ is equipped with its canonical bi-invariant metric. We will refer to $(\Phi_t)$ as an \textit{isometric extension} of $(\varphi_t)$.

In this paper, we will consider a specific case of isometric extensions of geodesic flows: the \textit{frame flow over $SM$}. On the orthonormal frame bundle $FM\to SM$, we may define the \textit{frame flow} $(\Phi_t)$. It corresponds to the parallel transport of a frame $(x,v,\mathbf{e}_1,\dotsc,\mathbf{e}_{n-1})$ along the geodesic starting at $(x,v)$, see Section \S\ref{subsec:frameflow} for precise definitions. This flow is an $SO(n-1)$ extension of the geodesic flow $(\varphi_t)$ on $SM$. We write $X_{FM}$ for the generator of this flow. By construction of the frame flow and the assumption that $(M,g)$ is closed and has everywhere negative sectional curvature, one can show that $(\Phi_t)$ satisfies (\ref{eq:partiallyhypcompact}) with $G=SO(n-1)$. The frame flow is thus partially hyperbolic. We have the splitting 
\begin{align}\label{eq:introframeflowsplitting}
    T(FM)=E^u_{FM}\oplus E^s_{FM}\oplus \mathbb{R}X_{FM}\oplus \mathbb{V}_{FM},
\end{align}
where the \textit{vertical bundle} $\mathbb{V}_{FM}$ is tangent to the fiber direction of $FM\to SM$ and is invariant under $d\Phi_t$. The bundles $E^{u,s}_{FM}$ are also invariant subbundles additionally satisfying the usual contraction and expansion properties 
\begin{align}\label{eq:contractionintroFM}
    &\forall p\in FM,\forall v\in E^s_{FM}, \|d_x\Phi_t(v)\|_{\Phi_t(p)}\leq Ce^{-\nu t}\|v\|_x, \,\text{ if $t>0$},\\
    &\forall p\in FM,\forall v\in E^u_{FM}, \|d_x\Phi_{t}(v)\|_{\Phi_t(p)}\leq Ce^{\nu t}\|v\|_x, \,\,\,\,\text{ if $t<0$}.\notag
\end{align}

\textbf{Statement of the main result.} Let  $M:=\Gamma\backslash \mathbb{H}^n$ be a closed hyperbolic $n$-manifold, where $\Gamma$ is a discrete, torsion-free, co-compact subgroup of orientation preserving isometries of $\mathbb{H}^n$. Its unit tangent bundle is written $SM$ and its orthonormal frame bundle is $FM$.

We write $\mathbb{N}=\mathbb{Z}_{\geq0}$.
Let $\varepsilon\in(0,1)$ and define 
\begin{align*}
    \mathbb{C}_{\varepsilon}:=\C\setminus\bigcup_{k\in\N}\bigl\{z\in \C:\Re(z)\in [\tfrac{1-n}{2}-k-\varepsilon,\tfrac{1-n}{2}-k+\varepsilon]\bigr\},
\end{align*}
which is illustrated in Figure \ref{figure:figure_1}.
In addition, we write $\mathbb{C}_0:=\C\setminus \{\Re(z)\neq \frac{1-n}{2}-k, k\in\mathbb{N}\}$. We will prove resolvent estimates in the connected components of $\mathbb{C}_{\varepsilon}$ inside $\{\Re(z)\leq 0\}$, namely
\begin{align}\label{eq:band_def}
    &\mathcal{B}_0(\varepsilon):=\biggl\{z\in \C:\Re(z)\in \biggl(\frac{1-n}{2}+\varepsilon,0\bigg]\biggr\}\\
    &\mathcal{B}_k(\varepsilon):=\biggl\{z\in \C:\Re(z)\in \biggl(\frac{1-n}{2}-k-1+\varepsilon,\frac{1-n}{2}-k-\varepsilon\biggr)\biggr\}
\end{align}
The resolvent estimates will be performed on the functional space $\mathcal{H}^{m,s}\subset \mathcal{D'(FM)}$ for $m\in \mathbb{R}$ and $s\in \Z_{>0}$ (the statement could also be proved for $s\in \R_+^*$ without additional efforts, but this choice simplifies the statement of Theorem \ref{thm:mainthm}). These spaces have been introduced in \cite{DGH2021,CL24}, they generalize the anisotropic spaces $\mathcal{H}^s$ of \cite{FS11} tailored to Anosov flows in the case of principal bundle extensions: the additional $m\in \R$ parameter measures Sobolev regularity in the direction of $\mathbb{V}_{FM}$. 

Since the frame flow preserves the canonical volume form on $FM$, the resolvent
\begin{align*}
    (X_{FM}+z)^{-1}=\int_{0}^{+\infty} e^{-zt}e^{-tX_{FM}}dt,
\end{align*}
is a bounded operator on $L^2(FM)$ for $\Re(z)>0$. As a consequence, $\mathbf{R}_+:z\mapsto (X_{FM}+z)^{-1}$ can be seen as a holomorphic function with values in $\mathcal{L}(C^{\infty}(FM),\mathcal{D}'(FM))$
\begin{theorem}[Meromorphic continuation \& Boundedness estimates]\label{thm:mainthm} (i) The resolvent associated to the frame flow $z\mapsto \mathbf{R}_+(z)$ extends as a meromorphic family of operators $z\mapsto (X_{{FM}}+z)^{-1}\in \mathcal{L}(C^{\infty}(FM),\mathcal{D}'({FM}))$ on $\mathbb{C}_0$. Moreover, for every $\varepsilon>0$, there only exists a finite number of poles in each band $\mathcal{B}_k(\varepsilon)$
    
(ii) For any $m\in \mathbb{R}$ and $s\in \mathbb{Z}_{>0}$ there exists a Hilbert space $\mathcal{H}^{m,s}({FM})\subset \mathcal{D}'(FM)$, containing $C^{\infty}({FM})$ as a dense subspace, such that for $z\in \mathbb{C}_0$ not a pole with $\Re(z)>-s$ and $z\in \mathcal{B}_k(\varepsilon)$
    \begin{align}\label{eq:mainthmboundedness}
        (X_{{FM}}+z)^{-1}\in \mathcal{L}(\mathcal{H}^{m-\lfloor \frac{1-n}{2}-k \rfloor,s},\mathcal{H}^{m,s}).
    \end{align}
    The boundedness property is refined by the following estimate on each $\mathcal{B}_k(\varepsilon)$ with $k\in \mathbb{N}$. For every $\varepsilon\in(0,1)$, $s\in \mathbb{Z}_{>0}$ and $z\in \mathbb{C}_{\varepsilon}$ with $\Re(z)>-s$, there exists $C_{\varepsilon,k}>0$ such that 
    \begin{align}\label{eq:mainthmestimate}
        \|(X_{FM}+z)^{-1}\|_{\mathcal{L}(\mathcal{H}^{m-\lfloor \frac{1-n}{2}-k \rfloor,s},\mathcal{H}^{m,s})}\leq C_{\varepsilon,k}\langle \Im(z) \rangle^{-\lfloor \frac{1-n}{2}-k \rfloor+s} ,\,\,z\in \mathcal{B}_k(\varepsilon),|\Im(z)|\gg 1
    \end{align}
    
\end{theorem}
An illustration of this Theorem is given in Figure \ref{figure:figure_1}. We note that our theorem does not give the exact position of the blue poles: \textit{a priori}, we only know that there are finitely many resonances in each connected region in $\mathbb{C}_{\varepsilon}$.
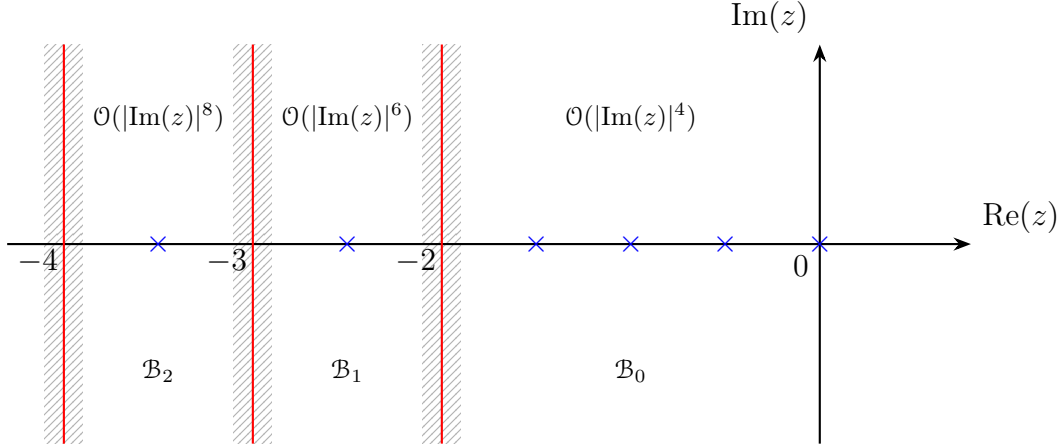
\begin{figure}[htbp]
    \centering
    \begin{tikzpicture}[>=Stealth, xscale=2.5, yscale=1.2]

        \tikzset{
            hatch zone/.style={pattern=north east lines, pattern color=gray!60, draw=none}
        }
        \fill[hatch zone] (-2.1, -2.2) rectangle (-1.9, 2.2);
        \fill[hatch zone] (-3.1, -2.2) rectangle (-2.9, 2.2);
        \fill[hatch zone] (-4.1, -2.2) rectangle (-3.9, 2.2);

        \draw[->, thick] (-4.3, 0) -- (0.8, 0) node[above right] {$\text{Re}(z)$}; 
        \draw[->, thick] (0, -2.2) -- (0, 2.2) node[above left] {$\text{Im}(z)$};
        \node[below left, inner sep=4pt] at (0,0) {$0$}; 

        \draw[thick, red] (-2, -2.2) -- (-2, 2.2);
        \node[below left, text=black, inner sep=2pt] at (-2, 0) {$-2$};
        
        \draw[thick, red] (-3, -2.2) -- (-3, 2.2);
        \node[below left, text=black, inner sep=2pt] at (-3, 0) {$-3$};
        
        \draw[thick, red] (-4, -2.2) -- (-4, 2.2);
        \node[below left, text=black, inner sep=2pt] at (-4, 0) {$-4$};

        \tikzset{
            blue cross/.style={color=blue, thick, inner sep=0pt, minimum size=4pt}
        }
        
        \node[blue cross] at (0,0) {$\times$}; 
        \node[blue cross] at (-0.5, 0) {$\times$};
        \node[blue cross] at (-1.0, 0) {$\times$};
        \node[blue cross] at (-1.5, 0) {$\times$};
        \node[blue cross] at (-2.5, 0) {$\times$};
        \node[blue cross] at (-3.5, 0) {$\times$};

        \node[align=center, font=\footnotesize] at (-1.0, 1.4) {$\mathcal{O}(|\text{Im}(z)|^4)$};
        \node[align=center, font=\footnotesize] at (-2.5, 1.4) {$\mathcal{O}(|\text{Im}(z)|^6)$};
        \node[align=center, font=\footnotesize] at (-3.5, 1.4) {$\mathcal{O}(|\text{Im}(z)|^8)$};

        \node[align=center, font=\footnotesize] at (-1.0, -1.4) {$\mathcal{B}_0$};
        
        \node[align=center, font=\footnotesize] at (-2.5, -1.4) {$\mathcal{B}_1$};
        
        \node[align=center, font=\footnotesize] at (-3.5, -1.4) {$\mathcal{B}_2$};

    \end{tikzpicture}
    \caption{Distribution of the poles of $z\mapsto(-X_{FM}-z)^{-1}$ on $\mathbb{C}_{\varepsilon}$ (corresponding to the unshaded region). Here $n=5$ so $\tfrac{n-1}{2}=2$. The resolvent estimate in $\mathcal{B}_k$ is with respect to the operator norm $\mathcal{L}(\mathcal{H}^{m-\lfloor \frac{1-n}{2}+k\rfloor ,2+k},\mathcal{H}^{m,2+k})$.}
    \label{figure:figure_1}
\end{figure}

Theorem \ref{thm:mainthm} will be a consequence of the key technical estimate in Theorem \ref{thm:mainthmfiberwise}. We remark that the estimate (\ref{eq:mainthmestimate}) may be far from optimal; we did not attempt to sharpen it further. For instance, a consequence of the work \cite[Theorem 1.7]{FT} is that the generator of a contact Anosov flow has uniformly bounded resolvent away from vertical lines of resonances. We believe that the bound (\ref{eq:mainthmestimate}) could be upgraded to a uniform bound using the techniques in \cite{FT}. We further think that the resolvent $(-X_{FM}-z)^{-1}$ could even be in $\mathcal{L}(\mathcal{H}^{m,s})$: the $m$ parameter represents a Fourier parameter in the direction of the fiber of $FM\to SM$, it could be that the tools in \cite{FT} appliy here to show that the resolvent does not downgrade the regularity in this direction.

A consequence of the resolvent estimate combined with the spectral gap property in Theorem \ref{thm:mainthm} is the exponential mixing of the frame flow for the natural normalized volume on $FM$. 
\begin{corollary}
\label{thm:expmixing}
    Let $\mu_{FM}$ be the natural smooth volume on $FM$ which writes locally as a product of the Haar measure on $SO(n-1)$ and of the Liouville measure on $SM$. Then the frame flow $(\Phi_t)$ is exponentially mixing: there exist $\gamma>0$ and for all $f,g\in C^\infty(M)$ a constant $C_{f,g}>0$ such that for all $t>0$ 
    \begin{align}\label{eq:expmixing}
        \biggl|\int (f\circ \Phi_{-t})g\,\,d\mu_{FM}-\int f d\mu_{FM}\int g d\mu_{FM}\biggr|\leq C_{f,g}e^{-\gamma t}
    \end{align}
\end{corollary}
For the deduction of Corollary \ref{thm:expmixing} from Theorem \ref{thm:mainthm}, the interested reader may refer to \cite[\S 4.5]{DGH2021} where this is done in the three-dimensional case. Given the estimate (\ref{eq:mainthmestimate}) and the spectral gap property, their proof applies in our setting.

\begin{remark}
    The rate of decay $\gamma$ is unknown here. From the quantum-classical correspondence in \cite{DyatlovFaureGuillarmou2015}, and following the arguments in \cite{DGH2021} in dimension 3, $\gamma$ should be explicitly given in terms of the first non-trivial eigenvalue of the positive scalar Laplacian $\Delta$ on $M$. Although interesting, we chose not to prove this. Outside of the locally symmetric case this correspondence is unknown, and we want this paper to be as self-contained as possible.  
\end{remark}

\textbf{Existing results.} There is a particular class of Anosov flows which are known to exhibit important \textit{chaotic} behaviour.
\begin{definition}\label{def:contact}
    An Anosov flow generated by a vector field $X$ is said to be contact if there exists a smooth $1$-form $\alpha$ such that $\alpha(X)=1$, $\ker(\alpha)=E^u\oplus E^s$, and the restriction $d\alpha_{|E^u\oplus E^s}$ is non-degenerate.
\end{definition}
\raggedbottom
A celebrated result is the exponential mixing of contact Anosov flows with respect to the Liouville form $\alpha\wedge (d\alpha)^{n-1}$ (if $\dim(M)=2n-1$) \cite{Dol98,Liv,NZ14,FT}.
In particular, the geodesic flow is exponentially mixing on closed negatively curved manifolds with respect to the Liouville measure on their unit tangent bundles. The more recent approach in \cite{NZ14,FT} uses tools from microlocal and semiclassical analysis, whereas the approaches in \cite{Dol98,Liv} are more dynamically oriented, relying on Markov partitions \cite{Dol98} or Lasota-Yorke inequalities \cite{Liv}. In this paper, we will adopt the more recent point of view involving microlocal analysis.

The key ingredient to establish exponential mixing of $(\varphi_t)$ is to prove \textit{resolvent estimates} for the unbounded operator $X$ on the \textit{anisotropic Sobolev space} $\mathcal{H}^s(SM)$, a functional space adapted to the hyperbolic behaviour of the flow. Together with the spectral gap property, one obtains the exponential decay of correlations. In the context of smooth r-normally hyperbolic trapping, the work \cite{Dy16} shows spectral gap and resolvent estimate properties for operators whose symbol has a Hamiltonian flow presenting a structure similar to the Anosov case. In fact, the ideas in \cite{Dy16} apply to treat exponential mixing for contact Anosov flows whose stable and unstable bundles are smooth. However, the smoothness assumption is known to be very rigid, and only holds for Anosov flows built from the geodesic flow on a hyperbolic manifold \cite{HK90,BFL90}. 

In the recent paper \cite{CG21}, the authors were able to prove exponential mixing for contact Anosov flows in dimension $3$. In this case $E^{u,s}$ are known to be $C^{1,\alpha}$ regular, which allows after some work to use the semiclassical proof of \cite{Dy16}. The present work will assume smoothness of $E^{u,s}$, but the flows under consideration will present additional neutral directions, see (\ref{eq:introframeflowsplitting}). For a general isometric extension of an Anosov flow (see Definition \ref{def:partiallyhypcompact} where the geodesic flow on $SM$ is replaced with a general Anosov flow) with suitable assumptions on $P$, $(\Phi_t)$ is known to be rapid mixing - the decay of correlation is faster than any power of $t$ - for measures which are locally the product of the Haar measure and an equilibrium state for the base flow by \cite{PZ25}. The work \cite{CL24} proved this with microlocal methods for volume-preserving flow under more explicit geometric assumptions on $P$ (non-torsion if $\dim(G)=1$, semisimple otherwise), on the base flow $(\varphi_t)$ (a non-joint integrability type condition) and $(\Phi_t)$ (requiring ergodicity if $\dim(G)\geq 2$).

Since the work \cite{BG80}, much effort has been made to understand the dynamical property of this flow under the negative curvature assumption on $M$, see \cite{CLMS24a,CLMS24b}. For instance, it is known since \cite{Moore1987} that it is exponentially mixing for the natural volume measure if $M$ is a hyperbolic manifol; however, the proof relies heavily on representation theory.

In \cite{DGH2021}, Guillarmou and Delarue (formerly Küster) were able to incorporate the microlocal toolbox from \cite{FS11} in the setting of the frame flow on hyperbolic 3-manifolds. They were able to prove exponential mixing for the normalized volume using those techniques. Note that $FM\to SM$ is a circle bundle if $n=3$, which allows the use of a fiberwise Fourier decomposition of elements in $\mathcal{D}'(FM)$, namely distributions on $FM$. The spectral gap property for Anosov flows then becomes a gap which is uniform in the Fourier parameter. Higher dimensions are more difficult to tackle due to the fact that $FM\to SM$ is in general an $SO(n-1)$-principal bundle. In the recent work \cite{CL24}, Cekić and Lefeuvre constructed the so-called \textit{Borel-Weil calculus}. This is a (semiclassical) pseudodifferential calculus that is adapted to $G$-equivariant pseudodifferential operators on $FM$; $X_{FM}$ belongs to this class. To such an operator we may associate a family of operators acting on specific line bundles. We will heavily use the Borel-Weil calculus of \cite{CL24} to tackle the case of higher-dimensional hyperbolic manifolds. Different approaches using Dolgopyat's inequality led to the proof of exponential mixing of the frame flow in the non-compact case on hyperbolic manifolds \cite{SW21,LP23,LiPanSarkar2026}.

\textbf{The strategy.} In order to prove Theorem \ref{thm:mainthm}, we will perform a Fourier decomposition of the generator $X_{FM}$. In the three-dimensional case \cite{DGH2021}, this gives a family of first order differential operators $X_k$ on line bundles $L^{\otimes k}$ over $SM$ with $k\in \Z$. In higher dimensions, we use the Borel-Weil calculus of \cite{CL24} to obtain a family of connections $\X_\K$ over line bundles $\El^\K$ over the \textit{flag bundle} $F\to SM$. We wish to estimate the norm of $(-X_{FM}-z)^{-1}$ on the spaces $\mathcal{H}^{m,s}(FM)$, see Definition \ref{def:anisotorpicnormFM}. It thus boils down to an estimate on the norm of each component $(-\X_\K-z)^{-1}$, well-defined by the calculus of Cekić-Lefeuvre \cite{CL24} and \cite{FS11}. The technical estimate that is needed is in Theorem \ref{thm:mainthm}. To do so, the reasoning goes as follows. 
\begin{itemize}
    \item  We first prove the estimate on $\{\Re(z)>-1\}$. Assuming the estimate (\ref{eq:fiberwiseestimates}) does not hold, we obtain a sequence of quasimodes $(u_h)\in \mathcal{H}^N_h(F,\El^{\K(h)})$ and of complex numbers $(z_h)$ satisfying 
    \begin{align}
        (-h\X_{\K(h)}-z_h)u_h=o(h^{2})
    \end{align}
    where $h$ is either $1/\langle\Im(z)\rangle$ or $\langle \K \rangle^{-1}$: this is the \textit{high-frequency regime} or \textit{high tensor power regime} dichotomy.
    \item As $h\to 0$, we then study the microlocal properties of $(u_h)$. As done in \cite{Dy16}, \cite{CG21} or \cite{DGH2021}, we study semiclassical measures $\mu$ associated to $(u_h)$. We introduce them in the setting of the Borel-Weil calculus in Section \ref{section:BWcalculus}. Since we assume $\Re(z_h)/h>-1$, we are able to prove a \textit{horocyclic invariance} of $u_h$, see Proposition \ref{prop:horo_invariance_quasimode_re_larger_minus1}. We then deduce from there a \textit{Lipschitz bound} on $\mu$, see Lemma \ref{lemma:lipschitz_bound}. We note that this kind of contradictory argument involving semiclassical measures is not new and traces back to the works\footnote{This was pointed out to us by Thibault Lefeuvre.} \cite{Burq1997,BurqSchlenker1998}.
    \item In the course of the proof, we introduce particular unstable and stable vector fields whose corresponding Hamiltonian vector fields are \textit{transverse to the trapped set} region where $u_h$ microlocalizes. Due to the presence of the line bundles $\El^\K$, the underlying \textit{classical dynamics} in the Borel-Weil calculus may not agree with the one originating from the standard symplectic form. We thus have to be more careful than in the pure contact flow case. 
    \item Combining the previous result yields a contradiction concerning the existence of the quasimodes $u_h$; $\mu$ has to simultaneously satisfy contradictory estimates, namely the Lipschitz bound and the propagation property with respect to the Hamiltonian flow of $\X_\K$.
    \item Using the \textit{horocyclic operators} introduced in (\ref{eq:horocyclic_op}), we are able to deduce resolvent estimates everywhere away from $\varepsilon$-neighborhoods of vertical lines $\{\Re(z)=\frac{1-n}{2}-k\}$ through an induction procedure on $k$.
    \end{itemize}
\textbf{Perspectives.} It is unclear for the moment if we should expect that $z\mapsto (-X_{FM}-z)^{-1}$ is meromorphic on the vertical lines $\{\Re(z)=\frac{1-n}{2}-k, k\in \mathbb{N}\}$. This issue appears in \cite{DGH2021} and has been reformulated in terms of the spectral theory of the Bochner Laplacian $\Delta_N$ on the space of traceless symmetric tensors $\bigotimes^N_{S,0}T^*M$, using the quantum-classical correspondence. Answering this question requires precise asymptotics on $\lambda_1(\Delta_N)$. However in general there is no reason to expect meromorphic continuation. For example, the trivial extension of an Anosov flow has $0$ as a pole of infinite order. The ergodic properties of the frame flow may be a key property allowing to have meromorphic continuation.

Another perspective is to find a suitable way to study the spectrum of $-X_{FM}+V$ where $V\in C^{\alpha}(FM)$. If $V$ is constant on each fibers of $FM$ then the Fourier decomposition reasoning may work, up to incorporate the approach of \cite{GuillarmouDePoyferre2022} to tackle the Hölder regularity. This may be interesting in order to understand the speed of mixing for equilibrium states which are locally product of the Haar measure and an equilibrium state of the geodesic flow. But if $V$ has non-trivial Fourier components, defining resonances does not seem trivial and it seems to be an interesting problem to address. It could potentially lead to the construction of other interesting equilibrium states for the frame flow by following the ideas of \cite{Humbert2025} in the case of transitive Anosov flows.\\
Finally, the exponential mixing of frame flows on manifolds with strictly $\frac{1}{4}$-pinched negative curvature still seems out of reach. In this work, the smoothness of $E^{u,s}$ was crucial. Although the approach of \cite{CG21} could possibly tackle the case where $E^{u,s}_{FM}$ are only $C^{1+\varepsilon}$, the works \cite{BCG95} and \cite{B25} proved that this implies smoothness of the latter bundles. The non-degeneracy of the curvature of the dynamical connection $d\Theta_{dyn}$ on $FM\to SM$ is non-degenerate on $E^u_{FM}\oplus E^s_{FM}$, together with the contact property of the base flow $(\varphi_t)$, nevertheless appears to be a key mechanism for exponential mixing (see Section \ref{section:framebundleandhyp} for the notations).\\

\textbf{Acknowledgments.} The author is grateful to his PhD advisor, Mihajlo Cekić, for his guidance during this project. He also thanks Thibault Lefeuvre for interesting discussions concerning the Borel-Weil calculus and horocyclic invariance of quasimodes.

The author thanks Louis-Brahim Beaufort and Tristan Humbert for several interesting discussions related to the content of the paper.

Finally, the author thanks Benjamin Delarue (formerly Küster) for his suggestions that lead to the improvement of the presentation.

\section{Analytic preliminaries: the Borel-Weil calculus}\label{section:BWcalculus}
 We provide a necessary overview of the construction in \cite{CL24} of the Borel-Weil calculus. Then, we define semiclassical measures in this setting. 
\subsection{Representations of $SO(n)$ and the Borel-Weil theorem }\label{subsection:BWtheorem}

Let $G$ be a compact connected \textit{semisimple} Lie group of dimension $m$. Choose $T$ a maximal torus of $G$ ($\mathfrak{t}$ the corresponding Lie subalgebra of $\mathfrak{g}$), that is a maximal compact connected Abelian subgroup. The dimension of $T$ as a submanifold of $G$ is called the rank and is written $d$.

Equivalence classes of irreducible unitary representations of $G$, whose set is written $\widehat{G}$, can be labeled by a subset of $(i\mathfrak{t})^*$ (linear forms $\mathfrak{t}\to i\R$). Elements of this set are the so-called \textit{highest weight} of $G$. Precisely, from a choice of simple roots one may form the so-called (closed) positive Weyl chamber $\mathcal{W}_+\subset (i\mathfrak{t})^*$. Choose a family of analytically integral weights (\textit{i.e.} elements of $(i\mathfrak{t})^*$ that descend to a character $T\to \mathbb{S}^1$ through the exponential map), denoted $(\lambda_i)_{1\leq i \leq d}\in (i\mathfrak{t}^*)^d$, that generates analytically integral weights in $\mathcal{W}_+$. Then we can construct the following labeling
\begin{align}\label{eq:labelling_irreps}
    \phi(\mathbf{k}):=\sum_{i=1}^dk_i\lambda_i,\,\mathbf{k}\in \mathbb{Z}_+^d.
\end{align}
Define an equivalence relation on $\mathbb{Z}_+^d$ by declaring that $\mathbf{k}\sim\mathbf{k'}$ if $\phi(\mathbf{k})=\phi(\mathbf{k}')$. The set of corresponding equivalence classes is written $\Lambda$ and can be seen to index complex irreducible representations of $G$.
As this will be useful in Lemma \ref{lemma:construction_particular_vf}, let us specify this in the case of $SO(n)$ for $n\geq 3$. Its rank is $\lfloor \frac{n}{2}\rfloor$ and the standard Cartan subalgebra of $SO(n)$ is given by $\mathfrak{t}=\mathrm{span}(\{E_{2j-1, 2j} - E_{2j, 2j-1}\}_{1\leq i \leq d})$, with $(E_{i,j})$ the usual elementary matrices. Write $(e_j)_{1\leq j \leq d}$ for the corresponding dual basis of $\mathfrak{t}^*$. Let us introduce, for the $n=2d+1$ case
\begin{align}\label{eq:fundament_weight_odd}
    \omega_j=i(e_1+\cdots+e_j),\, 1\leq j\leq d
\end{align}
and for $n=2d$
\begin{align}\label{eq:fundament_weight_even}
    \omega_j&=i(e_1+\cdots+e_j),\, 1\leq j\leq d-2\notag\\
    \omega_{d-1}&=i(e_1+\cdots+ e_{d-1}-e_d), \\
    \omega_{d}&=i(e_1+\cdots+ e_{d-1}+e_d). \notag
\end{align}
The following result is standard, and shows that $\Lambda=\mathbb{Z}_+^d$ for the case $G=SO(n)$.
\begin{proposition}\label{prop:labelling_highest_weights}

The finite-dimensional irreducible complex representations of $SO(n)$ have highest weight
$\lambda=\sum_{i=1}^d m_i\omega_i$ $m_i\in\mathbb Z_{\ge0}$ with $m_i\in \mathbb{Z}_+$.
\end{proposition}
\begin{proof}
    This is a standard result, see for instance \cite[\S 19.2, \S 19.4]{FultonHarris1991}.
\end{proof}
\subsubsection{The Borel-Weil correspondence.} This theorem establishes an isomorphism between each complex irreducible representations of $G$, indexed by $\mathbf{k}$, and the corresponding complex canonical representation of $G$ (called the Borel-Weil representation). The latter can be described as follows. To $\K\in \Lambda$ can be associated an analytically integral weight $\alpha_{\K}\in (i\mathfrak{t})^*$. It has the property of descending through the exponential map to a character $\gamma_{\K}:T\to \mathbb{S}^1$. Thus, one can form the associated line bundle 
\begin{align}\label{eq:associated_line_bw}
\mathbf{J}^{\K}:=G\times_{\gamma_{\K}}\mathbb{C},
\end{align}
by the quotient of $G\times\mathbb{C}$ by $(g,z)\sim (g\cdot t,\gamma_{\K}(t)^{-1}z)$ for $t\in T$. It can be shown that $G/T$ has a holomorphic structure with respect to which $\mathbf{J}^{\K}$ is a holomorphic line bundle, see \cite{Sepanski2007}. It can be endowed with the canonical $G$-action given by left-translation $g\cdot[hT,z]=[ghT,z]$. The Borel-Weil representation is formed by the data of the finite-dimensional vector space of holomorphic sections of $\mathbf{J}^\K$ written $H^0(G/T,\mathbf{J}^{\gamma_{\K}})$ and of the mapping 
\begin{align}\label{equiv}
    \Phi(g)s \,(hT):=g\cdot s(g^{-1}hT).
\end{align}

\begin{remark}\label{remark:BW}
    A section $s\in H^0(G/T,\mathbf{J}^{\gamma_\K})$ equivariantly lifts to a function $\overline{s}\in C^{\infty}(G,\mathbb{C})$ verifying $\overline{s}(g\cdot t)=\gamma_{\K}(t)^{-1}\overline{s}(g)$. Indeed, just define $\overline{s}$ by $s(hT)=[h,\overline{s}(h)]$. At the level of lifted sections, the previous representation reads $\overline{\Phi}(g)\overline{s}\,(g_0):=\overline{s}(g^{-1}g_0)$.
\end{remark}
For a representation  $V$, write $V^*$ for the corresponding dual representation.
\begin{theorem}
    Let $\lambda(\gamma)$ be an analytically integral weight whose corresponding character is $\gamma$ and define $\mathbf{J}^{-\lambda(\gamma)}$ to be the corresponding associated bundle (\ref{eq:associated_line_bw}) using $\gamma^{-1}$. If $-\lambda(\gamma)\notin \mathcal{W}^+$, then  $ H^0(G/T,\mathbf{J}^{-\lambda(\gamma)})=\{0\}$. If not, let $V^{-\lambda(\gamma)}$ be the unique complex irreducible representation of $G$ with highest weight $-\lambda(\gamma)\in \mathcal{W}^+$. There is the following isomorphism of representations 
    \begin{align}
        H^0(G/T,\mathbf{J}^{-\lambda(\gamma)})\cong (V^{-\lambda(\gamma)})^*,
    \end{align} 
\end{theorem}
\begin{proof}
    For more information and the proof of the theorem, we refer to \cite[Lemma 7.5]{Sepanski2007}. .
\end{proof}
As an example, the reader may bear in mind the $SO(3)$ case. Real irreducible representations are indexed by odd integers $2l+1$ ($l \geq 0$) and correspond to the usual representation on spherical harmonics $\mathbf{H}_{2l+1}:=\{P\in \mathbb{R}_l[x_1,x_2,x_3],\, P(\lambda x)=\lambda^lP(x) \,\,,  \Delta P=0\}$. The corresponding complexified representation is still irreducible. Doing this, $\mathbf{H}_{2l+1}\otimes \mathbb{C}$ actually identifies with $H^0(\mathbb{CP}^1,\mathcal{O}(2l))$ and thus provides a concrete manifestation of the previous theorem. Here $\mathcal{O}(2l)$ is $\mathcal{O}(1)^{\otimes {2l}}$ where $\mathcal{O}(1)$ is the dual to the tautological bundle on $\mathbb{CP}^1$.

The chosen system of roots spanning irreducible representation $(\lambda_i)_{1\leq i \leq d}\in (i\mathfrak{t}^*)^d$ allows one to construct individual line bundles $(J_1,\dotsc,J_d)$. We will write $\mathbf{J}^{\K}=\bigotimes _{1\leq i \leq d}J^{\otimes k_i}_i$ with $\K:=(k_i)_{1\leq i \leq d}$. The metric  $g^{\K}$ will stand for the induced metric on $\mathbf{J}^{{\K}}$ by the associated bundle construction. The bi-invariant metric on $G$ induces a metric on $G/T$, hence a Riemannian volume written $d\mu_T$. This allows us to define a scalar product on $L^2(G/T,\mathbf{J}^{\gamma_{\K}})$.

Finally, the data of the metric $g^{\K}$ on $\mathbf{J}^{\K}$ uniquely defines the associated Chern connection $\nabla_{\mathbf{J}^{\K}}$, whose $(0,1)$-part will be written $\overline{\partial}_{\mathbf{J}^{\K}}$. We thus have by definition  
\begin{align}\label{align:hol}
    H^0(G/T,\mathbf{J}^{\K}):=\ker\bigl(\,\overline{\partial}_{\mathbf{J}^{\K}}\bigr).
\end{align}
\subsubsection{Fourier decomposition on principal bundles :}\label{section:fiberwise} We now fix $P\to M$ to be a $G$-principal bundle with $G$ as above. Standard fiberwise Fourier decomposition gives the isomorphism 
\begin{align}\label{Fourier}
    \mathcal{F}\, : \,C^{\infty}(P)\to\bigoplus_{\lambda\in \hat{G}}C^{\infty}(M,\text{Hom}(V^{\lambda},E^{\lambda})).
\end{align}
where $V^{\lambda}$ denotes the unique (up to isomorphism) irreducible representation with highest weight $\lambda$ and $E^{\lambda}:= P\times_{\rho_{\lambda}} V^{\lambda}$ is the corresponding associated vector bundle. Using Theorem \ref{remark:BW} we can retrieve from a smooth function $f\in C^{\infty}(P)$ a family of sections $(f_{\K})$ of particular line bundles. The latter family is easier to manipulate than a family of sections of vector bundles whose rank may even go to infinity.
\begin{definition}\label{def:fiberwise}
    Let $T$ be a maximal torus of $G$. We can form the fiber bundles $F:=P/T\to M$ whose typical fiber is the flag manifold $G/T$. Each fiber possesses a holomorphic structure. 
    Now let $\mathbf{k}\in \Lambda$ and $\mathbf{L}:=(L_1,\cdots,L_d)$ where $L_i:=P\times_{\gamma(\lambda_i)}\mathbb{C}$, we can thus form $\mathbf{L}^{\otimes \mathbf{k}} := \bigotimes_{i=1}^{d}L_i^{\otimes k_i}$ which are \textit{fiberwise holomorphic line bundles over $F$}.
    
    Also define 
    \begin{align*}
        H^0(F,\mathbf{L}^{\mathbf{k}}):= \coprod_{x\in M} H^0(F_x,\mathbf{L}^{\mathbf{k}}_{|F_x}),
    \end{align*}
    which is a smooth vector bundle over M.
\end{definition}
From time to time, we will freely use the identification
\begin{align}\label{eq:identification}
    C^{\infty}(M,H^{0}(F,\mathbf{L}^{\mathbf{k}}))\cong C^{\infty}_{\mathrm{hol}}(F,\mathbf{L}^{\mathbf{k}})
\end{align}
where the subscript $\mathrm{hol}$ stands for fiberwise holomorphic sections of $\El^{\K}\to F$ (we can easily construct a fiberwise version of $\overline{\partial}_{J^\mathbf{k}}$, written $\overline{\partial_{\mathbf{k}}}$, thanks to \cite[\S 2.2.3]{CL24}). More generally, for a functional space $\mathcal{F}$ of sections of $\El^\K\to F$ (such as $\mathcal{D}'(F,\El^\K)$ the set of distributional sections), we will write $\mathcal{F}_{\mathrm{hol}}:=\mathcal{F}\cap \ker(\overline{\partial})$.

For notational simplicity, $\El^\K$ should be understood as $\mathbf{L}^{\otimes \K}$.
As in \cite[\S 2.2.3]{CL24}, we write $\Pi_{\mathbf{k}}:L^2(F,\El^{\K})\to L^2_{\mathrm{hol}}(F,\El^{\K})$ the orthogonal projection on fiberwise holomorphic sections. 

From $f\in C^{\infty}(P)$ one obtains the family of sections $(\mathcal{F}f(\cdot,\lambda))_{\lambda\in \hat{G}}$. Each of the previous sections has a value in $E^{\lambda}=P\times_{\lambda} V^{\lambda}$. We would like to get further a family $(f_\mathbf{k})_{\mathbf{k}\in \Lambda}$ of sections in $C^{\infty}_{\mathrm{hol}}(F,\mathbf{L}^{\mathbf{k}})$. For that, it suffices to prove the following useful Proposition.
\begin{proposition}\label{prop: fiber_id}
    Recall that $\Phi$ is the previously-defined canonical representation of $G$ in $H^0(G/T,J^{\mathbf{k}})$. We find that $H^0(F,\mathbf{L}^{\mathbf{k}})$ is an associated vector bundle with respect to $\Phi$, that is 
    \begin{align}\label{eq: fiber_id}
        H^0(F,\mathbf{L}^{\mathbf{k}})\cong P\times_{\Phi} H^0(G/T,J^{\mathbf{k}})
    \end{align}
\end{proposition}
\begin{proof}
  This seems to be a standard result, but since we did not manage to find a reference, we quickly prove the statement for completeness. Recall that $\pi:P\to M$ is the bundle projection. To each $p\in P$ we can associate $\psi_p:G\to P_x$ the canonical isomorphism $g\to p\cdot g$. For every $x\in M$, we further write $\Psi_T$ the (invertible) map assigning to an element $s\in (H^0(F,\El^\K))_x$ its $T$-equivariant lift to $P_x$.
  
  We can form the following map 
  \[
\begin{aligned}
 P\times H^0(G/T,\mathbf{J}^\K) &\longrightarrow H^0(F,\El^\K) \\
\Psi \colon \quad\quad\quad (p,s)\quad\quad &\longmapsto \Psi_T^{-1}\bigl((\psi_p)_*\overline{s}\bigr)
\end{aligned}
\]
where $\overline{s}\in C^{\infty}(G,\mathbb{C})$ is the $T$-equivariant lift of $s$, which descends to a holomorphic section of $\mathbf{J}^{\K}\to G/T$. In order for $\Psi$ to make sense, let us check that for any $(p,s)\in P\times H^0(G/T,\mathbf{J}^\K)$, $(\psi_p)_*\overline{s}$ is $T$-equivariant. For any $p':=p\cdot g'\in P_x$ 
\begin{align*}
    \overline{s}(\psi_p^{-1}(p'\cdot t))&=\overline{s}(g'\cdot t)\\
    &=\gamma_\K(t^{-1})\overline{s}(\psi_p^{-1}(p')),
\end{align*}
thus descending to a section $\mathbf{L}^\K\to F_{\pi(p)}$ thanks to $\Psi_T^{-1}$.
Finally $\Psi$ descend to the associated bundle; indeed for any $g,g'\in G$ ($p':=p\cdot g'$)
\begin{align*}
    \Psi(p\cdot g, \Phi(g^{-1})s)&=\Psi_T^{-1}\bigl(\overline{\Phi(g^{-1})s}(\psi_{p\cdot g}^{-1}(p\cdot g'))\bigr)\\
    &=\Psi_T^{-1}\bigl(\overline{s}(g\,\psi_{p\cdot g}^{-1}((p\cdot g)\cdot g^{-1}g'))\bigr)\\
    &=\Psi_T^{-1}\bigl(\overline{s}(\psi_p^{-1}(p'))\bigr)\\
    &=\Psi(p,s),
\end{align*}
where we used Remark \ref{remark:BW} in the second equality and the relation $\psi_p^{-1}(p')=g'$. We thus get the desired smooth isomorphism (holomorphicity is easily checked).

\end{proof}
This isomorphism induces an isomorphism of vector bundles $E^{\lambda}\cong H^0(F,\mathbf{L}^{\mathbf{k}})$ where $\mathbf{k}=(k_i)_{1\leq i \leq d}$ is such that $\lambda =\sum k_i \lambda_i$. To account for multiplicity of an irreducible representation in the Fourier decoposition of $f\in C^{\infty}(P)$, we write $d_\K:=\dim(H^0(G/T,\mathbf{J}^{\K}))$.

Due to the fiberwise Fourier transform and Borel-Weil theory, we were able to pass data from a smooth function $f\in C^{\infty}(P)$ to a family $(f_{\mathbf{k},i})_{\mathbf{k}\in \Lambda,1\leq i\leq d_\K}$ of fiberwise holomorphic sections of the line bundle $\mathbf{L}^{\mathbf{k}}$. For simplicity the multiplicity index in the Fourier decomposition will be omitted most of the time since it will not impact the arguments of the next sections.

To finish, suppose that we equip $P\to M$ with a connection $\mathbb{H}_P$. This connection induces a canonical connection $\mathbb{H}_F$ on the homogeneous $G$-bundle $F\to M$. The complexified Lie algebra $\mathfrak{g}_\C$ of $G$ splits according to 
\begin{align}\label{eq: splitting}
    \mathfrak{g}_\C=\mathfrak{t}_\C\oplus \mathfrak{n}^+\oplus \mathfrak{n}^-, 
\end{align}
with $\mathfrak{n}^\pm$ being the positive and negative root spaces for the adjoint representation of $\mathfrak{g}$. Thus, the $T$-principal bundle $P\to F$ can be equipped with the connection $\mathbb{H}_{P\to F}:=\mathbb{H}_P\oplus \Re(\mathfrak{n}^+\oplus \mathfrak{n}^-)$. This gives us the possibility of defining a \textit{horizontal connection} and a full vector connection on $\El^{\K}$ as explained in \cite[\S 2.2.2]{CL24}. For that purpose we will write for a vector field $X$ on $M$ (resp. $F$) its horizontal lift $X^{\mathbb{H}_F}$ to $F$ and $X^{\mathbb{H}_P}$ to $P\to M$ (resp. $X^{\mathbb{H}_{P\to F}}$ its lift to $P\to F$).

For the horizontal connection $\nabla^{\K}: C^{\infty}(F,\El^{\K})\to C^{\infty}(F,\El^{\K}\otimes \mathbb{H}_F^*)$, we define from $s\in C^{\infty}_{\mathrm{hol}}(F,\El^{\K})$ and $X\in C^{\infty}(M,TM)$ 
\begin{align}\label{eq:partialconnection}
    \overline{\nabla^\K_{X^{\mathbb{H}_F}}s}:=X^{\mathbb{H}_P}\overline{s},
\end{align}
which by construction descends to a section written $\nabla^{\K}_{X^{\mathbb{H}_F}}s$.

For the full connection $\overline{\nabla}^{\mathbf{k}}:C^{\infty}(F,\El^{\K})\to C^{\infty}(F,\El^{\K}\otimes T^*F)$, we define from $s\in C^{\infty}_{\mathrm{hol}}(F,\El^{\K})$ and $X\in \Gamma(TF)$ 
\begin{align}\label{eq:fullconnection}
    \overline{\overline{\nabla}_{X}^\K s}:=X^{\mathbb{H}_{P\to F}}\overline{s},
\end{align}
the induced section over $F$ will be written $\overline{\nabla}^{\K}_Xs$.

We record the following crucial property, which will be used extensively in the remaining parts. We refer to \cite[Proposition 2.2.8 \& Lemma 2.2.9]{CL24} for the proof.
\begin{proposition}\label{prop:holpreserved}
    The complexified horizontal space $\mathbb{H}_F\otimes \mathbb{C}:=\mathbb{H}_F(\C)$ and its dual can be equipped with a canonical fiberwise complex structure. We have 
    \begin{align*}
        \nabla^{\K}:C^{\infty}_{\mathrm{hol}}(F,\mathbf{L}^{\K})\to C^{\infty}_{\mathrm{hol}}(F,\mathbf{L}^{\K}\otimes \mathbb{H}_F^*)
    \end{align*}
    Moreover, $[\nabla^{\K},\Pi_{{\K}}]:=\Pi_\K^{\mathbb{H}_F(\mathbb{C})^*}\nabla_\K-\nabla_\K\Pi_\K=0$ where $\Pi^{\mathbb{H}_F(\C)^*}_{\K}:L^2(F,\El^\K\otimes\mathbb{H}_F(\C)^*)\to L^2_{\mathrm{hol}}(F,\El^\K\otimes\mathbb{H}_F(\C)^*)$ is the orthogonal projection on fiberwise holomorphic sections.
\end{proposition}

\subsection{Uniform calculus on $\El^{\K}\to F$ }\label{subsection:uniformcalculus} For an in-depth treatment of the following concepts, we again refer to \cite[\S 3]{CL24}. We wish to provide the reader with the tools that we will use in order to prove our main result. Only in this sub-section will $F$ stand for any closed manifold: we don't need the fiberwise holomorphic structure to define the uniform calculus on $\El^\K\to F$. We fix $L_1,..,L_d$ to be complex Hermitian line bundles on $F$ and $\nabla=(\nabla_1,\dotsc,\nabla_d)$ a family of unitary connections on $\mathbf{L}:=(L_1,..,{L}_d)$. In this section only, $\Lambda$ is any subset of $\Z^d$.

Let $h\in (0,1]$ and introduce $\Omega(h):=\bigl\{\mathbf{k}\in \Lambda, \,h\langle \mathbf{k}\rangle\leq 1\bigr\}$. Notice here that we can construct sequences of elements in $\Omega(h)$ such that $(\mathbf{k}(h))_h$ has a norm that goes to infinity as $h$ goes to $0$.  We wish to construct semiclassical pseudodifferential operators acting on sections of $\mathbf{L}^{\K}\to F$, and we additionally want to take into account the parameter $\mathbf{k}$. 

First, to take into account the parameter $\mathbf{k}$ we define the \textit{uniform class of semiclassical pseudodifferential operators on $F$}, $\Psi^{\bullet}_{h,\mathbf{k}}(F)$.

\begin{definition}\label{symbol}
    Given a covering $(U_i,\varphi_i)$ of $F$ by charts, we say that a smooth function $a_{h,\K}$ on $T^*M$ is in the uniform class $S^m_{h,\mathbf{k}}(T^*F)$ if, for any $i$, $h\in (0,1]$ and $\mathbf{k}\in \Omega(h)$ 
    \begin{align*}
        \bigl|\partial_{\xi}^{\beta}\partial_x^{\alpha}(a_{h,\mathbf{k}}\circ d\varphi_i^\top\,)(x,\xi)\bigr|\leq C_{\alpha,\beta,i} \langle \xi \rangle ^{m-|\beta|}, \, (x,\xi)\in T^*U_i,
    \end{align*}
    Notice that the bound is uniform both in $h$ and $\K$. The corresponding class of (semiclassical) pseudodifferential operators is
    \begin{align*}
    \Psi^m_{h,\K}(F):=\bigl\{\Op_h(a)+R,\, R\in h^{\infty}\Psi_{h,\K}^{-\infty}(F)\bigr\},
\end{align*}
where $\Op_h$ is a quantization procedure (see \cite[Appendix E]{DZ19}), elements in $\Psi_{h,\K}^{-\infty}(F)$ are operators with smooth Schwartz kernel having each of its $C^{\infty}$ seminorms uniformly bounded in $h$ and $\mathbf{k}\in \Omega(h)$.
\end{definition}
The reader should bear in mind the following important remark. It shows the important difference with the usual semiclassical calculus.
\begin{remark}\label{parameterdependant}
    The connection $h{\nabla}^{\K}$ takes the form $ihd+ih\K\cdot \beta$ in local coordinates with respect to a local unit section $s^{k_1}\otimes\cdots\otimes s^{k_d}$ ($\beta:=(\nabla_i s_i/s_i)_{1\leq i \leq d}$). If $\mathbf{k}\in \Omega(h)$ is fixed, then the principal symbol would be $i\xi\in C^{\infty}(T^*F,\mathbb{C}\to T^*F)$. But choosing for instance $\K(h)$ so that $\lim_{h\to 0}h\K(h)\to \mathbf{l}\in (i\mathfrak{t})^*$ gives the principal symbol $i\xi +i\mathbf{l}\cdot\beta$.
    To take all of this into account, we see that the principal should only be defined \textit{with respect to a mapping $h\mapsto \mathbf{k}(h)$}.
\end{remark}
This motivates the following definition 
\begin{definition}\label{def:shiftsymbol}
    The principal symbol $\sigma_{h,\mathbf{k}(h)}(A)$ of $A\in \Psi_{h,\K}(F)$ with respect to the function $h\mapsto \K(h)$ is defined as the principal symbol of the operator $A_{h,\K(h)}\in\Psi_h(F)$.
    In turn, $\sigma_A$ can be seen as a function taking values in functions $h\mapsto \mathbf{k}(h)$ and returning $\sigma_{A_{h,k(h)}}$.
\end{definition}
In semiclassical analysis it is usual to work with the \textit{fiber-radially compactified cotangent bundle} $\overline{T^*F}$, see \cite[\S E.1.3]{DZ19} for a precise definition. Roughly speaking, its interior is diffeomorphic to $T^*F$ while its boundary is diffeomorphic to the cosphere bundle $S^*F$. We can define for $A\in \Psi_{h,\K}(F)$ its elliptic set $\Ell(A)$ and its wavefront set $\text{WF}(A)$ as 
\begin{align}
    \Ell_h(A) &:= \bigcap_{h\mapsto \K(h)} \Ell(A_{h,\K(h)}) \subseteq \overline{T^*F} \label{eq:elliptic_set}, \\
    \text{WF}_h(A) &:= \bigcup_{h\mapsto \K(h)} \text{WF}(A_{h,\K(h)}) \subseteq \overline{T^*F}.\label{eq:wavefront_set}
\end{align}
The usual results on the inversion of microelliptic operators carry on in this calculus.\\
Building on this calculus, we now define the class $\Psi^{\bullet}_h(F,\El)$ of \cite[\S 3.2]{CL24}. 
\begin{definition}\label{def:unifclass}
    An element $\mathbf{A}\in \Psi_{h}(F,\mathbf{L})$ is the data of a family of operators $A_{h,\mathbf{k}}\in \Psi_{h,\mathbf{k}}(F, \mathbf{L}^\mathbf{k})$. In turn, we say that $A_{h,\mathbf{k}}\in \Psi_{h,\mathbf{k}}(F, \mathbf{L}^\mathbf{k})$ if, for every trivializing $U\subset F$, for all $s_i$ unit sections of $(L_i)_{\big| U}$, $\chi,\psi\in C^{\infty}_c(U)$, there exists $A_{h,\mathbf{k}}\in \Psi_{h,\mathbf{k}}(F)$ such that for all $f\in C^{\infty}(F)$ 
    \begin{align}\label{eq:localunifclass}
        \psi \mathbf{A}_{h,\mathbf{k}}(\chi f\mathbf{s}^{\mathbf{k}})=A_{h,\mathbf{k}}(f)\otimes \mathbf{s}^{\mathbf{k}}.
    \end{align}
\end{definition}
The principal symbol of such operators is defined with respect to a sequence $h\mapsto\mathbf{k}(h)$ (see Remark \ref{parameterdependant}) and the family of connections $({\nabla}_1,\dotsc,{\nabla}_d):={\nabla}$. The latter being motivated by the fact that it is natural to expect $\nabla_X^{\mathbf{k}(h)}$ to have $(x,\xi)\mapsto i\xi(X(x))$ as its principal symbol: the connection $1$-forms from Remark (\ref{parameterdependant}) should not appear. 
\begin{definition}\label{def:symbunifclass}
    In the local setting of the above definition, define the principal symbol of $\mathbf{A}\in \Psi_{h,\mathbf{k}}(F,\mathbf{L})$ with respect to $h\mapsto \K(h)$ and the family of connections $\tilde{\nabla}$ to be 
    \begin{align}\label{eq:symbunifclass}
        \sigma^{(\tilde{\nabla})}_{\mathbf{k}(h)}(\mathbf{A})(x,\xi):=\sigma_{h,\K(h)}(A)(x,\xi-h\mathbf{k}(h)\cdot \tilde{\beta}).
    \end{align}
    where $\tilde{\beta}_i:=\tilde{\nabla}_is_i/s_i$. That the definition (\ref{eq:symbunifclass}) is independent of the chosen charts or local unitary sections is proved in \cite[\S 3.2.3]{CL24}.
\end{definition}
Even if the theory works for any unitary connection on $\mathbf{L}$, we will work in a convenient setting for us.

\textbf{Convention.} For $F$ and $\El^\K$ as in Definition \ref{def:fiberwise}, the principal symbols (and other connection dependent quantities) will always be computed respect to the vector of connections $\overline{\nabla}^{\K}$ defined in \ref{eq:fullconnection}. For instance, the principal symbol $h\overline{\nabla}^{\K}_X$ is $i\xi(X)$, and is, by construction, independent of a choice $h\mapsto \mathbf{k}(h)$.\\

We will write $T_{\beta,\K}$ for the phase space translation $(x,\xi)\mapsto (x,\xi-h\mathbf{k}(h)\cdot \beta)$ on $F$. Notice that (for $d=1$ to simplify) if we replace a unit trivializing section $s$ by $e^{i\theta(\cdot)}s$, another trivializing section, then with respect to the latter $A_{h,\mathbf{k}}$ becomes $A'_{h,\mathbf{k}}(\bullet)=e^{-i\theta} A_{h,\mathbf{k}}(e^{i\theta}\bullet)$. The new connection form is $\beta'=\beta+d\theta$. Thanks to the stationary phase formula it can be checked that  $T_{\beta}^{-1}\WF_{\mathbf{k}(h)}(A_{h,\mathbf{k}(h)})=T_{\beta'}^{-1}\WF_{\mathbf{k}(h)}(A'_{h,\mathbf{k}(h)})$ (see \cite[Lemma 3.1.9]{CL24} for a detailed proof). Thus, the following definition of the wavefront set with respect to a family of connection $\nabla$ is independant of any choice of trivializing sections over an open set $U$
\begin{align}\label{eq:wfconnection}
    \text{WF}^{{\nabla}}(\mathbf{A}):=\bigcup_{h\mapsto \mathbf{k}(h)}\biggl(\bigcup_{U,\chi,\psi} T_{\beta}^{-1}\text{WF}(A_{h,\mathbf{k}(h)})\biggr).
\end{align}
where we recall that $A_{h,\K}$ is defined in \ref{eq:localunifclass}. The elliptic set $\Ell^{{\nabla}}(\mathbf{A})$ (being defined from the principal symbol, the elliptic set depends on the chosen connection) is defined in a similar fashion as in the uniform calculus, see \cite[Definition 3.2.5]{CL24}.
\begin{remark}\label{remark:quantif}
    Adapting \cite[Appendix E]{DZ19}, it is clear that we may construct a quantization $\Op_{h,\K}:S^{\bullet}(T^*F)\to \Psi^{\bullet}(F,\mathbf{L}^\K)$ by fixing a covering of cutoff charts $\{(U_i,\varphi_i,\chi_i)\}$ and local trivializations $\mathbf{s}_j$ over $U_j$. However in view of the principal symbol property, the quantization formula of $a_{h,\mathbf{k}}$ would need to use the quantization of $\tilde{\varphi_i}^*(a_{h,\mathbf{k}}\circ T_{\beta,\K})$ in $\varphi_i(U_i)$ ($\tilde{\varphi_i}$ being the symplectic lift of the cart map $\varphi_i$). This is in order to recover the principal part of $a_{h,\K}$ as a principal symbol of $\Op_{h_,\K}(a_{h,\K})$.
\end{remark}
Finally, we may define semiclassical Sobolev spaces for distributional sections of $\El^\K$ in the usual fashion.
\begin{definition}\label{def:sobolev_full}
    Define ${\Delta}_{\K}:=({\nabla}^\K)^*{\nabla}^\K$ and set 
    \begin{align}
        H^s_h(F,\mathbf{L}^\K):=(\mathbf{1}+h^2\overline{\Delta}_\K)^{-s/2}L^2(F,\mathbf{L}^\K),
    \end{align}
    where $(\mathbf{1}+h^2\overline{\Delta}_\K)^{-s/2}$ is defined using the spectral theorem. The latter is a Hilbert space when equipped with its natural scalar product.
\end{definition}
\subsubsection{Propagation properties in the uniform calculus }\label{section:propagation} An important consequence of the definition of the principal symbol in the semiclassical uniform calculus is the fact that the subsequent "classical dynamics" of an operator $\Op_{h,\mathbf{k}}(a_{h,\K})$ is not simply the Hamiltonian flow of $a$ with respect to the Liouville $2$-form $\omega_0$ on $T^*F$.

Write $\mathbf{F}_{\overline{\nabla}}:=(F_{\overline{\nabla}_1},\dotsc,F_{\overline{\nabla}_d})\in C^{\infty}(F,\Lambda^2T^*F)^d$ the vector of curvature forms of each $L_i$ (each one of them is purely imaginary). We define the following $2$-form on $T^*F$ 
\begin{align}\label{eq:twisted_symplectic}
    \omega_{h,\mathbf{k}}:=\omega_0+ ih\mathbf{k}\cdot \pi_{T^*F}^*\mathbf{F}_{\overline{\nabla}}.
\end{align}
Thanks to the Bianchi identity, this form is closed. A computation in local coordinates gives non-degeneracy of $\omega_{h,\K}$ \cite[Lemma 3.2.10]{CL24}, making $(T^*F,\omega_{h,\mathbf{k}})$ a symplectic manifold. For a smooth function $f$ on $T^*F$, write $H_f^{\omega_{h,\mathbf{k}}}$ the corresponding Hamiltonian vector field and $(\Phi^{\omega_{h,\K},f}_t)$ its Hamiltonian flow.  Write $p_X,p_Y$ the usual Hamiltonian functions of $X,Y\in \Gamma(TF)$ : $p_X(x,\xi)=\xi(X(x))$. We have the following useful identity for $\omega_{h,\K}$ (we refer the reader to \cite[3.2.18]{CL24}) 
\begin{align}\label{eq:vfield_symplectic}
    \{p_X,p_Y\}^{\omega_{h,\mathbf{k}}}=p_{[X,Y]}+ih\K\cdot \mathbf{F}_{\overline{\nabla}}(X,Y),
\end{align}
where $\{\cdot,\cdot \}^{\omega_{h,\mathbf{k}}}$ is the Poisson bracket associated to $\omega_{h,\K}$.
Also, the following property will be important for us. The generator of the frame flow descended to $F$, $X_F$, satisfies $i_{X_F}\mathbf{F}_{\overline{\nabla}}=0$. As justified by the following proposition (see \cite[Lemma 3.2]{DGH2021} for instance), it implies that $H_{p_{X_F}}^{\omega_{h,\mathbf{k}}}=H_{p_{X_F}}^{\omega_0}$ ; one will recover the usual sink and source type structure of the Hamiltonian flow of $X_F$.
\begin{lemma}\label{label:hamiltonianflow}
    If $Y\in \Gamma(TF)$ is such that $i_Y\mathbf{F}_{\overline{\nabla}}=0$, then  $H_{p_{Y}}^{\omega_{h,\mathbf{k}}}=H_{p_{Y}}^{\omega_0}$.
\end{lemma}
\begin{proof}
    By the definition of $\omega_{h,\K}$, $\iota_{H^{\omega_0}_{p_Y}}\omega_{h,\mathbf{k}}= \iota_{H^{\omega_0}_{p_Y}}\omega_0+ ih\mathbf{k}\cdot(i_Y\mathbf{F}_{\overline{\nabla}})$. The right-hand side is equal to $dp_{Y}$ by the assumption on $Y$. The non-degeneracy of $\omega_{h,\K}$ concludes.
\end{proof}
To finish, we will use extensively the following theorem, proved in \cite[\S 3.2.4]{CL24}. 
\begin{proposition}\label{prop:unifpropagation}
    Let $Y$ be a vector field on $F$ and set $\mathbf{Y}=i_Y\overline{\nabla}^{\mathbf{k}}$ the induced family of operators in $\Psi_{h}(M,\mathbf{L})$. As we saw above, $\mathbf{Y}(:=(\mathbf{Y}_\mathbf{k})_{\mathbf{k}})\in \Psi_h(M,\mathbf{L})$. Then \\
    \begin{enumerate}
        \item For all $\mathbf{A}\in \Psi_{h,\mathbf{k}}$, we have that $e^{t\mathbf{Y}}\mathbf{A}e^{-t\mathbf{Y}}\in \Psi_{h}(M,\mathbf{L})$.
        \item The principal symbol of the previous operator with respect to a sequence $h\mapsto \mathbf{k}(h)$ is $\sigma_{h,k(h)}(\mathbf{A})\circ \Phi_t^{\omega_{h,\mathbf{k}(h)},p_Y}$. Moreover,  $\sigma_{h,\mathbf{k}(h)}(h^{-1}[\mathbf{Y},\mathbf{A}])=i\,H^{\omega_{h,\mathbf{k}(h)}}_{p_Y}\sigma_{h,\K(h)}(\mathbf{A})$.
    \end{enumerate}
\end{proposition}

\subsection{The Borel-Weil calculus.}\label{subsection:BW} $F$ and $\El^\K$ are now as in Definition \ref{def:fiberwise}. We were able to pass from the data of a smooth function on $P$ to the data of a family of Fourier modes, namely sections in $C^{\infty}_{\mathrm{hol}}(F,\mathbf{L}^{\mathbf{k}})$. More precisely, we associated to the data of $f\in C^{\infty}(P)$ the family of sections $\mathcal{F}f:=(f_{\K,i})_{\K\in \Lambda, i\leq d_{\K}}$, $d_{\K}$ being the dimension of the irreducible representation corresponding to the label $\K\in \Lambda$. From this, the work done by Cekić-Lefeuvre allows to build a calculus such that the class of $G$-equivariant pseudodifferential operators on $P$ would fit in \cite[\S 3.3]{CL24}. Such operators $A$ satisfy by definition 
\begin{align*}
        A(R_g^{*}f)=R_g^{*}A(f), \text{ $R_g$ right-multiplication by $g$}.
\end{align*}
Those operators induce, by the usual Fourier decomposition and the Borel-Weil theorem, a family of operators $\mathbf{A}_{\mathbf{k}}:C^{\infty}_{\mathrm{hol}}(F,\mathbf{L}^{\mathbf{k}})\to C^{\infty}_{\mathrm{hol}}(F,\mathbf{L}^{\mathbf{k}})$. For instance, a vector field $X\in \Gamma(\mathbb{H}_P)$ satisfying $(R_g)_*X=X$ enjoys such a property (\textit{i.e.} horizontal lifts of vector fields on $M$). It induces a family of operators $(\X_\K)$, which is explicitly seen to be $\nabla^\K_{X^{\mathbb{H}_F}}=\overline{\nabla}^\K_{X^{\mathbb{H}_F}}$ (notice the absence of horizontal line over the connection symbol), this will be our main object of interest in the following sections.
\begin{definition}\label{def:bwclass}
    For $m\in \mathbb{R}$, define 
    \begin{align*}
        \Psi_{h,\mathrm{BW}}^m(P):=\bigl\{\Pi_{\mathbf{k}} \mathbf{A}\Pi_\mathbf{k} \text { s.t. } \mathbf{A}\in\Psi_{h}(F,\mathbf{L}) \text{ admissible}\bigr\},
    \end{align*}
    where \textit{admissibility} (following terminology of \cite[\S 3.3]{CL24}) of $\mathbf{A}\in \Psi(F,\El)$ means \begin{align}\label{eq:hol_preserved}
        [\mathbf{A}_{\bullet},\Pi_{\bullet}]\in h^{\infty}\Psi^{-\infty}_h(F,\mathbf{L}^{\bullet}).
    \end{align}
    In particular, such operators must preserve holomorphicity of sections up to negligible remainder.
    
\end{definition}
There are two important subbundles of $T^*F$. The annihilator of $\mathbb{V}_{F}$ (resp. $\mathbb{H}_{F}$) is written $\mathbb{H}_{F}^*$ (resp. $\mathbb{V}_F^*$) and is called the co-horizontal (resp. co-vertical) space. The main technical feature of working with fiberwise holomorphic section is that only the co-vertical direction $\overline{\mathbb{H}_F^*}\subset \overline{T^*F}$ will be under consideration. Indeed, for $\mathbf{A}:=\Pi_{\bullet}\mathbf{A}_{\bullet}\Pi_{\bullet}\in \Psi_{h,\mathrm{BW}}(P)$, we have by admissibility
\begin{align*}
    h\overline{\partial}_{\bullet}\mathbf{A}_{\bullet}\in\mathcal{O}(h^\infty)\Psi^{-\infty}
\end{align*}
but $h\overline{\partial}_{\mathbf{k}}\in \Psi^1_{h,\mathbf{k}}(F,\mathbf{L}^\mathbf{k})$ verifies that $\Ell^{\overline{\nabla}}(h\overline{\partial}_{\mathbf{k}})=\overline{T^*F}\setminus \overline{\mathbb{H}_F^*}$, thus $\text{WF}^{\overline{\nabla}}(\mathbf{A}_\K)\subseteq \overline{\mathbb{H}^*_F}$ for all $\K$ by elliptic inversion in the uniform calculus.
\begin{remark}\label{remark:ellipticitybw}
    Following the construction of anisotropic spaces for Anosov flows in \cite[Chapter 9]{CL24}, it is tempting to try to carry this construction using the generator of the frame flow $X_{FM}$ (we recall that the latter is not Anosov because of the additional neutral directions in (\ref{eq:introframeflowsplitting})). The symplectic lift of the flow of $X_{FM}$ has the same dynamical properties as in \cite[\S 9.1.2]{CL24} for the Anosov case, replacing $E_0^*$ by $E_0^*\oplus \mathbb{V}_{FM}^*$. The construction of the escape function $G_m$ carries on as in \cite[Lemma 9.1.9]{CL24}. However the proof of the meromorphic continuation property of the resolvent of an Anosov vector field uses crucially the ellipticity of the $\xi(X)$ in the complementary directions of $E_u^*\oplus E_0^*$, a property which is not available for $X_{FM}$: $\xi(X_{FM}(x))=0$ for all $(x,\xi)\in \mathbb{V}_{FM}^*$. The main benefit of working with fiberwise holomorphic sections is to forget the co-vertical direction $\mathbb{V}_{FM}^*$ and only consider $\mathbb{H}_F^*$.
\end{remark}
Before moving on to the next section, we briefly overview some important definition and properties of the calculus in the class $\Psi_{h,\mathrm{BW}}$, we again refer to \cite[\S 3.3]{CL24} for details and proofs. 

First, in the spirit of the exponential map quantization, one can quantize symbols in $S^m_{h,\K}(\mathbb{H}_F^*)$ in the following way, see for more details \cite[\S 3.3]{CL24}. For two points $x_1,x_2\in M$ close enough, write $\gamma$ the unique geodesic joining them and $\tau^F_{\gamma}$ the corresponding parallel transport $F_{x_1}\to F_{x_2}$ with respect to $\mathbb{H}_F$. We may now define $\tau_{\gamma}^\K$ to be the parallel transport map on $\El^\K_{|F_{x_1}}$ with respect to $\overline{\nabla}^\K$ along $\tau_{\gamma}^F$.
\begin{definition}\label{def:bwquantiz}
    Let $f\in C^{\infty}(F,\El^\K)$ and $a\in S^m_{h,\K}(\mathbb{H}_F^*)$. Let $\Pi_{\K}(\bullet)$ be the fiberwise holomorphic projection restricted to $F_\bullet$.  Define for any $hT\in F$
    \begin{equation}
\begin{aligned}\label{eq:bwquant}
\Op_{h,\mathrm{BW}}(a)f(hT)
:= (2\pi h)^{-n}\Pi_\K(\pi(hT))
\int_{M\times T^*_{\pi(hT)}M}\!\!
& \chi(y,\pi(hT))
e^{\frac{i}{h}\xi(\exp_{\pi(hT)}^{-1}(y))} \\
& \cdot a(\bullet,(D_{hT}\pi)^T(\xi)) \\
& \cdot \tau_{y\to \pi(hT)}
\Pi_\K(y)f_{|F_y}
\, d\xi\, dy .
\end{aligned}
\end{equation}
where $\chi$ is a cutoff on $M\times M$ equal to $1$ in a neighborhood of the diagonal. \\

\end{definition}
\begin{remark}\label{remark:bwquant}
    In \cite{DGH2021} the case where $G=\mathbb{S}^1$ was treated (in dimension $3$, $FM\to SM$ is a $\mathbb{S}^1$-principal bundle). Thus $F:=FM/T=SM$, this is the case were the Borel-Weil calculus simply corresponds to the uniform calculus: the line bundles $\El^\K$ are over $M$. In this case the fiberwise holomorphic projections do not appear since each fiber of $F$ consists in a single element. Moreover $d=1$ in this case. For such a case, \cite{DGH2021} introduces the following quantization (placing ourselves in local coordinates and with respect to a trivializing section $s^n$)
    \begin{align*}
        {s}^{-n(h)}(\Op_{h,n(h)}(a_{h,n(h)})fs^{n(h)}):=(2\pi h)^{-5}\int e^{\frac{i}{h}\langle\xi,x-y\rangle}a(x,\xi-hn(h)\beta)\chi(x)f(y)dyd\xi,
    \end{align*}
    where $\beta:=\nabla s/s$. For $n=3$, the projections maps $\Pi_\K$ are unnecessary and it is reasonable to wonder the link with quantization (\ref{eq:bwquant}). We refer the interested reader to \cite[Lemma 3.3.20]{CL24} for a proof that both quantization coincides up to a $\mathcal{O}(h)$ remainder (in their proof, the variable $z\in G/T$ will not appear and the claim is a mere application of the Kuranishi trick).
\end{remark}
\begin{definition}\label{def:bwsymbol}
    Let $m\in \mathbb{R}$ and $\mathbf{A}:=(\Pi_\K\mathbf{A'}\Pi_\K)_\K\in \Psi^m_{h,\mathrm{BW}}$. The principal symbol of $\mathbf{A}$ with respect to a sequence $\K(h)$ and the connection $\overline{\nabla}$, written $\sigma^{\mathrm{BW},\overline{\nabla}}_{h,\K(h)}(\mathbf{A})\in S^m(\mathbb{H}_F^*)$, is defined to be the restriction of $\sigma^{\overline{\nabla}}_{h,\K(h)}(\mathbf{A}')\in S^m(T^*F)$ to $\mathbb{H}_F^*$. Similarly, $\WF^{\mathrm{BW},\overline{\nabla}}(\mathbf{A})$ (resp. $\Ell^{\mathrm{BW},\overline{\nabla}}(\mathbf{A})$) is defined by restriction of $\WF^{\overline{\nabla}}(\mathbf{A'})$ (resp. $\Ell^{\overline{\nabla}}(\mathbf{A}')$) to $\mathbb{H}_F^*$.
    The class $\Psi^{\mathrm{comp}}_{\mathrm{BW},h}$ of compactly microsupported operators will stand for elements of $\Psi_{\mathrm{BW},h}$ with compact wavefront sets ($\WF^{\mathrm{BW}}(\mathbf{A})\cap \partial{\overline{T^*F}}=\emptyset$).
\end{definition}
If $\mathbf{A}=\Pi_\K\mathbf{A'}\Pi_\K=\Pi_\K\tilde{\mathbf{A}}\Pi_\K$ we have that $\sigma^{\overline{\nabla}}_{h,\K(h)}(\mathbf{A})_{|\mathbb{H}_F^*}-\sigma^{\overline{\nabla}}_{h,\K(h)}(\tilde{\mathbf{A}})_{|\mathbb{H}_F^*}=\mathcal{O}_{S^{m-1}_h}(h)$ by \cite[Lemma 3.3.14]{CL24}, giving sense to the previous definition.

From now on, if clear from the context, we will omit to mention the BW superscript and that the previous objects are computed with respect to $\overline{\nabla}$.
Let us now gather the main properties of the calculus in $\Psi_{h,\mathrm{BW}}(P)$ thanks to \cite[\S 3.3]{CL24}.
\begin{proposition}\label{prop:bwproperties1}
    \begin{itemize}
        \item $\Psi^\bullet_{h,\mathrm{BW}}$ is an algebra ; For $\mathbf{A}\in \Psi^{m}_{h,\text{BW}}(P)$ and $\mathbf{B}\in \Psi^{m'}_{h,\text{BW}}(P)$, $\mathbf{A}\mathbf{B}\in \Psi^{m+m'}_{h,\text{BW}}(P)$.
        \item Elliptic inversion : Let $\mathbf{A}\in \Psi^{m}_{h,\text{BW}}(P)$ and $\mathbf{B}\in \Psi^{m'}_{h,\text{BW}}(P)$ be such that $\WF(\mathbf{A})\subset \Ell(\mathbf{B})$. Then there exists $\mathbf{Q},\mathbf{Q}'\in \Psi^{m-m'}_{h,\mathrm{BW}}$ such that 
        \begin{align}\label{eq:bwproperties1_1}
            \mathbf{A}=\mathbf{B}\mathbf{Q}+\mathcal{O}_{\Psi^{-\infty}_{h,\mathrm{BW}}}(h^\infty)=\mathbf{Q'}\mathbf{B}+\mathcal{O}_{\Psi^{-\infty}_{h,\mathrm{BW}}}(h^\infty).
        \end{align}
        \item Gårding's inequality : Let $\mathbf{A}\in \Psi^{m}_{h,\text{BW}}(P)$ and suppose that for any $h\mapsto \K(h)\in \Omega(h)$, $\Re(\sigma^{\mathrm{BW}}_{h,\K(h)}(\mathbf{A}))\geq 0$. Then there exists $C>0$ such that for all $u\in C^{\infty}(F,\El^\K)$, we have 
        \begin{align}\label{eq:bwproperties1_2}
            \Re(\mathbf{A}u,u)_{L^2(F,\El^\K)}\geq -Ch\|u\|_{H^{(m-1)/2}(F,\El^\K)}^2
        \end{align}
        \end{itemize}
\end{proposition}
This proposition gathered the basic properties of the calculus in $\Psi_{h,\text{BW}}$. In Section \ref{section:resolventestimatesfirstband}, we will extensively use additional properties of the calculus concerning propagation of singularities. This is the context of the following more advanced proposition. For the definitions of \textit{radial sink and sources}, we refer the reader to \cite[Appendix E]{DZ19} for the definitions or to \S\ref{subsection:anisofourierwise} for the specification to our setting. The Hamiltonian flow of an Anosov vector field presents a sink-source structure, see Section \ref{section:anisotropicspace}. We define $\Im(\mathbf{P}):=\frac{1}{2i}(\mathbf{P}-\mathbf{P}^*)$ the imaginary part of $\mathbf{P}\in \Psi_{h,\mathrm{BW}}$. For simplicity Proposition \ref{prop:bwproperties2} is specified to the case where $Y:=X^{\mathbb{H}_F}$ for $X$ an Anosov vector field on $M$ generating a volume-preserving flow and $i_Y\mathbf{F}_{\overline{\nabla}}=0$. We further assume for simplicity that the contraction and expansion rate of the flow of $X$ are equal to $1$. This is for instance achieved when $X$ generates the geodesic flow on a closed hyperbolic manifold (this will be exactly the setting we will consider in the following sections). $Y$ is called a basic vector field on $F$. The propagation estimates should still hold for general $\mathbf{P}\in \Psi_{h,\mathrm{BW}}$ whose principal symbol $\mathbf{p}$ has a Hamiltonian flow $(\Phi_t^{\omega_{h,\K},\mathbf{p}})$ preserving $\mathbb{H}_F^*$, see Proposition \ref{prop:horizontal_preserved} for sufficient conditions.

The Hamiltonian flow of $\langle \xi_{\mathbb{H}_F^*}\rangle^{-1}\mathbf{p}_Y$ on $\mathbb{H}_F^*$ extends to a flow on the compactified co-horizontal space $\overline{\mathbb{H}_F^*}$, this abstraction is convenient in order to deal with sink and sources on the co-horizontal space.

\begin{proposition}\label{prop:bwproperties2}
    Let $Y$ be a basic vector field on $F$ such that $i_Y\mathbf{F}_{\overline{\nabla}}=0$ and whose flow preserves the natural volume on $F\to M$. 
    \begin{itemize}
    \item Egorov's Theorem. Similarly to Proposition \ref{prop:unifpropagation}, for all $\mathbf{A}\in \Psi^m_{h,\mathrm{BW}}$ one has  $e^{t\mathbf{Y}}\mathbf{A}e^{-t\mathbf{Y}}\in \Psi_{h,\mathrm{BW}}$.
        The principal symbol of the previous operator with respect to a sequence $h\mapsto \mathbf{k}(h)$ is $\sigma_{h,\K(h)}(\mathbf{A})\circ \Phi^{\omega_{h,\mathbf{k}(h)},p_Y}\in S^m(\mathbb{H}_F^*)$. Moreover,  $\sigma_{\K(h)}^{\mathrm{BW}}([\mathbf{Y},\mathbf{A}])=H^{\omega_{h,\mathbf{k}(h)}}_{p_Y}\sigma_{\K(h)}^{\mathrm{BW}}(\mathbf{A})$.
        \item Propagation of singularities. 
Let $\mathbf{A}, \mathbf{B} \in \Psi^0_{h,\mathrm{BW}}(P)$ satisfy the following condition. For all $(x,\xi)\in
\WF_h^{\mathrm{BW}}(\mathbf{A})$ there exists $t\in \mathbb{R}$ such that
$
\Phi_{t}^{\omega_0,\langle \xi_{\mathbb{H}_F^*}\rangle ^{-1}p_Y}(x,\xi) \in \Ell^{\mathrm{BW}}(\mathbf{B}).
$
Then there exists $\mathbf{B} \in \Psi^{0}_{h,\mathrm{BW}}$ such that for all $N\in\mathbb{R}$, $k>0$ large enough and $f_{h,\K}\in H^N_h(F,\El^\K)$, there
exists $C>0$ satisfying for all $h>0$ and $\K\in \Omega(h)$
\begin{equation}\label{eq:propag_standard}
\|\mathbf{A} f_{h,k}\|_{H_h^N}
\leq C\Big(
\|\mathbf{B} \mathbf{Y} f_{h,\K}\|_{H_h^N}
+ \|\mathbf{A}' f_{h,\K}\|_{H_h^N}
+ h^N \|f_{h,\K}\|_{H_h^{-k}}
\Big).
\end{equation}
\item Radial sink estimate. We consider the operator $\mathbf{P}=-ih\mathbf{Y}+\eta+ih\nu$ where $\Lambda,\nu\in \mathbb{R}^+$, its principal symbol is $\mathbf{p}_Y=\xi(Y)+\eta$. We assume that $\bigl(\Phi_t^{\omega_0,\langle\xi_{\mathbb{H}_F^*}\rangle^{-1}\mathbf{p}_Y}\bigr)$ has a sink structure $L\subset \{\langle \xi_{\mathbb{H}_F^*} \rangle p_Y=0\}\cap \partial \overline{\mathbb{H}_F^*}$ in $\overline{\mathbb{H}_F^*}$, see \cite[E.51]{DZ19}. Take $N>0$ and $s\in \mathbb{R}^+$ large enough.
Then for all $\mathbf{A}\in \Psi^0_{h,\mathrm{BW}}$ whose wavefront set is near $L$, there exists $\mathbf{B}\in \Psi^0_{h,\mathrm{BW}}$ with wavefront and elliptic set in a slightly larger neighborhood than $\mathrm{WF}^{\mathrm{BW}}(\mathbf{A})$, $\mathbf{A}'\in \Psi^0_{h,\mathrm{BW}}$ with wavefront in $\mathrm{WF}^{\mathrm{BW}}(\mathbf{B})$ disjoint from $L$, and constant $C>0$ such that for all $f_{h,\K}\in H^N_h(F,\El^{\K(h)})$, $h>0$, $\K\in \Omega(h)$ and $k$ large enough 
\begin{align}\label{eq:propag_sink}
\|\mathbf{A}f_h\|_{H^N_h}\leq C\bigl(\|\mathbf{B}\mathbf{P}f_h\|_{H^N_h}+\|\mathbf{A}'f_h\|_{H^N_h}+h^s\|f_h\|_{H^{-s}_h}\bigr)
\end{align}
\item Radial source estimate. Let $\mathbf{P}$ be as in the previous point. Assume that $(\Phi_t^{\omega_{0},\langle\xi_{\mathbb{H}_F^*}\rangle^{-1}\mathbf{p}_Y})$ has on $\overline{\mathbb{H}_F^*}$ a source structure $L'\subset \{\langle \xi \rangle\mathbf{p}_1=0\}\cap \partial \overline{\mathbb{H}_F^*}$. Take $N\in \mathbb{R}^-$ to satisfy the threshold condition $N<\nu$ and $s\in \R^+$ large enough.
Then for all $\mathbf{A}\in \Psi^0_{h,\mathrm{BW}}$ whose wavefront set is near $L'$, there exists $\mathbf{B}\in \Psi^0_{h,\mathrm{BW}}$ with wavefront and elliptic set in a slightly larger neighborhood than $\mathrm{WF}^{\mathrm{BW}}(\mathbf{A})$, and constant $C>0$ such that for all $f_{h,\K}\in H^N_h(F,\El^{\K(h)})$, $h>0$, $\K\in \Omega(h)$ and $k$ large enough 
\begin{align}\label{eq:propag_source}
\|\mathbf{A}f_h\|_{H^N_h}\leq C\bigl(\|\mathbf{B}\mathbf{P}f_h\|_{H^N_h}+h^k\|f_h\|_{H^{-k}_h}\bigr)
\end{align}
\end{itemize}
\end{proposition}
The threshold conditions are particularly simple in this case thanks to the volume preserving assumptions on $X$ and the fact that the contraction and expansion rate of the flow are equal to $1$. More general statement exists but they will not be used in this work. 

In the version of Egorov theorem in the Borel-Weil calculus, we used that the horizontal lift $Y^{\mathbb{H}_F}$ on $F$ of any vector field $Y$ on $M$ induces an operator $\mathbf{Y}_{\K}:=\nabla^{\K}_{Y}$ which is in the calculus. This is a consequence of Proposition \ref{prop:holpreserved}. In the rest of the paper, this is precisely the kind of operator that will be under consideration. Finally let us stress the following.
    \begin{remark}\label{rk:bwchoiceconnection}
        In the uniform calculus, we are free to choose a family of unitary connections to define the quantization. However, in the Borel-Weil calculus, the choice is made at the level of the principal bundle connection $\mathbb{H}$ on $P\to M$. The connections $\overline{\nabla}$ under consideration are then naturally associated to this choice, see (\ref{eq:fullconnection}). 
    \end{remark}
Lastly, we need the following functional spaces.
\begin{definition}\label{def:sobolBW}
    We have the following scale of Sobolev spaces 
    \begin{align*}
        H^s_{h,\mathrm{hol}}(F,\mathbf{L}^\K):=(\mathbf{1}+h^2\Delta_\K)^{-s/2}L^2_{\mathrm{hol}}(F,\El^\K),
    \end{align*}
    where $\Delta_{\K}:=\nabla_\K^*\nabla_\K$, see (\ref{eq:partialconnection}).
     
\end{definition}
The latter space can alternatively be seen to be equal to $\Op_{h,\mathrm{BW}}(\langle \xi_{\mathbb{H}_F^*}\rangle^{-s/2})L^2_{\mathrm{hol}}(F,\El^\K)$, the principal symbol of the horizontal Laplacian $\Delta_\K\in \Psi^2_{h,\mathrm{BW}}$ being $|\xi_{\mathbb{H}_F^*}|^2_{g_F}$ \cite[\S3.3.2]{CL24}.
\subsubsection{Semiclassical measures in the Borel-Weil calculus.} In this paragraph, we fix a sequence $h\mapsto \K(h)$ with $\K(h)\in \Omega(h)$. Each computation (such as principal symbols) will be done with respect to this sequence, see Definition \ref{def:symbunifclass}.\label{subsubsection:measure} We define semiclassical measures associated to $h$-dependent sequences $(v_h)$ in $\mathcal{D}'_{\mathrm{hol}}(F,\El^{\K(h)})$. 
\begin{definition}\label{def:semimeasure}
    Let $(v_h)\in \mathcal{D}_{\mathrm{hol}}'(F,\El^{\K})$ be a sequence of fiberwise holomorphic distribution on $F$. We say that $v_h\rightharpoonup \mu$ with $\mu$ a real and positive Radon measure on $\mathbb{H}_F^*$ if, for all $a\in C^{\infty}_c(\mathbb{H}_F^*)$, we have 
    \begin{align*}
        \langle \Op_{h,\mathrm{BW}}(a)v_h,v_h\rangle_{L^2(F,\El^\K)} \to \int_{\mathbb{H}_F^*} a\,\, d\mu.
    \end{align*}

\end{definition}
We have the following important proposition :
\begin{proposition}\label{prop:propagation_measure}
    Let $(v_h)\in \mathcal{D}_{hol}'(F,\El^{\K(h)})$ be such that there exists $s>0$ and $C>0$ (independent of $h$) verifying $||v_h||_{H^{-s}_{h,\mathrm{hol}}(F,\El^{\K(h)})}\leq C$. Then, after an eventual extraction, $v_h\rightharpoonup \mu$ in the sense of Definition \ref{def:semimeasure}. 
    
    After extraction, assume that $h\K(h)\to \mathbf{l}$. Let $\mathbf{P}\in \Psi^{m}_{h,\mathrm{BW}}(P)$ and write $\Im(\mathbf{P})$ for its imaginary part, as in Proposition \ref{prop:bwproperties2}. We assume for simplicity that both $\mathbf{P}$ and $\Im(\mathbf{P})$ have $h$-independent principal symbols, and that $\mathbf{p}:=\sigma_{h,\K(h),\mathrm{BW}}(\mathbf{P})$ has a Hamiltonian flow preserving the co-horizontal space $\mathbb{H}_F^*$ (see Remark \ref{rk:choicemeasure}). Then, if $(v_h)$ satisfies the following for some $N>0$  
    \begin{align*}
        \|\mathbf{P} v_h\|_{H^{-N}_{h,\mathrm{hol}}}=o(h^{\beta}),
    \end{align*}
    one obtain the following support and propagation properties for $\mu$ 
    \begin{itemize}
        \item $\supp(\mu)\subseteq \bigl\{\mathbf{p}=0\bigr\}$ \text{if $\beta>0$} (ellipticity),\\
        \item $\displaystyle \int_{\mathbb{H}_F^*} \bigl(H_{\mathbf{p}}^{\omega_{\mathbf{l}}}+2\sigma_{h,\mathrm{BW}}(h^{-1}\Im(\mathbf{P}))\,\bigr)a \,d\mu =0$ \text{if $\beta\geq 1$} (Hamiltonian propagation).
    \end{itemize}
    where the principal symbol is computed with respect to the sequence $h\mapsto \mathbf{k}(h)$.
\end{proposition}
By definition, $\mathbf{p}$ is an element of $S^m(\mathbb{H}_F^*)$ and is the restriction to the (co)horizontal space of the symbol $\sigma_{h,\K(h)}(\tilde{\mathbf{P}})$ where $\mathbf{P}:=\Pi_{\K}\tilde{\mathbf{P}}\Pi_{\K}+\mathbf{R}$, $\mathbf{R}\in h^\infty\Psi^{-\infty}$ (which is well-defined by \cite[Lemma 3.3.14]{CL24}). In general $H_{\sigma_{h,\K(h)}(\mathbf{\tilde{P}})}^{\omega_{\mathfrak{l}}}$ has no reason to be tangent $\mathbb{H}_F^*$, this is precisely the hypothesis we make in Proposition \ref{prop:propagation_measure}. After the proof, we provide a useful procedure for generating Hamiltonian functions verifying this.
\begin{proof}
    The proof is almost identical to that of \cite[E.44]{DZ19} in the scalar case. We quickly summarize the proof to emphasize the well-behaved properties of the Borel-Weil calculus. Consider $(h_k)$ such that $h_k\to 0$ as $k\to +\infty$ and write $\Lambda_{h,N}:=\Op_{h,\mathrm{BW}}(\langle \xi_{\mathbb{H}_F^*}\rangle^{-N})$.
    
    Consider the following functional 
    \begin{align*}
        I_{h}(a):=\langle \Op_{h,\mathrm{BW}}(a)v_h,v_h\rangle_{L^2_{\mathrm{hol}}(F,\El^{\K(h)})}, \,a\in C^{\infty}_c(\mathbb{H}_F^*).
    \end{align*}
    On the other hand, one notices that by composition 
    \begin{align*}
        \Op_{h,\mathrm{BW}}(a)=\Lambda_{h,N}^*\Op_{h,\mathrm{BW}}(a\langle \xi_{\mathbb{H}_F^*}\rangle^{2N})\Lambda_{h,N}\,+\,\mathcal{O}_{H_h^{-N}\to H_h^N}(h).
    \end{align*}
    Using this relation, the assumption on $(v_h)$ and $L^2$-boundedness in the uniform calculus gives the following bound
    \begin{align*}
        |I_h(a)|&\leq \|\Op_{h,\mathrm{BW}}(a\langle \xi_{\mathbb{H}_F^*}\rangle^{2N})\Lambda_{h,N}\,v_h\|_{L^2}\|v_h\|_{H^{-N}_{h,\mathrm{hol}}}+\mathcal{O}(h)\\
        &\leq C^2\sup_{(x,\xi)} |\langle \xi_{\mathbb{H}_F^*}\rangle^{-N}a(x,\xi)|+\mathcal{O}(h)\\
        &\lesssim \sup_{(x,\xi)} |a(x,\xi)|,
    \end{align*}
    with an implicit constant depending on $C$ and on the support of $a$. In the second inequality, we use the Calderon-Vaillancourt inequality \cite[Proposition 3.2.3]{CL24}. Taking $\{a_l\}$ a dense family of test functions in $C_c(\mathbb{H}_F^*)$, we can use the previous uniform bound for each $a_l$ to diagonally extract as in \cite[E.44]{DZ19} a subsequence $(I_{h_j}(a))$ which converges for every smooth $a$. We define $I:C^{\infty}_c(\mathbb{H}_F^*)\to \mathbb{R}$ by the previous limits. The uniform bound allows to extend $I$ to a continuous functional on $C^0_c(\mathbb{H}_F^*)$. The Riesz theorem gives the desired measure. That the latter is positive follows from the Gårding inequality in the BW calculus, which itself follows from the usual Gårding inequality, see \cite[Lemma 3.3.16]{CL24}.
   Since $I$ is real valued for real valued $a$, we conclude. For simplicity, we do not continue to label the extracted sequence that allowed to obtain $I_{h_j}(a)\to\mu(a)$.

    For the support property, notice that for all $\mathbf{A}\in \Psi^{\mathrm{comp}}_{h,\mathrm{BW}}$ chosen to be microsupported on $\{\mathbf{p}\neq 0\}$ there exists $\mathbf{Q}\in\Psi^{\mathrm{comp}}_{h,\mathrm{BW}}$ such that $\mathbf{A}=\mathbf{Q}\mathbf{P}+\mathbf{R}$ with $\mathbf{R}\in h^{\infty}\Psi^{-\infty}_{h,\mathrm{BW}}$. Using the estimate on $v_h$
    \begin{align*}
        \langle \mathbf{A}v_h,v_h\rangle &=\langle \mathbf{Q}(\mathbf{P}v_h),v_h\rangle +\mathcal{O}(h^{\infty})\\
        &\to 0,
    \end{align*}
    by definition of $\mu$ we obtain $\int \sigma_{h,\mathrm{BW}}(\mathbf{A})d\mu=0$. As a consequence $\supp(\mu)\subseteq \bigl\{\mathbf{p}=0\bigr\}$.

   As for the propagation property of $\mu$,
    we may construct $\mathbf{A}\in\Psi^{\mathrm{comp}}_{h,\mathrm{BW}}$ self-adjoint with $\sigma_{h,\K(h),\mathrm{BW}}(\mathbf{A})=:a\in C^{\infty}_c(T^*F)$. Then from $\mathbf{P}v_h=o_{H^{-N}_h}(h)$
    \begin{align*}    \dfrac{\Im(\langle\mathbf{P}v_h,\mathbf{A}v_h\rangle)}{h}=\dfrac{\langle [\mathbf{A},\mathbf{P}]v_h,v_h\rangle}{2ih}+\dfrac{\langle \Im(\mathbf{P})\mathbf{A}v_h,v_h\rangle}{ih}\to 0,
    \end{align*}
    By Proposition \ref{prop:unifpropagation} we ontain the desired conclusion.

\end{proof}
The assumption that the Hamiltonian flow of $\mathbf{P}$ preserves the co-horizontal space on $F$ seems important. Otherwise, since our test functions $a$ are defined on $\mathbb{H}_F^*$, $H^{\omega_{h,\mathbf{k}(h)}}_{\mathbf{p}}a$ does not make sense. However, for us this will always be verified for $\mathbf{P}$ since it will always be of the form $\nabla^{\K}_{X^{\mathbb{H}_F}}$, $X\in C^{\infty}(SM,T(SM))$ (and the Hamiltonian flow of basic vector fields preserves $\mathbb{H}_F^*$, see \cite[Proposition 3.2.15]{CL24} or the proof below).
    More generally, an easy way to generate such $\mathbf{p}$ is to pullback a smooth function $f$ on $T^*M$ to $T^*F$ by 
    \begin{align*}
        \pi^*f\,(z,D_{hT}\pi^{\top}(\xi)+\eta):=f(\pi(z),\xi),\,\,z\in F,\xi\in T^*_{\pi(z)}M,\eta\in \mathbb{V}_F^*(z).
    \end{align*}
    \begin{proposition}\label{prop:horizontal_preserved}
        $\pi^*f\in C^{\infty}(T^*F)$ has a Hamiltonian flow $(\Phi_t^{\omega_\mathbf{l},\pi^*f})$ that preserves $\mathbb{H}_F^*$ for any $\mathbf{l}\in (i\mathfrak{t})^*$.
    \end{proposition}
   \begin{proof}
        
    We may prove this statement using local coordinates. Since $F\to M$ is a fiber bundle (whose rank is $m=\frac{n(n-1)}{2}-d$), choose local product coordinates $(x_1,\cdots,x_n,y_1,\cdots,y_m)$ on $U\subset F$ where the $y$ parameter stand for the fiber coordinates. We pick a local basis for the horizontal space $\mathbb{H}_F$ and the vertical space $\mathbb{V}_F$
    \begin{align*}
    H^i &= \partial_{x_i} + \sum_{j=1}^m A_i^j \partial_{y_j}, \quad 1 \leq i \leq n, \\
    V^j &= \partial_{y_j}, \quad 1 \leq j \leq m
\end{align*}

    The adapted coframe is spanned by $(dx_i)_{1\leq i \leq n}$ (which is a local basis of $\mathbb{H}_F^*$) and
    \begin{align*}
        \tilde{y}_{j}=dy_j-\sum_{i=1}^n A_i^jdx_i,
    \end{align*}
    it is clear that $\tilde{y}_j(V_i)=\delta_{ij}$ and $\tilde{y}_j(H_i)=0$, meaning that $(\tilde{y}_j)_{1\leq j \leq m}$ is a local basis of $\mathbb{V}_F^*$ on $U$. We dropped the exterior derivative notation for $\tilde{y}_j$ as we want to stress that this form is not closed: curvature terms of the connection $\mathbb{H}_F$ appear when applying the exterior derivative. A $1$-form on $T^*U$ decomposes as $\sum_ip^idx_i+\sum_\alpha\tilde{p}^\alpha \tilde{y}_{\alpha}$. With respect to the coordinate chart $(x,y,p,\tilde{p})$ of $T^*U$, the Liouville $1$-form reads as $\alpha=\sum_ip^idx_i+\sum_\alpha\tilde{p}^\alpha d\tilde{y}_{\alpha} $. By differentiating $\alpha$, the canonical symplectic form $\omega_0$ on $T^*U$ is
\begin{align*}
\omega_0 &= \sum_i dp_i \wedge dx_i 
 + \sum_{\alpha} d\tilde{p}_\alpha \wedge \tilde{y}_{\alpha} \\
&\quad + \sum_{\alpha}\tilde{p}_{\alpha} \sum_{i,j} \left( \partial_{x_j}A^\alpha_i + \sum_\beta A_j^\beta \partial_{y_\beta}A_i^\alpha \right) dx_j \wedge dx_i \\
&\quad + \sum_{\alpha}\tilde{p}_\alpha \sum_{i,\beta} \partial_{y_{\beta}}A_i^\alpha \, \tilde{y}^\beta \wedge dx_i .
\end{align*}
We will write $H_{i,j}^\alpha:= \partial_{x_j}A^\alpha_i + \sum_\beta A_j^\beta \partial_{y_\beta}A_i^\alpha $ to simplify the notations. Let $X$ be a vector field on $T^*U$, it decomposes in coordinates as 
\begin{align*}
    X=\sum_i X_iH^i+P_i\partial_{p_i}+\sum_{\alpha}Y_{\alpha}V^\alpha+\tilde{P}_{\alpha}\partial_{\tilde{p}_{\alpha}}.
\end{align*}
. Contracting $\omega_0$ by $X$, we obtain 
    \begin{align*}
        \iota_X\omega_0=&\sum_i P_idx_i-X_idp_i\,+\,\sum_{\alpha}\tilde{P}_\alpha\tilde{y}_{\alpha}-Y_{\alpha}d\tilde{p}_{\alpha}\\
        &+\sum_\alpha \tilde{p}_\alpha\sum_{i,j}H^\alpha_{i,j}(X_jdx_i-X_idx_j)\\
        &+\sum_{\alpha} \tilde{p}_{\alpha}\sum_{i,\beta}\partial_{y_\beta}A^\alpha_i( Y_\beta  dx_i-X_i \tilde{y}^\beta).
    \end{align*}
    We need to take into account the magnetic part of $\omega_\mathbf{l}$, since $\mathbf{F}_{\overline{\nabla}}(\mathbb{H}_F,\mathbb{V}_F)=0$ by \cite[Lemma 2.2.11]{CL24}, we write locally
    \begin{align*}
        \mathbf{F}_{\overline{\nabla}}=\gamma^{ij}dx_i\wedge dx_j+\gamma^{\alpha\beta}\tilde{y}^\alpha\wedge \tilde{y}^\beta.
    \end{align*}
The differential of the Hamiltonian function $\pi^*f$ writes in local coordinates 
\begin{align*}
    D(\pi^*f)=\sum_i\partial_{x_i}fdx_i+\partial_{p_i}fdp_i.
\end{align*}
By identification of the $d\tilde{p}^\alpha$ components in $i_X\omega_\mathbf{l}=D(\pi^*f)$ we obtain $Y^\alpha=0$: the Hamiltonian flow of $\pi^*f$ does not leave the fibers of $U\to M$. The co-horizontal space $\mathbb{H}_F^*\cap T^*U$ is defined by the equation $\tilde{p}_{\alpha}=0$, $1\leq \alpha\leq m$. By identification of the $\tilde{y}^\alpha$ components and using that $Y^\alpha=0$ together with $\tilde{p}_{\alpha}=0$
\begin{align*}
    \tilde{P}^\alpha=0,
\end{align*}
which in turn implies that $X=H^{\omega_\mathbf{l}}_{\pi^*f}$ integrate locally to a flow living in $\mathbb{H}_F^*$.

   \end{proof}
   Notice that the assumption that $\mathbf{F}_{\overline{\nabla}}$ has no mixed horizontal - vertical component was key in the previous proof. Otherwise, the identification of the $\tilde{y}^{\alpha}$ component would contain non-vanishing terms of the form $X^i\gamma^{i\alpha}$, preventing us from concluding $\tilde{P}^{\alpha}=0$.
   
   Example of such $\pi^*f$ are $\xi(X^{\mathbb{H}_F})$ for $X\in C^{\infty}(M,TM)$ (so $f=\xi(X)$ on $T^*M$) or $|\xi_{\mathbb{H}_F^*}|^2$, the principal symbol of the horizontal Laplacian on $F$.
\begin{remark}\label{rk:choicemeasure}  We have chosen to define semiclassical measures in the setting of the Borel-Weil calculus, thus the connection $\overline{\nabla}$ is canonical once a choice of connection $\mathbb{H}$ on $P$ is done. Another choice of horizontal space $\mathbb{H}'$ on $P\to M$ would give another semiclassical measure. Let us write $\overline{\nabla}'$ the corresponding family of unitary connections on $\mathbf{L}\to F$. The difference $\overline{\nabla}'-\overline{\nabla}$ is a $1$-form on $F$. Consider a sequence $(v_h)$ in $\mathcal{D}'(F,\El^{\K(h)})$ which converges in the sense of semiclassical measure to $\mu^{\overline{\nabla}}$ (meaning the quantization used is with respect to $\overline{\nabla}$, see Definition \ref{def:bwquantiz}). Then one can notice that $(v_h)$ also converges in the semiclassical sense, with respect to the quantization involving $\overline{\nabla}'$, to the Radon measure $\mu^{\overline{\nabla}'}$. Moreover $\mu^{\overline{\nabla}'}:=(T_{\overline{\nabla}'-\overline{\nabla}})_*\mu^{\overline{\nabla}}$, see (\ref{eq:wfconnection}) for the definition of the translation $T_{\bullet}$ on $T^*F$.
\end{remark}
\section{Frame bundle of hyperbolic manifolds}\label{section:framebundleandhyp} 
\subsection{Frame bundle and frame flow}\label{subsec:frameflow} Let $(M^n,g)$ be an oriented Riemannian manifold and write $SM$ its unit tangent bundle and $\nabla^{LC}$ the Levi-Civita connection on $TM$. We denote by $FM\to M$ the so-called \textit{frame bundle} of $M$. Thanks to the metric $g$, it is a $SO(n)$-principal bundle whose fiber over $x\in M$, $(FM)_x$, consists of the set of oriented orthonormal basis of $T_xM$. Define $\mathcal{N}\to SM$ the \textit{normal bundle} of $M$, the vector bundle consisting over each point $(x,v)\in SM$ of the vector space $v^{\perp}$ (with respect to $g$). $\mathcal{N}$ is a sub-bundle of $\pi_{SM\to M}^*(TM)$ which can be equipped with the metric $g$. More precisely, we can restrict the pulled-back metric $\pi_{SM\to M}^*g$ on $\pi_{SM\to M}^*(TM)$ to $\mathcal{N}\hookrightarrow \pi_{SM\to M}^*(TM)$. The \textit{unitary frame bundle} $\pi:FM\to SM$ is an $SO(n-1)$-principal bundle over $SM$ and is the bundle of oriented and orthonormal frames of $\mathcal{N}$. We will write $\mathbb{V}_{FM}$ the vertical space of this principal bundle (\textit{i.e} $\ker(D\pi)$), and $g_{FM}$ its canonical metric which locally writes as the product of the Sasaki metric on $SM$ and of the usual bi-invariant metric on $SO(n-1)$. 

By construction, the normal bundle $\mathcal{N}$ is a vector bundle associated to the standard vector representation of $SO(n-1)$ on $\mathbb{R}^{n-1}$ 
\begin{align}\label{eq:normalbundle}
    \mathcal{N}\cong FM\times_{\rho_{\mathrm{vect}}}\mathbb{R}^{n-1}
\end{align}
The vector bundle $\pi_{SM\to M}^*(TM)\to SM$ is canonically endowed with the pullback connection $\pi_{SM\to M}^*\nabla^{LC}$, we therefore equip $\mathcal{N}$ with the connection $p_{\mathcal{N}}^\perp(\pi_{SM\to M}^*\nabla^{LC})$ where $p_{\mathcal{N}}^\perp:\pi_{SM\to M}^*TM\to \mathcal{N}$ is the orthogonal projection with respect to $g$.

Write $(\varphi_t)$ the geodesic flow on $SM$ and for $(x,v)\in SM$, denote $\mathcal{P}_t:\mathcal{N}_{(x,v)}\to \mathcal{N}_{\varphi_t(x,v)}$ the parallel transport map along $(\varphi_t)$ induced by the previous connection on $\mathcal{N}$.
The frame flow $(\Phi_t)$ on $FM$ is then defined as the parallel transport of the frames of $\mathcal{N}$ along the orbits of the geodesic flow. Precisely, it is defined by
\begin{align}\label{eq:frameflow}
    \Phi_t(x,v,e_1,\dotsc,e_{n-1})=(\varphi_t(x,v),\mathcal{P}_t(e_1),\dotsc,\mathcal{P}_t(e_{n-1}))\in (FM)_{\varphi_t(x,v)}
\end{align}
where $(e_1,\dotsc,e_{n-1})\in (FM)_{(x,v)}$. We will write $X_{FM}$ the section of $T(FM)$ that generates $(\Phi_t)$. This flow is an \textit{isometric extension of the geodesic flow} in the sense of Definition \ref{def:partiallyhypcompact}, with $P=FM$.

The main consequence of the isometric extension property of Anosov flow is its \textit{partially hyperbolicity}.
\begin{proposition}\label{prop:framepartialhyp}
    There exists two sub-bundles $E^s_{FM},E^u_{FM}$ of $T(FM)$, both of dimension $n-1$, such that 
    \begin{itemize}
        \item $E^s_{FM}$ and $E^u_{FM}$ are $d\Phi_t$ invariant for all $t\in \mathbb{R}$.
        \item $T(FM)=\mathbb{R}X_{FM}\oplus E^u_{FM}\oplus E^s_{FM}\oplus \mathbb{V}_{FM}$.
        \item There exists $C>0$ and $\nu>0$ such that $\|d\varphi_t(v)\|_{\varphi_t(p)}\leq Ce^{-\nu t}\|v\|_p$ for $v\in (E^s_{FM})_p$ and $t>0$, $\|d\varphi_{-t}(v)\|_{\varphi_{-t}(p)}\leq Ce^{-\nu t}\|v\|_p$ for $v\in (E^u_{FM})_p$ and $t>0$.
    \end{itemize}
\end{proposition}

    The constant $\nu$ appearing in the definition may be chosen to be equal to the one appearing in the definition of the Anosov flow $(\varphi_t)$.

Importantly, the sub-bundle $\mathbb{R}X_{FM}\oplus E^u_F\oplus E^s_F$ being complementary to $\mathbb{V}$ in $T(FM)$, we can define the \textit{dynamical connection on $FM$} 
\begin{align}\label{eq:horizontal}
    \mathbb{H}^{\mathrm{dyn}}_{FM\to SM}:=\mathbb{R}X_{FM}\oplus E^u_{FM}\oplus E^s_{FM}.
\end{align}
Let us recall that in full generality the bundles $E_{FM}^{u,s}$ are merely Hölder regular, see \cite{B25} for an extensive treatment of related questions. Thus the horizontal distribution is $C^{\alpha}$ for $\alpha>0$. This connection can equivalently be seen as a $\mathfrak{g}$-valued $1$-form $\Theta\in C^{\alpha}(T^*(FM)\otimes \mathfrak{g})$. However in the setting of hyperbolic manifolds, $\alpha=\infty$.
\subsubsection{Frame bundle  and frame flow on hyperbolic manifolds}\label{subsub:framebundlehyp} The following algebraic description of hyperbolic manifolds and their frame bundle follows \cite[\S3.1]{DyatlovFaureGuillarmou2015}. The universal model used here for hyperbolic manifolds is 
\begin{align*}
    \mathbb{H}^{n} := \bigl\{ x \in \mathbb{R}^{n+1} \,|\, \langle x, x \rangle_{\mathbb{M}} := x_0^2 - \sum_{j=1}^{n} x_j^2 = 1, x_0>0 \bigr\}.
\end{align*}
We will assume from now on that $n\geq 3$. The group of orientation preserving isometries of $\mathbb{H}^n$ is $PSO(1,n)=SO(1,n)/\{\pm Id\}$. Its Lie algebra can be described as follows 
\begin{align}\label{eq:splittinghypframe}
    \mathfrak{so}(1,n)&=\Bigl\{
\begin{pmatrix} 0_{1\times 1} & 0 \\ 0 & \mathbf{k} \end{pmatrix} : \mathbf{k} \in \mathfrak{so}(n)
\Bigr\}
\oplus
\Bigl\{
\begin{pmatrix} 0 & \mathbf{p}^{\top} \\ \mathbf{p} & 0 \end{pmatrix} : \mathbf{p} \in \mathbb{R}^{n}
\Bigr\}\\
&:=\mathfrak{g}'\oplus \mathfrak{p}\notag,
\end{align}
where $\mathfrak{so}(n)$ is the set of antisymmetric matrices of $\mathbb{R}^{n+1}$. Note that $\mathfrak{g}'$ integrate through the exponential map of $PSO(1,n)$ to the maximal compact subgroup $G'\cong SO(n)$. This subgroup is the isotropy group of the element $e_0=(1,0,\dotsc,0)\in \mathbb{H}^n$. We thus have the identification $PSO(1,n)/G'\cong \mathbb{H}^n$, and seeing $SO(n-1)$ as a subgroup $G$ of $G'$  gives that $PSO(1,n)/G\cong S\mathbb{H}^n$. Geometrically, $G$ is the isotropy subgroup of $(e_0,e_1)\in T_{e_0}\mathbb{H}^n$. The frame bundle of the universal cover, $F\mathbb{H}^n$, identifies with $PSO(1,n)$.

The following subalgebra of $\mathfrak{so}(1,n)$ is of importance 
\begin{align*}
    \mathfrak{a}:=\Bigl\{
\begin{pmatrix} 0 & \mathbf{t}^\top \\ \mathbf{t} & 0_{n,n} \end{pmatrix}, \mathbf{t}=(t,\dotsc,0,0)^{\top}, t\in \R
\Bigr\}.
\end{align*}
Let us write $\mathfrak{g}$ the Lie algebra of $G$. In $\mathfrak{so}(1,n)$, there exists two further canonical Lie subalgebras 
\begin{align}\label{eq:rootspacehyp}
    \mathfrak{n}^{\pm}:=\Biggl\{
\begin{pmatrix} 0 & 0 & -\mathbf{v}^\top \\ 0 & 0  & \mp\mathbf{v}^\top \\ - \mathbf{v} & \pm\mathbf{v} & 0_{n-1,n-1} \end{pmatrix}, \mathbf{v}\in\mathbb{R}^{n-1}
\Biggr\},
\end{align}
this clearly gives the following decomposition 
\begin{align}\label{eq:splittinghypcartan}
    \mathfrak{so}(1,n)=\mathfrak{g}\oplus \mathfrak{a}\oplus \mathfrak{n}^+\oplus \mathfrak{n}^-.
\end{align}
Those algebras give an algebraic description of the stable and unstable manifolds of the frame flow (and of the geodesic flow) on hyperbolic manifolds. Introduce the following elements of $\mathfrak{so}(1,n)$ 
\begin{align}\label{eq:commutationrelatalgebra}
    \mathfrak{X}:= E_{0,1}+E_{1,0}, \, R_{i,j}:= E_{i,j}-E_{j,i}, \, U_i^{\pm}:=-(E_{0,i+1}+E_{i+1,0})\mp R_{1,i+1}\, ,
\end{align}
for $1\leq i <j \leq n$ and $E_{i,j}$ the usual basis elements for size $n+1$ matrices. We readily notice that $\mathfrak{g}$ is spanned by $(R_{i,j})_{2\leq i,j\leq n}$ and $\mathfrak{n}^{\pm}$ by $(U_i^{\pm})_{1\leq i \leq n-1}$.

The commutation relations of those elements will appear to be of importance in the following sections 
\begin{align}\label{eq:commutation_algebraic}
&[\mathfrak{X}, U_i^\pm] = \pm U_i^\pm, \quad
[U_i^\pm, U_j^\pm] = 0, \quad
[U_i^+, U_i^-] = 2\mathfrak{X}, \quad 
[U_i^\pm, U_j^\mp] = 2 R_{i+1,j+1},  \\
&\,\,\,\,\,\,\,\,\,\,\,\,\,\,\,\,[R_{i+1,j+1}, \mathfrak{X}] = 0, \quad
[R_{i+1,j+1}, U_k^\pm] = \delta_{jk} U_i^\pm - \delta_{ik} U_j^\pm \notag,
\end{align}
with $1\leq i,j,k \leq n-1$ and $i\neq j$. From now on, we will identify elements in the Lie algebra with their corresponding left-invariant vector field on the corresponding Lie group.

Let us fix $\Gamma$ a torsion-free co-compact lattice of $PSO(1,n)$. The quotient space $\Gamma\setminus PSO(1,n)/G':=M$ is an hyperbolic manifold. $\mathfrak{X}$ induces a left-invariant vector field $X_{F\mathbb{H}^n}$, the generator of the frame flow on the universal cover. Quotienting by $\Gamma$ gives $X_{FM}$, the generator of the frame flow on $FM$. It also descends to a vector field $X$ on $SM:=\Gamma \setminus PSO(1,n)/G$, generating the geodesic flow on $SM$. 
The relations (\ref{eq:commutation_algebraic}) show that the frame flow $(\Phi_t)$ on $M$ is partially hyperbolic, with the following stable and unstable bundles (each of the left-invariant vector fields on $\mathrm{PSO}(1,n)$ descend to $FM$) 
\begin{align}\label{eq:stableunstablehyp}
    E^u_{FM}=\mathrm{span}(U_1^-,...,U_{n-1}^-),\\
    E^s_{FM}=\mathrm{span}(U_1^+,...,U_{n-1}^+)\notag,
\end{align}
where we used that the left-invariant vector fields $(U_i^{\pm})$ descend on $FM$ through the covering map $F\mathbb{H}^n\to \Gamma \setminus F\mathbb{H}^n=FM$ (but notice that they do not descend to $SM$ since $n\geq 3$ by (\ref{eq:commutation_algebraic})).
The vertical space $\mathbb{V}_{FM}$ is generated by $(R_{i+1,j+1})_{1\leq i<j\leq n-1}$.

    The geodesic flow $(\varphi_t)$ on $SM$ is Anosov. The unstable and stable bundles are given by $E^{u/s}_{SM}(\Gamma g G):=D_g\pi(E^{u/s}(g))$. It will be convinient to see that they can alternatively be seen as associated vector bundles. Indeed, we can consider the adjoint representation of $G$ on $\mathfrak{n}^\pm$, $(\mathrm{Ad}(G),\mathfrak{n}^{\pm})$, which is well-defined by (\ref{eq:commutationrelatalgebra}). Then, it can easily be seen that 
    \begin{align}\label{eq:unstableassociated}
        E^u_{SM}=FM\times _{\mathrm{Ad}(G)}\mathfrak{n}^-,\\
        E^s_{SM}=FM\times _{\mathrm{Ad}(G)}\mathfrak{n}^+.\notag
    \end{align}

We also define the subbundles $E_s^*, E_u^*$ and $E_0^*$ of $T^*(SM)$ 
 \begin{align}\label{dualbundles}
     E_s^*(E^s\oplus E^0)=0\notag,\\
     E_u^*(E^u\oplus E^0)=0,\\
     E_0^*(E^u\oplus E^s)=0\notag.
 \end{align}

 \subsubsection{The setting for the Borel-Weil calculus }\label{subsub:settingBWhyp} As explained in Section \ref{section:BWcalculus}, the Borel-Weil calculus of \cite{CL24} is conditionned by a choice of a horizontal distribution on the principal bundle $P\to M$. In our case, the connection we consider for the $SO(n-1)$ principal bundle $FM\to SM$ is the dynamical connection
    $\mathbb{H}^{\mathrm{dyn}}_{FM\to SM}=\mathbb{R}X_{FM}\oplus E^u_{FM}\oplus E^s_{FM}$, defined in (\ref{eq:horizontal}). Since this is the only principal connection we will consider in this note, we will omit most of the time the 'dyn' superscript. We write $\Theta_{\mathrm{dyn}}$ the corresponding connection 1-form on $FM$, with values in $\mathfrak{so}(n-1)$.
    
    Let us stress that there exists \textit{a priori} another canonical distinct choice of horizontal distribution on $FM\to SM$, this choice will have a geometric nature whereas the dynamical connection is dynamical. Indeed, the Levi-Civita connection on $TM\to M$ induces a vector connection on the normal bundle $\mathcal{N}\to SM$ as explained in the beginning of Section \ref{subsec:frameflow}. We can parallel transport element of $FM$ (a frame of $\mathcal{N}$) along any geodesic path on $SM$ for the Sasaki metric. We write $\tau^{LC}_{\gamma}$ such a parallel transport map. Then for $p\in FM$, set 
\begin{align*}\
        \mathbb{H}_{FM\to SM}^{LC}(p):=\biggl\{{\partial_t}_{|t=0}\tau^{LC}_{\gamma}(p),\gamma:(-\varepsilon,\varepsilon)\to M\text{ geodesic s.t }  (\gamma(0),\dot{\gamma}(0))=(x,v) \biggr\}, 
    \end{align*}
 with $\pi_{FM\to SM}(p)=(x,v)$. Thus the Levi-Civita connection on $\mathcal{N}$ induces a $SO(n-1)$ principal bundle connection on $FM\to SM$. 

This connection \textit{a priori} does not coincide with the dynamical connection on $FM$ (\ref{eq:horizontal}). Indeed, the Levi-Civita connection is always smooth whereas the dynamical connection (on $FM$) is merely continuous. But the dynamical connection is smooth here, see (\ref{subsub:framebundlehyp}) for the algebraic description of $E^{u,s}_{FM}$. There is even more: a concrete computation allows to see that those two connections coincide. We will thus simply write $\mathbb{H}_{FM\to SM}$ the principal bundle connection under consideration.  The following remark justifies the previous affirmation.
\begin{remark}\label{rk:LCvsDyn}
    It will suffice to consider only $\mathbb{H}^n$ (resp. $S\mathbb{H}^n$) at $eG'$ (resp. $eG$). The identity element $e\in PSO(1,n)$ is identified with an oriented orthonormal frame at the point $eG$. We will give a description of the horizontal space for the Levi-Civita connection at $e$, which will be a subset of $\mathfrak{g}$: the description of this space at other point will simply be given by the left-translation of the space at $e$.
    The Levi-Civita connection on $T\mathbb{H}^n$ induces a connection $\mathbb{H}^{LC}_{FM\to M}$ (here $FM$ stands for the full oriented orthonormal frame bundle $FM\to M$ whose structure group is $SO(n)$). Indeed, at $e$
    \begin{align}\label{eq:horizontal_LC_M}
        \mathbb{H}_{FM\to M}^{\mathrm{LC}}(e):=\biggl\{{\partial_t}_{|t=0}\tau^{LC}_{\gamma}(e),\gamma:(-\varepsilon,\varepsilon)\to \mathbb{H}^n\text{ geodesic s.t\,} \gamma(0)=eG' \biggr\},
    \end{align}
    Algebraically, this space is equal to $\mathfrak{p}$, defined in (\ref{eq:splittinghypframe}). Every geodesic starting at $eG'$ is explicitly given by $\pi_{FM\to M}(e\,\mathrm{exp}(tY))$ where $Y\in \mathfrak{p}$. The parallel transport of the identity frame $e$ along the geodesic path $\pi_{FM\to M}(\mathrm{exp}(tY))$ is simply $e\,\mathrm{exp}(tY)$. Thus (\ref{eq:horizontal_LC_M}) implies that $\mathbb{H}_{FM\to M}^{LC}(e)=\mathfrak{p}$.
    
    We may now write the identity frame as $e=(eG,\mathbf{u})=:(eG',e_1,\mathbf{u})$ where we recall that $\mathbf{u}$ is such that $(e_1,\mathbf{u})$ is an oriented orthonormal frame. The horizontal space associated to the Levi-Civita connection on $\mathcal{N}\to S\mathbb{H}^n$ at $eG$ is 
    \begin{align*}
        \mathbb{H}_{FM\to SM}^{\mathrm{LC}}(eG):=\mathbb{H}^{\mathrm{LC}}_{FM\to M}(eG')\oplus \biggl\{{\partial_t}_{|t=0}\tau^{S_{eG'}\mathbb{H}^n}_{\gamma}(\mathbf{u}),\gamma:(-\varepsilon,\varepsilon)\to S_{eG'}\mathbb{H}^n\text{ geodesic s.t\,} \gamma(0)=eG'\biggr\},
    \end{align*}
The sphere $S_{eG'}\mathbb{H}^n$ is given the usual round-metric. The parallel transport of the frame $\mathbf{u}$ is given algebraically by $\mathrm{exp}(tZ)$ where $Z\in \mathfrak{g}$ belongs to the matrix space
\begin{align*}
    \mathfrak{q}:=\Biggl\{
\begin{pmatrix} 0 & 0 & 0 \\ 0 & 0  & \mathbf{v}^\top \\ 0 & -\mathbf{v} & 0_{n-1} \end{pmatrix}, \mathbf{v}\in\mathbb{R}^{n-1}
\Biggr\},
\end{align*}
which is a consequence of the fact that the action of elements in $SO(n-1)\subset PSO(1,n)$ fixes $eG$. Finally

\begin{align}\label{eq:horizontal_LC_algebraic}
        \mathbb{H}_{FM\to SM}^{\mathrm{LC}}(e)&=\mathfrak{p}\oplus\mathfrak{q}\\
        &=\mathfrak{a}\oplus \mathfrak{n}^+\oplus\mathfrak{n}^-.\notag
    \end{align}
    The dynamical connection is induced by $\mathfrak{n}^-\oplus \mathfrak{n}^+\oplus \mathfrak{a}$ and this gives the claimed correspondence.
\end{remark}
In order to use the Borel-Weil calculus, we have to choose a fixed maximal torus $T$ of $G$. Its corresponding Lie algebra is chosen to be 
\begin{align}\label{eq:chosentorus}
\mathfrak{t}:=\mathrm{span}\bigl((R_{j,j+1})_{2 \leq j\leq n-1 \text{ even}}\bigr).
\end{align}
We write $d:=\lfloor \frac{n-1}{2}\rfloor$ to be the rank of $SO(n-1)$, which is the dimension of $\mathfrak{t}$. We are now in the setting of Section \S\ref{def:fiberwise} with $P=FM$ and $G=SO(n-1)$. We may form the flag bundle $F:=FM/T\to SM$ and the line bundles $\mathbf{L}:=(L_1,\dotsc,L_{\lfloor \frac{n-1}{2}\rfloor})$ on $F$, see Definition \ref{def:fiberwise}. Following Section \ref{section:fiberwise}, those line bundles are equipped with the horizontal connections $\nabla=(\nabla_1,\dotsc,\nabla_d)$ (resp. full connection $\overline{\nabla}:=(\overline{\nabla}_1,\dotsc,\overline{\nabla}_d)$) associated to the connection $\mathbb{H}_{FM\to SM}$ (resp. $\mathbb{H}_{FM\to F}:=\mathbb{H}_{FM\to SM}\oplus \Re(\mathfrak{n}_{\mathfrak{so}(n-1)}^+\oplus \mathfrak{n}_{\mathfrak{so}(n-1)}^-)$). We insist that $\mathfrak{n}_{\mathfrak{so}(n-1)}^{\pm}$ are the Lie subalgebras corresponding to positive and negative roots with respect to $\mathfrak{t}$, see (\ref{eq: splitting}), and that they have nothing to do with $\mathfrak{n}^{\pm}$.  The corresponding curvature is, as before, written $\mathbf{F}_{\overline{\nabla}}$).
When clear from the context, we will write $X$ in place of $X_F$ for the generator of the frame flow induced on the flag bundle $F$.

Finally, $X_F$ generates a partially hyperbolic flow in the sense of Proposition \ref{prop:framepartialhyp} ($F$ in place of $FM$). We write $E^u_F,E^s_F,E^0_F$ and $\mathbb{V}_{F}$ the bundles from the partially hyperbolic splitting; $E_{u,F}^*,E_{s,F}^*,E_{0,F}^*$ $\mathbb{V}_{F}^*$ are the corresponding dual bundles. We thus see that $\mathbb{H}_{F}=E^u_F\oplus E^s_F \oplus E^0_F$. 

The following property, also appearing in \cite[Lemma 4.7]{DGH2021} and \cite[Lemma 4.2.2]{CL24}, will play an important in view of Lemma \ref{label:hamiltonianflow}.
\begin{lemma}\label{lemma:curvature}
    $i_X \mathbf{F}_{\overline{\nabla}^\K}=0$.
\end{lemma}
\begin{proof}
    The curvature of the line bundles $\mathbf{L}^{\mathbf{k}}$ and the one of the principal bundle (from which we construct the associated line bundles) are related by the following relation. For any $Y,Z\in \mathbb{H}_F(gT)$
    \begin{align}\label{eq:curvaturelinelift}
        \mathbf{F}_{\overline{\nabla}^\K}(X,Y)=\alpha_{\mathbf{k}}\bigl(\,\Pi_{\mathfrak{t}}(d\Theta_{\mathrm{dyn}}(X^{\mathbb{H}_{\text{FM}\to F}},Y^{\mathbb{H}_{\text{FM}\to F}}))\,\bigr)
    \end{align}
    where $\alpha_{\K}$ is the highest weight associated to $\K\in\Lambda$ and $\Pi_{\mathfrak{t}}$ is the orthogonal projection from $\mathfrak{g}$ onto the Lie algebra of the choosen maximal torus $T$. For a proof, see \cite[Lemma 2.2.12]{CL24}.
    
    Contracting by $X^{\mathbb{H}_{\text{FM}}}=X_{FM}$ the curvature of $FM$ gives, for any $Z\in \mathbb{H}_{\text{FM}}$ 
    \begin{align}\label{contraction_by_X_curvature}
        i_X d\Theta_{\mathrm{dyn}}(Z)&=-\Theta_{\mathrm{dyn}}([X_{FM},Z])\\
        &=0,\notag
    \end{align}
    by the commutation relations (\ref{eq:commutation_algebraic}).
\end{proof}
The so-called \textit{horocyclic operators} will play an important role when iterating resolvent estimates for large $\Re(z)$. Let $(\rho,V)$ be a representation of $G$ on a real vector space $V$. Consider $\mathcal{V}_{\rho}\to SM$ the corresponding associated vector bundle, equipped with the canonical connection $\nabla^{\mathcal{V}_{\rho}}$ (constructed through the principal bundle connection $\mathbb{H}_{FM\to SM}$), we can form 
\begin{align}\label{eq:horocyclic_op}
 C^{\infty}(SM,\mathcal{V}_{\rho}) &\to C^{\infty}(SM,\mathcal{V}_{\rho}\otimes E_s^*)\notag \\
\mathcal{U}^- : \,\,\,\,\,\,\,\,\,\,\,\,\,\,\,\,\,\,\,\,\,\,\,\,\,\,s &\mapsto \text{pr}_{E_s^*}\circ\nabla^{\mathcal{V}_{\rho}}s,
\end{align}

where $\text{pr}_{E_s^*}$ stands for the orthogonal projection on the (co)stable bundle; notice that using our conventions, $E_s^*$ identifies with the dual bundle of $E_u$. Those operators have also been used in \cite[Section \S4]{DyatlovFaureGuillarmou2015} or \cite[Lemma 4.2]{KusterWeich2021} to prove the Quantum-Classical correspondence in the locally symmetric setting. The unstable horocyclic operator thus derives section of a vector bundle only along unstable directions, and it extends naturally to distributional sections of $\mathcal{V}_{\rho}$.
Their main feature on hyperbolic manifolds is the following commutation property, also appearing in \cite[Lemma 4.2]{KusterWeich2021}. For a section $t$ of an associated vector bundle $E\to SM$, we recall that $\overline{t}$ stands for the equivariantly lifted section on $FM\to SM$.
\begin{lemma}\label{lemma:commutationhorocyclic}
    Equip the vector bundle $\mathcal{V}_{\rho}\otimes E_s^*\to SM$ with its natural tensor product connection $\nabla^{\mathcal{V}\otimes E_s^*\to SM}$. We have the following commutation relation 
    \begin{align*}\label{eq:horocyclic_op_commutation}
        [\nabla^{\mathcal{V}_{\rho}}_X,\mathcal{U}^-]&:= \nabla^{\mathcal{V}_{\rho}\otimes E_s^*}_X\mathcal{U}^--\mathcal{U}^-\nabla_X^{\mathcal{V}_{\rho}}\\
        &=-\mathcal{U}^-.
    \end{align*}
\end{lemma}
\begin{proof}
    $\mathcal{V}_{\rho}\otimes E_s^*$ is an associated vector bundle for the tensor product of the representation $(\rho,V)$ and the dual of the $\mathrm{Ad}(G)$ representation on $\mathfrak{n}^-$. 
    Thus for any $s\in C^{\infty}(SM,\mathcal{V}_{\rho})$ 
    \begin{equation}\label{eq:horocyclic_lifted}
        \overline{\mathcal{U}^-s}=\sum_{i=1}^{n-1}U_i^-\overline{s}\otimes (U^-_i)^* \text{ on } FM,
    \end{equation}
    where each $(U_i^-)^*$ is the dual of the left-invariant vector field $U_i^-$. Each $(U_i^-)^*$ can be seen as a constant function in $C^{\infty}(FM,(\mathfrak{n}^-)^*)$. We thus see that 
    \begin{align}\label{eq:hroocyclic}
        &\overline{\mathcal{U}^-\nabla^{\mathcal{V}_{\rho}}_Xs}=\sum_{i=1}^{n-1}U_i^-(X\,\overline{s})\otimes (U^-_i)^*,\\
        &\overline{\nabla^{\mathcal{V}_{\rho}\otimes E_s^*}_X\mathcal{U}^-s}=\sum_{i=1}^{n-1}X(U_i^-\,\overline{s})\otimes (U^-_i)^*,
    \end{align}
    and the commutation relations $[X,U_i^-]=-U_i^-$ gives the desired relation.
    
    The result follows from the commutation relation $[\mathfrak{X},U_i^-]=-U_i^-$.
\end{proof}
\textbf{Some important notation.} We recall that $(\Phi_t)$ will stand for the frame flow on $FM$ and $(\varphi_t)$ for the geodesic flow on $SM$. The vector field $X_{FM}$ induces a vector field on the flag bundle $F$ that we write $X_F$. For a smooth function $a\in C^{\infty}(T^*F)$, we write $(\Phi^{\omega_{\mathbf{l}},a}_t)$ its Hamiltonian flow with respect to $\omega_0+i\mathbf{l}\cdot \mathbf{F}_{\overline{\nabla}}$ and $H^{\omega_{\mathfrak{l}}}_a$ the corresponding vector field. The flow of $X_F$, written $(\Phi_t^F)$, is partially hyperbolic. The splitting is written
\begin{align}\label{eq:splittingflag}
    TF:=E^u_F\oplus E^s_F\oplus \mathbb{R}X_F\oplus \mathbb{V}_F
\end{align}
To prevent eventual confusions, Hamiltonian flows will always have a superscript indicating the symplectic form in consideration: this will be crucial at some point, and this prevents the confusions with the partially hyperbolic flows $(\Phi_t^F)$ and $(\Phi_t)$.

\section{Anisotropic spaces for $\X_\K$ and $X_{FM}$}\label{section:anisotropicspace}

\subsection{Anisotropic spaces for each Fourier mode }\label{subsection:anisofourierwise}By the $G$-equivariance property of $X_{FM}$, the Borel-Weil calculus allows us to decompose this vector field as a family $(\X_\K):=(\nabla_{X_F}^{\K})_{\K\in \hat{G}}$ of differential operators, each acting on $\mathcal{D}'_{\mathrm{hol}}(F,\El^\K)$. Following the foundational work \cite{FS11}, the authors in \cite[\S4.2]{CL24} have defined Hilbert spaces $\mathcal{H}^N_{h,\mathrm{hol}}$ ($N\in \mathbb{R}^+$) on which $z\mapsto (-\X_\K-z)^{-1}$ is a meromorphic family of operator on $\{\Re(z)>-cN\}$, $c$ being a constant depending on the expansion rate of the flow of $(\phi_t)$. Here, $c=1$ because the expansion rate of the geodesic flow is equal to $1$ on hyperbolic manifolds.

For any real number $N>0$ we may define the Hilbert spaces $\mathcal{H}^N_{{h,\mathrm{hol}}}$ in the following way. The Hamiltonian flow of $X_F$ with respect to $\omega_{h,\mathbf{k}}$ is $(\Phi_t^{\omega_{h,\K},\xi(X_F)})$. Lemma \ref{label:hamiltonianflow} gives that $\Phi_t^{\omega_{h,\K},\xi(X_F)}= \Phi_t^{\omega_{0},\xi(X_F)}=(\Phi_t^F,(d\Phi_{-t}^F)^\top)$, and Lemma \ref{prop:horizontal_preserved} justify that such a flow preserves the co-horizontal space of $F$, namely 
\begin{align}\label{eq:dualsplittingflag}
    \mathbb{H}_F^*:=E_{s,F}^*\oplus E_{u,F}^*\oplus E_{0,F}^*.
\end{align}
This setting is convenient because $(\Phi_t^{\omega_{0},\xi(X_F)})_{|\mathbb{H}_F^*}$ presents a \textit{sink and source structure}. More precisely in the setting of hyperbolic manifolds, there exists $C>0$ such that 
    \begin{enumerate}
        \item If $\xi\in E_{s,F}^*$, $|(d_x\Phi^F_t)^{-\top}(\xi)|_{\varphi_t(x)}\leq C e^{-t} |\xi|_x$ for $t\geq 0$.
        \item If $\xi\in E_{u,F}^*$, $|(d_x\Phi^F_t)^{-\top}(\xi)|_{\varphi_t(x)}\leq C e^{t} |\xi|_x$ for $t\leq 0$.
    \end{enumerate}
This allows us to define the corresponding {order function} $\mathbf{m}$ and the associated \textit{escape function} $\mathbf{G}_\mathbf{m}$.
\begin{proposition}\label{prop:escapefunctionflag}
    For $-u,s>0$, there exists $\mathbf{m}_{u,s}\in C^{\infty}(\mathbb{H}_F^*,[u,s])$ homogeneous of degree $0$ for $|\xi_{\mathbb{H}_F^*}|$ large enough such that: $\mathbf{m}_{u,s}\equiv s$ (resp. $\mathbf{m}_{u,s}\equiv u$) in a (co)horizontal conical neighborhood of $E_{s,F}^*$ (resp. $E_{u,F}^*$), $\mathbf{m}_{u,s}\equiv 0$ near $E_{0,F}^*\cap\mathbb{H}_F^*$ and $H_{\xi(X_F)}^{\omega_0}\mathbf{m}_{u,s}\leq 0$.\\
    Moreover, the escape function $\mathbf{G}_{\mathbf{m}_{u,s}}:=\mathbf{m}_{u,s}\log(\langle \xi_{\mathbb{H}_F}^*\rangle)\in C^{\infty}(\mathbb{H}_F^*)$ satisfies for $\xi$ with large enough norm: $H_{\xi(X_F)}^{\omega_0}\mathbf{G}_{\mathbf{m}_{u,s}}\leq 0$ and $H_{\xi(X_F)}^{\omega_0}\mathbf{G}_{\mathbf{m}_{u,s}}\leq \min(-u,s)$ near $(E_{s,F}^*\cap E_{u,F}^*)\cap \mathbb{H}_F^*$
\end{proposition}
\begin{proof}
    Since we only consider $\mathbb{H}_F^*$, the proof is \textit{verbatim} the same as in \cite[\S4.1.2]{FS11} and \cite[\S9.1]{Lefeuvre2026}.
\end{proof}
This allows us to define the desired anisotropic spaces. The symbol $e^{\mathbf{G}_{\mathbf{m}_{u,s}}}$ is in $S_{1-}^{\mathbf{m}_{u,s}(\cdot)}(\mathbb{H}_F^*)$, see \cite[\S6]{Lefeuvre2026} for the introduction of symbols of variable order (which clearly also makes sense in the context of the Borel-Weil calculus). Modifying $\Op_{h,\mathrm{BW}}(e^{\mathbf{G}_{\mathbf{m}_{u,s}}})$ up to a lower order term give the operator $\mathbf{A}_{h,\K}(\mathbf{m}_{u,s})\in \Psi^{\mathbf{m}_{u,s}(\cdot)}(F,\El^\K)$, which is elliptic, formally self-adjoint and invertible. We can thus define 
\begin{align}\label{eq:anisotropicdefinition}
    \mathcal{H}^{[u,s]}_{h,\mathrm{hol}}(F,\El^\K):=\mathbf{A}_{h,\K}(\mathbf{m}_{u,s})^{-1}L^2(F,\El^\K),
\end{align}
which can be equipped with the natural inner product $\langle \mathbf{A}_{h,\K}(\mathbf{m}_{u,s})\cdot, \mathbf{A}_{h,\K}(\mathbf{m}_{u,s})\cdot\rangle_{L^2}$ making it a Hilbert space in which $C^{\infty}_{\mathrm{hol}}(F,\El^\K)$ embeds densely. Most of the time the dependency in $\K$ will be omitted, we will simply write $\mathcal{H}^{[u,s]}_{h,\mathrm{hol}}$ for the previous space.\\
\begin{remark}\label{rk:orderfunction}
    Let $N\in \mathbb{R}^+_*$. In the following pages, we may assume that $u=-N$, $s=N$ and in this case we will write $\mathcal{H}^N_{h,\mathrm{hol}}:=\mathbf{A}_{h,\K}(\mathbf{m}_{-N,N})^{-1}L^2$. Thus for instance $\mathcal{H}^N_{h,\mathrm{hol}}$ is microlocally equivalent to $H^{N}_{h,\mathrm{hol}}$ (resp. $H^{-N}_{h,\mathrm{hol}}$) near $E_{s,F}^*$ (resp. near $E_{u,F}^*$). Here and in the rest of the paper we may only consider order functions of the type $\mathbf{m}_{-N,N}$ (written $\mathbf{m}_{N}$ for simplicity), except in the following instance.  
    
    As in \cite[Lemma 4.9]{DGH2021} or \cite[Proposition 4.6]{CG21} for instance, we will at some point need an unstable horocyclic invariance property of particular elements in $\mathcal{H}^N_{h,\mathrm{hol}}$ called \textit{resonances}. After applying such a horocyclic operator, we obtain an element of $\mathcal{H}^{[-N-1,N-1]}_{h,\mathrm{hol}}=:\mathcal{H}'^N_{h,\mathrm{hol}}$, its order function is $\mathbf{m}'_N:=\mathbf{m}_{N}-1/N$ with escape function $\mathbf{G}':=\mathbf{G}_{\mathbf{m}'_N}$. The latter function is not an order function in the precise sense of Proposition \ref{prop:escapefunctionflag} since it is not zero anymore near $E_0^*$ but one can check that the construction still works when requiring that $\mathbf{m}$ is non-zero near $E_0^*$ with value strictly between $u$ and $s$, see \cite[Lemma 8.6]{CekicMicrolocalMethods}. In order to still have that $H_{\xi(X_F)}^{\omega_0}\mathbf{G}'$ is strictly negative near the stable and unstable codirections, it is important that $N>1$ (see the decay property of Proposition \ref{prop:escapefunctionflag}). In this case, $\mathcal{H}'^N_{h,\mathrm{hol}}$ still define an anisotropic space on which the following spectral properties of $\X_\K$ remain valid.
\end{remark}
The following is the main result of \cite[\S4.1.2]{FS11} and \cite[\S9.1]{Lefeuvre2026}, which stated in the setting of the Borel-Weil calculus corresponds to \cite[\S4.3.1]{CL24}.
\begin{proposition}\label{prop:existenceresonances}
    The operator $(\mathbf{X}_{\mathbf{k}}+z)$ is, restricted to its canonical domain on $\mathcal{H}^{[u,s]}_{h,\mathrm{hol}}$, Fredholm of index $0$ in the band $\bigl\{\Re(z)>-\min(-u,s)\}$. The resolvent $(-\mathbf{X}_{\mathbf{k}}-z)^{-1}$ is well-defined and holomorphic for $\Re(z)\gg0$ as an operator on $\mathcal{H}^{[u,s]}_{h,\mathrm{hol}}(F,\mathbf{L}^{\mathbf{k}})$. Thus by the analytic Fredholm Theorem \cite[Section 21.1.4.2]{CL24}, $(-\mathbf{X}_{\mathbf{k}},D_{\mathcal{H}^{[u,s]}}(\mathbf{X}_\mathbf{k}))$ has a pure-point spectrum in the region $\bigl\{\Re(z)>-\min(u,s)\}$, independent of the choice made in the construction.
\end{proposition}
We now work only with $\mathcal{H}^N$ in view of the previous remark (the case of $\mathcal{H}'^N$ is dealt with exactly in the same way). A point $z_0\in \C$ will then be called a \textit{resonance} if there exists $N>0$ large enough such that it is a pole of $z\mapsto (-\X_\K-z)^{-1}$ on $\mathcal{H}^{u,s}_{h,\mathrm{hol}}$. Alternatively, it is a resonance if and only if $\ker_{\mathcal{D}'_{E_{u}^*}}(\X_\K+z_0)\neq 0$. The elements of the previous set are called \textit{resonant states}. Finally, a distribution $u\in \mathcal{H}^N$ is called a generalized state at $z_0$ if $u\in \ker(\X_\K+z_0)^J$ for $J\geq 2$. The set of all resonances of $\X_\K$ is written $\sigma_{\mathrm{PR}}(\X_\K)$.
\begin{remark}\label{rk:noellitpicityanisotropic}
    It could be tempting to try to define anisotropic spaces on $FM$ directly. Indeed the Hamiltonian function $(x,\xi)\mapsto\xi(X_{FM}(x))$ actually presents a sink-source structure on the whole of $T^*(FM)$, replacing $E_0^*$ by $E_0^*\oplus \mathbb{V}_{FM}^*$. However, the lack of ellipticity of $X_{FM}$ in the vertical (co)directions prevents from generalizing the proof from \cite[\S9.1]{Lefeuvre2026}. This is why the holomorphicity is key in \cite[\S4]{CL24} (and hence in our paper): it allows to get rid of the vertical codirections, at the cost of performing a Fourier decomposition.
\end{remark}
Also, notice the following important fact. The identification (\ref{eq:identification}) allows to realize an element of $\mathcal{D}'_{\mathrm{hol}}(F,\El^\K)$ as a distributional section in $\mathcal{D}'(SM,H^0(F,\El^\K))$. Through this identification, $\X_\K$ is identified with $\nabla^{H^0(F,\El^\K)}_X$. Let us justify this assertion.
\begin{lemma}\label{lemma:intertwinconnection}
    There exists a smooth isomorphism $i_\K:C^{\infty}_{\mathrm{hol}}(F,\El^\K)\to C^{\infty}(SM,H^0(F,\El^\K))$ that satisfies
    \begin{align}\label{eq:intertwinconnection}
        \X_\K (i_{\K}s)=i_\K(\nabla_X^{H^0(F,\El^\K)} s),
    \end{align}
    for any $s\in C_{\mathrm{hol}}^{\infty}(F,\El^\K)$. More generally, if $Y\in C^{\infty}(SM,T(SM))$ and $Y^{\mathbb{H}_F}\in C^{\infty}(F,\mathbb{H}_F)$ is its horizontal lift to $F$, we have 
    \begin{align*}
        i_{\K}(\nabla_Y^{H^0(F,\El^\K)} s)=\nabla^\K_{Y^{\mathbb{H}_F}}(i_\K s)
    \end{align*}
\end{lemma}
\begin{proof}
    Recall that a section $s\in C_{\mathrm{hol}}^{\infty}(F,\El^\K)$ equivariantly lifts to a $T$-equivariant section $\overline{s}\in C^{\infty}_{T}(FM,\C)$ (satisfying an additional condition for holomorphicity). Also, $s'\in C^{\infty}(SM,H^0(F,\El^\K))$ lifts to the $G$-equivariant section $\overline{s'}\in C_G^{\infty}(FM,H^0(G/T,\mathbf{J}^\K))$, see (\ref{eq: fiber_id}). There is a natural mapping $\overline{j_\K}$ constructed as follows
    \begin{align*}
       C_G^{\infty}(FM,H^0(G/T,\mathbf{J}^\K))   &\longrightarrow C_T^{\infty}(FM,\mathbb{C}) \\
\overline{j_\K}: \overline{s} &\longmapsto \bigl(p\mapsto\Psi_T(\overline{s}(p))(e)\bigr),
    \end{align*}
    with $\Psi_T:H^0(G/T,\mathbf{J}^\K)\to C^\infty_T(G,\C)$ being the $T$-equivariant lift mapping. The elements in the image of $\overline{j_\K}$ are indeed $T$-equivariant, for all $t\in T$ 
    \begin{align*}
        \overline{j_\K}(\overline{s})(p\cdot t)&=\Psi_T(\Phi(t^{-1})\overline{s}(p))\,(e)\\
        &=\Psi_T(\overline{s}(p))\,(t)\\
        &=\gamma_\K(t^{-1})\overline{j_\K}(\overline{s})(p)
    \end{align*}
where we used the action of the Borel-Weil $\Phi$ representation on equivariant lifts in the first line, see Remark \ref{equiv}. It is clear that for any vector field $X\in C^{\infty}(P,TP)$ we have $\overline{j_\K}(X\,\overline{s})=X(\overline{j_\K}(\overline{s}))$ by linearity of the evaluation map, thus giving the desired result if we show that $\overline{i}_\K$ is bijective. By the $G$-equivariance property of elements in $C^{\infty}_G(FM,H^0(G/T,\mathbf{J}^\K))$, injectivity follows. Indeed
\begin{align}\label{eq:equivariance_BW_proof}
    \overline{j_\K}(\overline{s})(p\cdot g)&=\Psi_T(\Phi(g^{-1})\overline{s}(p))\,(e)\\
    &=\Psi_T(\overline{s}(p))(g),\notag
\end{align}
thus if $\overline{s}$ is such that $\overline{j_\K}(\overline{s})=0$, (\ref{eq:equivariance_BW_proof}) implies that $\overline{s}=0$. Finally, let $\overline{s}_T\in C_T^\infty(M,\mathbb{C})$. We can uniquely define $\overline{s}_G\in C^\infty(FM,H^0(G/T,\mathbf{J}^\K))$ so that $\Psi_T(\overline{s}_G(p))(e)=\overline{s}_T(p)$ thanks to (\ref{eq:equivariance_BW_proof}) again. Let $j_\K$ be the isomorphism induced by $\overline{j_\K}$ at the level of associated bundles, it's inverse is written $i_\K$ and is the desired mapping.

\end{proof}

It is a well-known fact that, for any vector bundle $E\to SM$ equipped with a connection $\nabla_E$, the construction of (semiclassical) anisotropic spaces carries on in the case of the operator $\nabla_X$ \cite[\S9.1]{Lefeuvre2026}, the latter are written $\mathcal{H}^N_h(SM,E)$. Here, the isomorphism $i_\K$ maps an element $u_\K\in \mathcal{H}^N_{h,\mathrm{hol}}(F,\El^\K)$ satisfying $(-\X_\K-z)u_\K=0$ to an element $i_\K(u_\K)\in \mathcal{H}^N_{h}(SM,H^0(F,\El^\K))$ satisfying 
\begin{align}\label{eq:renanceintertwin}
    (-\nabla_X^{H^0(F,\El^\K)}-z)\,i_\K(u_\K)=0.
\end{align}
This will be useful in Section \ref{Section:resolvent_estim_re_bigger_1}.
\begin{remark}\label{eq:choicemaximaltorusres}
    The operator $\X_\K$ acts on fiberwise holomorphic sections of $\El^\K\to F$ and thus depend on a choice of maximal torus $T$. However, the set of resonances does not depends on the chosen maximal torus. This can easily be seen by the fact that resonances of $\X_\K$ correspond to resonances on $H^0(F,\El^\K)$ as explained above, which further correspond to resonances on $E^{\lambda(\K)}$ by the Borel-Weil theorem, see the definition of the fiberwise Fourier transform (\ref{Fourier}). Indeed, the isomorphism of representation between $H^0(G/T,\mathbf{J}^\K)$ and $V^{\lambda(\K)}$ induce an isomorphism of the associated vector bundles. This isomorphism will intertwin the natural associated connection on $SM$.
 \end{remark}
 Finally, the Borel-Weil calculus clearly generalizes in a vector valued setting as noticed in \cite[Remark 3.3.17]{CL24}. We can define the class $\Psi_{h,\mathrm{BW}}(FM,E_1\to E_2)$ for $E_1,E_2$ two fiberwise holomorphic vector bundles over $F$. Thus for instance, up to taking an invertible, self-adjoint and elliptic lower order perturbation $\mathbf{A}_{h,\K}(m_{u,s},E_1)$ of $\Op_{h,\mathrm{BW}}(\mathrm{Id}_{E_1}e^{G_{m_{u,s}}})$, we can define the anisotropic space 
 \begin{align}\label{eq:resonancevecorbundle}
     \mathcal{H}^{[u,s]}_{h,\mathrm{hol}}(F,\El^\K\otimes E_1):=\mathbf{A}_{h,\K}(m_{u,s},E_1)^{-1}L^2(F,\El^\K\otimes E_1).
 \end{align}
 \textbf{A convention.} We write $\mathbf{A}^N_{h,\K}:=\mathbf{A}^N_{h,\K}(\mathbf{m}_{-N,N})$ and $\mathbf{A}'^N_{h,\K}:=\mathbf{A}^N_{h,\K}(\mathbf{m}_{-N-1,N-1})$ the anisotropic operators corresponding respectively to $\mathcal{H}^N$ and $\mathcal{H}'^N$, see Remark \ref{rk:orderfunction}. In the twisted case (by a fiberwise holomorphic vector bundle $E$ say), we will write the previous operators $\mathbf{A}_{h,\K}^N(E)$ and $\mathbf{A}_{h,\K}'^N(E)$.
\subsubsection{Anisotropic spaces for $X_{FM}$}\label{subsub:anisotropicframeflow} After performing a Fourier decomposition of the operator $X_{FM}$ as a family of operators in the Borel-Weil calculus, we can go in the other direction to define functional spaces adapted to $X_{FM}$. They were introduced in \cite[Section 4.1]{DGH2021} in the case of circle extensions and defined in full generality in \cite[Section 4.4.1]{CL24}.
\begin{definition}\label{def:anisotorpicnormFM}
    We define the following bi-graded norm, for any $f\in C^{\infty}(FM)$ 
    \begin{align*}
        ||f||^2_{\mathcal{H}^{m,s}}:=\sum_{\mathbf{k}\in \hat{G}}d_\K\langle \K \rangle^{2m}\sum_{i=1}^{d_\K} ||f_{\mathbf{k},i}||^2_{\mathcal{H}^s_{\langle \K\rangle^{-1},\mathrm{hol}}}.
    \end{align*}
    In turn, $\mathcal{H}^{s,m}(FM)$ is defined as the completion of $C^{\infty}(FM)$ for the previous norm.
\end{definition}

\textbf{The technical result leading to Theorem \ref{thm:mainthm} } We may now state the main technical theorem that will allow us to obtain the resolvent estimate in Theorem \ref{thm:mainthm}. It is a component-wise estimate for the resolvent of $\X_\K$ on the anisotropic spaces defined Proposition \ref{prop:existenceresonances}. Recall first notation $\mathcal{B}_k(\varepsilon)$ defined in (\ref{eq:band_def}), with $\varepsilon\in (0,1)$.  

\begin{theorem}\label{thm:mainthmfiberwise}
    For any $k\in \N$ and any $\varepsilon>0$, the set $\bigcup_{\K\in \Lambda}\sigma_{\mathrm{PR}}(\X_\K)$ is finite in $\mathcal{B}_k(\varepsilon)$. Let $s\in \mathbb{N}^*$, for any $z\in \mathbb{C}_{\varepsilon}$ such that $s>-\Re(z)$ there exists $C_{\varepsilon,\lceil \Re(z)\rfloor}>0$ such that 
    \begin{align}\label{eq:fiberwiseestimates}
        \|(-\X_{\K}-z)^{-1}\|_{\mathcal{H}^{s}_{\langle \K \rangle^{-1},\mathrm{hol}}}\leq C_{\varepsilon,\lfloor \Re(z)\rfloor}\langle \mathbf{k}\rangle^{-\lfloor \Re(z) \rfloor}\langle \Im(z)\rangle^{-\lfloor \Re(z) \rfloor+s},\,|\Im(z)|\gg 1
    \end{align}
\end{theorem}

\begin{remark}\label{rk:finitenumberres}
    Note that estimate (\ref{eq:fiberwiseestimates}) is stated for $z$ with large enough imaginary part. Actually using the Quantum-Classical correspondence as in \cite[Appendix A]{DGH2021}, one can expect that the resonances in each $\mathcal{B}_k(\varepsilon)$ have zero imaginary part. Also, since $\bigcup_{\K\in \Lambda}\sigma_{\mathrm{PR}}(\X_\K)$ is finite in each strip, the estimate holds in the whole $\mathcal{B}_k(\varepsilon)$ for $\langle \K \rangle$ large enough.
    
    Finally, we stated resolvent estimate with $\Re(z)\leq 0$, but inspecting the proof in Section \ref{section:resolventestimatesfirstband} we can actually extend to any $\Re(z)\geq 0$: the exponent over $\langle \K \rangle$  will simply be $1$. 
\end{remark}

We may now justify how Theorem \ref{thm:mainthmfiberwise} implies Theorem \ref{thm:mainthm}. The following sections will then be devoted to the proof of Theorem \ref{thm:mainthmfiberwise}
\begin{lemma}\label{lemma:thmimpliesthm}
    Theorem \ref{thm:mainthmfiberwise} implies Theorem \ref{thm:mainthm}.
\end{lemma}
\begin{proof}
    This is an easy adaptation of the argument in \cite[Corollary 4.4.3]{CL24}. We fix $\varepsilon>0$ and $z\in \mathcal{B}_k(\varepsilon)$ not in $\bigcup \sigma_{\mathrm{PR}}(\X_\K)$. Let $f\in C^{\infty}(FM)$, for any $s>-\Re(z)$ we set 
    \begin{align}\label{eq:globalresolventfromfourier}
        (-X_{FM}-z)^{-1}f:=\mathcal{F}^{-1}\bigl( ((-\X_\K-z)^{-1}f)_\K\bigr).
    \end{align}
    For $\Re(z)>0$, the previous formula agrees with $(-X_{FM}-z)^{-1}$ seen as a bounded operator on $L^2(FM)$ ($X_{FM}$ is anti self-adjoint with respect to its natural domain in $L^2$). Theorem \ref{thm:mainthmfiberwise} allows to perform the meromorphic continuation for each Fourier modes. Using resolvent estimates (\ref{eq:fiberwiseestimates}), we have for any $m\in \mathbb{R}$ 
    \begin{align*}
        \|(-X_{FM}-z)^{-1}f\|_{\mathcal{H}^{m,s}}^2&=\sum_{\K\in \Lambda}d_{\K}\langle \K \rangle^{2m}\sum_{i=1}^{d_\K}\|(-\X_\K-z)^{-1}f_{\K,i}\|_{\mathcal{H}^N_{\langle \K \rangle^{-1},\mathrm{hol}}}^2\\
        &\leq C_{\varepsilon}\langle z \rangle^{2(-\lfloor \Re(z) \rfloor+s)} \sum_{\K\in \Lambda}d_{\K}\langle \K \rangle^{2(m-\lfloor \Re(z) \rfloor)}\sum_{i=1}^{d_\K}\|f_{\K,i}\|_{\mathcal{H}^N_{\langle \K \rangle^{-1}},\mathrm{hol}}^2\\
        &\leq C_{\varepsilon}\langle z \rangle^{2(-\lfloor \Re(z) \rfloor+s)}||f||_{\mathcal{H}^{m-\lfloor \Re(z) \rfloor,s}}
    \end{align*}
    hence the desired implications. 
\end{proof}
\begin{remark}\label{rk:}
    To a resonant state $u_\K\in\mathcal{H}^N_{\mathrm{hol}}(F,\El^\K)$ at $z_0$ corresponds a resonant state $u:=\mathcal{F}^{-1}(\delta_{\K'\K}u_\K)$ in $FM$ which satisfies $(-X_{FM}-z_0)u=0$. From Theorem \ref{thm:mainthmfiberwise} we see that each resonant state on $FM$ corresponding to $z_0\in \mathcal{B}_k(\varepsilon)$ has a finite number of Fourier component: each of its Fourier component is a resonant state on a $\mathcal{H}_{h,\mathrm{hol}}(F,\El^\K)$.
\end{remark}
\setcounter{tocdepth}{2}
\section{Resolvent estimates of $\X_{\mathbf{k}}$ in $\Re(z)>-1$}\label{section:resolventestimatesfirstband}

\subsection{The semiclassical formulation}\label{subsection:semiclassicalformul}

Take $\varepsilon\in (0,1)$ and write $\{\beta_j\}_{j\leq k(\varepsilon)}$ to be the finite set of all the poles of $z\mapsto (-X_{FM}-z)^{-1}$ in the band $\mathcal{B}_0(\varepsilon)$, see Section \ref{section:introandmain} for the notation. That this set is finite is justified at the end of Section \ref{subsec:proof_re_minus_1}. We define 
\begin{align*}
    \lambda_{\varepsilon}:=\max_{j\leq k(\varepsilon)} |\Im(\beta_j)|,
\end{align*}
allowing us to define the region where we will perform the resolvent estimates, namely $\mathcal{B}_0(\varepsilon)':=\mathcal{B}_0(\varepsilon)\,\cap\, \{|\Im(z)|\geq \lambda_{\varepsilon}+1\}$. This restriction is only a matter of notational convenience, we do not seek to optimize the region where we prove resolvent estimates; also as explained in Remark \ref{rk:finitenumberres}, for $\langle \K \rangle$ large enough the estimate can be proved on the whole strip $\mathcal{B}_0(\varepsilon)$.
\begin{remark}
    Actually, using the paper \cite{DyatlovFaureGuillarmou2015}, we can strongly suspect that $\lambda_{\varepsilon}$ is $0$. For instance in the $n=3$ case, the work \cite[Appendix A]{DGH2021} proves that $\lambda_{\varepsilon}=0$. However the proof should rely on the \textit{quantum-classical correspondence}, which is far from understood outside of the setting of locally symmetric spaces. Since we want this paper to be as self-contained as possible, we have to accept the lack on knowledge on $\lambda_{\varepsilon}$.
\end{remark}
In order to prove the estimate in Theorem \ref{thm:mainthmfiberwise}, we can prove two separate estimates (as done in \cite[Section 4.5.1]{CL24} to obtain resolvent bound on $\{\Re(z)\geq 0\}$)  
\begin{proposition}\label{prop:two_asymptotics_estimates}
    Let $N>1$. There exists $C_{\varepsilon}$ such that for all $z\in \mathcal{B}_0(\varepsilon)'\cap\{\Re(z)\in(-1,0]\}$ 
    \begin{align}
\|(-\X_{\K}-z)^{-1}\|_{\mathcal{H}^{N}_{\langle \K \rangle^{-1},\mathrm{hol}}}
&\le C_{\varepsilon}\langle \mathbf{k}\rangle,
&\qquad&
\text{if }\langle \mathbf{k}\rangle \ge \langle \Im(z)\rangle,
\label{eq:case-k}
\\[0.5em]
\|(-\X_{\K}-z)^{-1}\|_{\mathcal{H}^{N}_{\langle \Im(z)\rangle^{-1},\mathrm{hol}}}
&\le C_{\varepsilon}\langle \Im(z)\rangle,
&\qquad&
\text{if }\langle \Im(z)\rangle \ge \langle \mathbf{k}\rangle.
\label{eq:case-im}
\end{align}
where $N$ is chosen strictly greater than $-\lceil \Re(z)\rceil$.
\end{proposition}
Estimate (\ref{eq:case-k}) (resp. (\ref{eq:case-im})) will be called the \textit{large $\langle \K \rangle$ regime} (resp. \textit{the large $\langle \Im(z)\rangle$ regime}). Let us quickly justify why (\ref{eq:case-k}) and (\ref{eq:case-im}) imply estimate (\ref{eq:fiberwiseestimates}). This is also justified in \cite[Section 4.5.1]{CL24}, but for the reader's convenience, we repeat the argument. That (\ref{eq:case-k}) implies (\ref{eq:fiberwiseestimates}) is clear because $\langle\Im(z)\rangle\geq 1$. In the case of the estimate (\ref{eq:case-im}), we use the following equivalence of norms (appearing for instance in \cite[Lemma 4.3]{DGH2021} and \cite[Proposition 4.5.1]{CL24}), there exists $C>0$ such that for all $h,h'>0$
\begin{align*}
    \|u_h\|_{\mathcal H_h^N}
\le C \left(\frac{\max\{h,h'\}}{\min\{h,h'\}}\right)^N
\|u_h\|_{\mathcal H_{h'}^N},
\end{align*}
which translates in the case $\Im(z)\geq \langle \K \rangle$ as
\begin{align*}
    \|(-\X_{\K}-z)^{-1}\|_{\mathcal{H}^{N}_{\langle \K \rangle^{-1},\mathrm{hol}}}&\leq \dfrac{\langle \Im(z)\rangle^N}{\langle \K \rangle^N}\times \|(-\X_{\K}-z)^{-1}\|_{\mathcal{H}^{N}_{\langle \Im(z) \rangle^{-1},\mathrm{hol}}}\\
    &\leq C_{\varepsilon}\langle \Im(z)\rangle^{N+1}\\
    &\leq C_{\varepsilon}\langle \K \rangle\langle \Im(z)\rangle^{N+1}
\end{align*}

To prove the resolvent estimate in $\mathcal{B}_0(\varepsilon)'$, an additional difficulty will be that in dimensions greater than $3$, we also need to prove estimates in regions where $\{\Re(z)\leq-1\}$. This will be done by an induction procedure after having applied an unstable derivative to quasimodes (see below for their definition), to allow the use of resolvent estimate in the right adjacent band. We will thus need our resolvent estimate to hold in anisotropic spaces with slightly different order function.
\begin{lemma}\label{lemma:modified_sobolev_estimate}
    Let $N>2$, then Proposition \ref{prop:two_asymptotics_estimates} holds with estimates in the anisotropic space $\mathcal{H}'^N_{h,\mathrm{hol}}$. 
\end{lemma}
\begin{proof}
    For $N>2$, $\mathbf{m}'_N=\mathbf{m}_N-1/N$ defines an order function and a corresponding escape function $\mathbf{G}'$, see Remark \ref{rk:orderfunction}.
    As the proof of this lemma is the same as the proof of Proposition \ref{prop:two_asymptotics_estimates}, except for some minor modifications, we will only explain the slight differences during the course of the proof of Proposition \ref{prop:two_asymptotics_estimates} when needed.
\end{proof}
\subsection{Proof of (\ref{eq:fiberwiseestimates}) for $\Re(z)\in(-1+\varepsilon,0]$}\label{subsec:proof_re_minus_1}\label{subsec:proof_re_minus_1}
\subsubsection{The large $\langle \K\rangle$ case}\label{subsec:proof_re_minus_1_large_k}  In this subsection, we prove the resolvent estimate (\ref{eq:case-k}). We fix $N>1$ to be the anisotropic Sobolev exponent. Assuming (\ref{eq:case-k}) does not hold, we get sequences $(u_j)$, $(\K_j)$ and $(z_j)$ such that, after an eventual extraction 
\[
\begin{cases}
(-i\langle \mathbf{k}_j \rangle^{-1}\X_{\mathbf{k}_j}-i\langle \mathbf{k}_j\rangle^{-1} z_j)u_j = o_{\mathcal{H}^N_{\langle \mathbf{k}_j\rangle^{-1},\mathrm{hol}}}(\langle \mathbf{k}_j \rangle^{-2}), 
& \text{with } \Re(z_j)\to \nu,  \\[6pt]
\|u_j\|_{\mathcal{H}_{\langle \mathbf{k}_j\rangle^{-1},\mathrm{hol}}^{N}} = 1.
\end{cases}
\]
\begin{lemma}\label{lemma:k_goes_infinity}
    The sequence $\langle \K_j\rangle$ goes to $+\infty$.
\end{lemma}

\begin{proof}
    By contradiction, after extraction we can assume that $\mathbf{k}_j:=\K_0\in \Lambda$. By the large tensor power assumption, this would imply that $\Im(z_j)$ is bounded. Thus we can assume that $z_j\to z_0$, and moreover 
    \begin{align*}
        \|(-i\langle \K_j\rangle^{-1}\X_{\mathbf{k}_j}-i\langle \K_j\rangle^{-1} z_j)u_j\|_{\mathcal{H}^N_{\langle \K_j\rangle^{-1},\mathrm{hol}}}\to 0.
    \end{align*}
    This implies that $z_0$ is in the spectrum, but this contradicts the assumption that $(z_j)$ is in the resonance-free region $\mathcal{B}_0'(\varepsilon)=\mathcal{B}_0(\varepsilon)\cap\{\Im(z)\geq \lambda_{\varepsilon}+1\}$. The proof of the existence of the resonance-free region is postponed to the end of Section \ref{section:resolventestimatesfirstband}.
\end{proof}

Here we have by assumption that $\nu\in (-1+\varepsilon,0]$. Without loss of generalities, we may assume that $\Im(z_j)\geq 0$ for all $j$. After an eventual further extraction 
\begin{equation}\label{eq:semiclassical_limits_large_k}
    \lim_{j\to +\infty}\dfrac{\Im(z_j)}{\langle \mathbf{k_j}\rangle}=\eta \in [0,1],\quad \langle \K_j \rangle^{-1}\mathbf{k}_j\to \mathbf{l}\in \mathbb{S}^{d-1},
\end{equation}
To write it in a semiclassical fashion, define $h_j:=\langle \K_j \rangle^{-1}$. Seeing $z_j$ and $\K_j$ as functions of $h_j$, we may write $z(h_j)$ and $\K(h_j)$. We will omit the j-index from now on. Finally, we can write the previous estimate as follows
\begin{align}\label{eq:semiclassical_pb_large_k}
\begin{cases}
(-ih\X_{\mathbf{k}(h)}+h\Im(z(h))-ih\Re(z(h)))u_h = o(h^2), 
\\[6pt]
\|u_j\|_{\mathcal{H}_{\langle \mathbf{k}_j\rangle^{-1},\mathrm{hol}}^{N}} = 1.
\end{cases}
\end{align}
We will write $\mathbf{P}_h(z_h)u_h=o(h^2)$ for the previous line. The principal symbol of $\mathbf{P}_h(z_h)$ is
\begin{align}\label{eq:principal_symbol_largek}
    \sigma _{h}^{\mathrm{BW}}(x,\xi)(\mathbf{P}_h(z_h)):=p(x,\xi)+h\Im(z_h):=\xi(X_F(x))+h\Im(z_h),\quad (x,\xi)\in \mathbb{H}_F^*.
\end{align}
\\
Since $(u_h)$ is uniformly bounded in the semiclassical Sobolev space $H_h^{-N}(F,\El^{\K(h)})$  (replace $N$ by $N+1$ here if dealing with order function $\mathbf{m}'$), it converges after possible extraction to a semiclassical measure $\mu$ on $\mathbb{H}_F^*$ by Proposition \ref{prop:propagation_measure}. Again by this Proposition, we obtain the following properties.
\begin{proposition}\label{prop:semiclassicalmeasure_first_properties}
    The previous Radon measure $\mu$ on $\mathbb{H}_F^*$ verifies the following two properties 
    \begin{itemize}
        \item $\supp(\mu)\subset p^{-1}\bigl(\{-\eta\}\bigr)\cap \mathbb{H}_F^*$,
        \item $(\Phi_t^{\omega_0,p})_*\mu \,=\, e^{2\nu t}\mu, \forall t\in \R$.
    \end{itemize}
   
\end{proposition}
\begin{proof}
    The support property is clear from Proposition \ref{prop:propagation_measure}. 
    The propagation property of Proposition \ref{prop:propagation_measure} gives that $(\Phi_t^{\omega_{\mathbf{l}},p})_*\mu=e^{2\nu t}\mu$.
    Recall that in the case of the generator $X_F$, one has $i_X\mathbf{F}_{\overline{\nabla}}=0$ by Lemma \ref{lemma:curvature}. Finally, Lemma \ref{label:hamiltonianflow} implies $(\Phi_t^{\omega_{\mathfrak{l}},p})=(\Phi_t^{\omega_0,p})$. 
\end{proof}
The following sets play an important role
\begin{equation}\label{eq:subset_TstarF}
\begin{aligned}
\Gamma_{\pm} &:= E_{u/s,F}^* \oplus E_{0,F}^*, \\
\Gamma_{\pm}(\rho)
    &:= \rho\,\pi_{F\to SM}^*\alpha + E_{u/s,F}^*, \\
K_\rho
    &:= \Gamma_+(\rho)\cap\Gamma_-(\rho)
     = \rho\,\pi_{F\to SM}^*\alpha .
\end{aligned}
\end{equation}
$\Gamma_+$ (resp. $\Gamma_-$) is called the unstable (resp. stable) tail, $K_{\rho}$ can be seen as the image of the section $\rho\pi_{F\to SM}^*\alpha$ and is a subset of the \textit{trapped set} $K:=\Gamma_+\cap\Gamma_-=E_0^*$. The following Lemma is important for the final contradiction: we show by radial propagation estimates that $\mu$ must be non-zero on any open neighborhood of $K_{-\eta}$. It is a standard argument that appears for instance in \cite[Lemma 2.6]{Dy16}, \cite[Lemma 3.3]{CG21} or \cite[Lemma 4.6]{DGH2021}, we just need to adapt it in the case of the BW calculus. 
\begin{proposition}\label{prop:support_prop_measure}
    The semiclassical measure $\mu$ is supported in $\Gamma_+(-\eta)$. Moreover for any open neighborhood $V$ of $K_{-\eta}$ we have that $\mu(V)>0$.
\end{proposition}
\begin{proof}
    The key point is that the (radial) propagation estimates boil down to usual propagation estimates: by the fiberwise holomorphic property of quasimodes, we may forget the co-vertical $\mathbb{V}_F^*$ and perform estimates on $\mathbb{H}_F^*$. The sink and source structure being exactly the same in $\mathbb{H}_F^*$ as in the Anosov case (see Section \ref{subsection:anisofourierwise}), a generalization becomes possible.
    
    The Hamiltonian flow associated with each $-ih\X_{\K}$ is simply the usual symplectic lift of $(\Phi^F_t)$ (associated to $X_F$) by the proof of Proposition \ref{prop:semiclassicalmeasure_first_properties}. The set $L:=E_{s,F}^*\cap \partial \overline{\mathbb{H}_F^*}$ is a source for $(\Phi^{\omega_0,p+\eta}_t)=(\Phi^{\omega_0,p}_t)$. The principal symbol of $\frac{1}{h}\Im(P_h(z_h))$ is $-\nu+o(1)$. In view of applying the radial source estimate of Proposition \ref{prop:bwproperties2}, $\nu$ has to be such that $N>-\nu$. This is automatically verified as we assumed $N>1$ and $\nu\in(-1+\varepsilon,0]$. Take $\mathbf{B}\in \Psi^{0}_{h,\mathrm{BW}}(FM)$ such that $L\subset\text{ell}_h^{\mathrm{BW}}(B)$ and $ \mathbf{B}(\mathbf{P}_h(\lambda_h)u_h)\in H^N_h(F,\El^\K)$. Then by the radial source estimate of Proposition \ref{prop:bwproperties2}, there exists $\mathbf{A}_L$ compactly supported near $L$ verifying $\text{WF}_h^{\mathrm{BW}}(\mathbf{1}-\mathbf{A}_L)\cap L=\emptyset$ and such that 
    \begin{align*}
        \|\mathbf{A}_Lu_h\|_{\mathcal{H}^N_{h,\mathrm{hol}}}\leq Ch^{-1}\|\mathbf{B}\mathbf{P}_h(z_h)u_h\|_{\mathcal{H}^N_{h,\mathrm{hol}}}+Ch^N=o(h),
    \end{align*}
    where we used the microlocal equivalence between $\mathcal{H}^N$ and $H^N$ near $E_s^*$ and the fact that $(u_h)$ satisfies (\ref{eq:semiclassical_pb_large_k}). Thus $\supp(\mu)\cap U=\emptyset$ for any $U$ a small open neighborhood of $L$. Finally, any relatively compact set $W$ in $p^{-1}(\{-\eta\})\setminus\Gamma_+(-\eta)$ verifies $\Phi_{-T}^{\omega_0,p}(W)\subset U$ for large enough $T>0$, by the sink-source structure in $\mathbb{H}_F^*$. The first propagation of singularities estimate in Proposition \ref{prop:bwproperties2} leads to the claimed support property.
    
    For the claim on the mass of $\mu$, the argument is again the same as in \cite[Lemma 3.3]{CG21} and \cite[Lemma 4.6]{DGH2021} (except for the fact that we work in higher generalities by using the Borel-Weil calculus). We write $L':=E_{u,F}^*\cap \partial\overline{\mathbb{H}_F^*}$ the sink of the Hamiltonian flow $(\Phi_t^{\omega_0,p})$. The space $\mathcal{H}^N_{h,\mathrm{hol}}$ is microlocally equivalent to $H^{-N}_{h,\mathrm{hol}}$ near $E_{u,F}^*$. Since $\nu>-N$, the radial sink estimate of Proposition \ref{prop:bwproperties2} applies. Consider $\mathbf{B}_1\in \Psi_{h,\mathrm{BW}}$ with $L'\subset \Ell_h^{\mathrm{BW}}(\mathbf{B}_1)$. There exists $\mathbf{A}_{L'}, \mathbf{B}_{L'}$ such that $L'\subset \Ell_h^{\mathrm{BW}}(\mathbf{A}_{L'})$, $\WF_h^{\mathrm{BW}}(\mathbf{B}_{L'})\subset \Ell_h^{\mathrm{BW}}(\mathbf{A}_{L'})\setminus L'$ and 
    \begin{align}\label{eq:proof_sink_estimate}
        \|\mathbf{A}_{L'}u_h\|_{\mathcal{H}^N_{h,\mathrm{hol}}}&\leq Ch^{-1}\|\mathbf{B}_1\mathbf{P}_h(z_h)u_h\|_{\mathcal{H}^N_{h,\mathrm{hol}}}+\|\mathbf{B}_{L'}u_h\|_{\mathcal{H}^N_{h,\mathrm{hol}}}+Ch^N\\
        &\leq \|\mathbf{B}_{L'}u_h\|_{\mathcal{H}^N_{h,\mathrm{hol}}} + o(h). \notag
    \end{align}
    Let $V$ be a relatively compact neighborhood of $K_{-\eta}$. Assume for the sake of contradiction that there exists $\mathbf{A}_{\mathrm{comp}}\in \Psi_{h,\mathrm{BW}}^{\mathrm{comp}}$ verifying $\WF_h(\mathbf{1}-\mathbf{A}_{\mathrm{comp}})=\emptyset$ on $\overline{V}$ such that $\mathbf{A}_{K_{-\eta}}u_h=o(1)$ (\textit{i.e} $\mu$ does not charge a small open neighborhood of $K_{-\eta}$). Further define $\mathbf{A}_0\in \Psi^0_{h,\mathrm{BW}}$ such that $\WF^{\text{BW}}_h(\mathbf{A}_0)\subset \overline{\mathbb{H}_F^*}\setminus\{\xi(X)=-\eta\}$ and  $\WF^{\text{BW}}_h(1-\mathbf{A}_0)\subset\{|\xi(X)+\eta|<\varepsilon_0\}$, for some $\varepsilon_0>0$ . By the support property of $\mu$, it is clear that $\mathbf{A}_0u_h=o(1)$ in anisotropic norm. Since we would like to control the mass of $\mathbf{A}_{L'}u_h$, we have to control the one of $\mathbf{B}_{L'}u_h$ by (\ref{eq:proof_sink_estimate}): we can control the latter by the norms of $\|\mathbf{A}_{\bullet}u_h\|_{\mathcal{H}^N_{h,\mathrm{hol}}}$ with $\bullet\in \{\mathrm{comp},L,0\}$ as follows.  By the dynamics of the Hamiltonian flow on $\mathbb{H}_F^*$, we see that for all $(x,\xi)\in \WF^{\text{BW}}_h(\mathbf{B}_{L'})$, there exists a time $T(x,\xi)>0$ such that $\Phi_{-T(x,\xi)}^{\omega_0}(x,\xi)\in \Ell_h^{\mathrm{BW}}(\mathbf{A}_{\bullet})$, where $\bullet \in \{L,0,K_{-\eta}\}$.
    We may thus apply the propagation of singularity estimate of Proposition \ref{prop:bwproperties2}. Rigorously, we need to apply such estimates in $L^2(F,\El^\K)$ with the operators $\mathbf{A}_{\bullet}^{(N)}:=(\mathbf{A}^N_{h,\K(h)})^{-1}\mathbf{A}_{\bullet}\mathbf{A}^N_{h,\K(h)}$, $\mathbf{B}_{L'}^{(N)}$ and $\mathbf{P}^{(N)}_h(z_h)$, but the conjugation does not change the preceding propagation properties. Keeping this in mind one obtain the estimate
    \begin{align*}
\|\mathbf{B}_{L'} u_h\|_{\mathcal{H}_{h,\mathrm{hol}}^{N}}
&\lesssim
 \|\mathbf{A}_{K_{-\eta}} u_h\|_{\mathcal{H}_{h,\mathrm{hol}}^{N}}
+  \|\mathbf{A}_L u_h\|_{\mathcal{H}_{h,\mathrm{hol}}^{N}}
+  \|\mathbf{A}_0 u_h\|_{\mathcal{H}_{h,\mathrm{hol}}^{N}}
+  h^{-1}\|\mathbf{P}_h(\lambda) u_h\|_{\mathcal{H}_{h,\mathrm{hol}}^{N}}
\\
&=o(1).
\end{align*}
Combining all the previous estimates :
\begin{align*}
   \|u_h\|_{\mathcal{H}_{h,\mathrm{hol}}^{N}}\leq \| (1-\mathbf{A}_{L}-\mathbf{A}_{K_{-\eta}}-\mathbf{A}_{0}-\mathbf{A}_{L'})u_h\|_{\mathcal{H}_{h,\mathrm{hol}}^{N}}+\sum_{\bullet\in\{L,L',0,K_{-\eta}\}} \|\mathbf{A}_{\bullet}u_h\|_{\mathcal{H}_{h,\mathrm{hol}}^{N}}=o(1),
\end{align*}
which contradicts (\ref{eq:semiclassical_pb_large_k}).

\end{proof}
\begin{remark}\label{rk:modified_N}
    In the case of quasimodes in $\mathcal{H}'^{N}$, one needs to replace $\mathbf{A}_{h,\K(h)}^N$ by $\mathbf{A}_{h,\K(h)}'^N$ and that the microlocal regularity near $E_{s,F}^*$ (resp. $E_{u,F}^*$) is $H^{N-1}$ (resp. $H^{-N-1}$). To still have the threshold condition (concerning $\nu$) for the radial estimates in Proposition \ref{prop:bwproperties2}, it is thus needed that $N>2$ for the estimates. 
\end{remark}

\textbf{Some important unstable vector fields.} As in \cite[Lemma 4.8]{DGH2021} (although our approach is slightly different) we aim at constructing  particular (local) unstable vector fields on $F$ whose Hamiltonian vector fields with respect to $\omega_{\mathbf{l}}$ exhibits good \textit{transversality properties} to $K_{-\eta}$. Before doing so, we draw the attention of the reader to the following key fact. That $\eta$ could be zero is allowed in this high tensor power regime. In that case, the semiclassical measure concentrates on $\{\xi(X(x))=0\}$. In that case, will aim at constructing (horizontal) unstable vector fields $U\in C^{\infty}(F,E^u_F)$ whose Hamiltonian vector fields are, on their support, transversal to the submanifold $K_0\hookrightarrow \mathbb{H}_F^*$. Such vector fields must satisfy (see Lemma \ref{lemma:transversality_horocyclic}) $\omega_{\mathbf{l}}(H^{\omega_\mathbf{l}}_{\xi(U(x))},H^{\omega_\mathbf{l}}_{\xi(S(x))})\neq0$ for $S$ some stable vector field on $F$. By (\ref{eq:twisted_symplectic}), this quantity reads as 
    \begin{align*}
        \omega_{\mathbf{l}}(H^{\omega_\mathbf{l}}_{\xi(U(x))},H^{\omega_\mathbf{l}}_{\xi(S(x))})=\xi([U,S](x))+i\mathbf{l}\cdot \mathbf{F}_{\overline{\nabla}}(U(x),S(x)).
    \end{align*}
    From $\eta=0$, we deduce that the first term must be zero on $K_0$ (the zero section of $T^*F$) \textit{i.e.} the Hamiltonian vector field of $U$ with respect to $\omega_0$ is tangent to $K_0$. So, in order to have the transversality property, it is of importance that $h\K(h)\to\mathbf{l}\neq 0$. 
    
    We will see that in the case where $\langle \Im(z)\rangle$ is large (see (\ref{eq:case-im})), $\mathbf{l}$ can be zero but this will be compensated by $\eta\neq 0$. This is where the contact property of the geodesic flow comes in. This mechanism already appears in \cite{DGH2021} and is clearly still of importance for us.
    
In the definition of the parameter $\mathbf{l}=(l_1,\dotsc,l_d)$ in (\ref{eq:semiclassical_limits_large_k}), although $\mathbf{l}\neq 0$ by definition of the large $\langle \K \rangle$ regime, some of its components may vanish. To take that into account, we will assume without losing generalities that its first component verifies $l_1\neq 0$: that will explain why we will mostly care about the curvature of $L_1$, namely $\mathbf{F}_{\overline{\nabla}_1}$.
We recall that a basic vector field $Y$ on $F$ can be written as a horizontal lift $Y:={Y}_1^{\mathbb{H}_F}$ (with $Y_1$ a vector field on $SM$).
We make the following local construction.

\begin{lemma}\label{lemma:construction_particular_vf}
    Let $hT\in F$ and $\epsilon>0$, there exists a small enough open neighborhood $W_{\epsilon}\subset F$ of $hT$, a local basis of $E^u_F$ consisting of basic unstable vector fields $(\mathbf{U}^-_j)_{1\leq j\leq n-1}$ and basic stable vector field $(\mathbf{U}_{j}^+)_{j\in\{1,2\}}$ in $E^s_F$ verifying on $W_{\epsilon}$ for $i\in\llbracket 2,n-1\rrbracket$
\begin{align}
    &\lvert i^{-1}\mathbf{F}_{\overline{\nabla}_1}(\mathbf{U}_1^+,\mathbf{U}_j^-)+2\rvert<\epsilon, \quad \lvert\pi_{F\to SM}^*\alpha([\mathbf{U}_1^+,\mathbf{U}_j^-])\rvert <\epsilon,\label{eq:curvature_final_horo}\\
        &\lvert i^{-1}\mathbf{F}_{\overline{\nabla}_1}(\mathbf{U}_2^+,\mathbf{U}_1^-)-2\rvert<\epsilon, \quad \lvert\pi_{F\to SM}^*\alpha([\mathbf{U}_2^+,\mathbf{U}_1^-])\rvert <\epsilon.\notag
\end{align}
Moreover if the rank $d$ is greater than $2$ (that is $n\geq 5$), we have for $k\in \llbracket 2,d\rrbracket$ and $j\in \llbracket 2,n-1 \rrbracket$
\begin{align}\label{eq:curvature_rank_bigger_2}
    \lvert i^{-1}\mathbf{F}_{\overline{\nabla}_k}(\mathbf{U}_1^+,\mathbf{U}_j^-)+2\rvert<\epsilon,\quad \lvert i^{-1}\mathbf{F}_{\overline{\nabla}_k}(\mathbf{U}_2^+,\mathbf{U}_1^-)-2\rvert<\epsilon,
\end{align}

\end{lemma}

\begin{proof}
    Let us first prove the lemma for $n=4$ (in higher-rank cases, the procedure is similar, see the end of the proof). In that case the fibers of $FM\to SM$ are isometric to $SO(3)$, whose rank is $1$. Recall from (\ref{eq:chosentorus}) that we choose the maximal torus generated by $R_{2,3}$. We will write $\tilde{\Theta}_{\mathrm{dyn}}:T(FM)\to i\mathbb{R}$ the connection $1$-form corresponding to the $\mathbb{S}^1$-principal bundle $FM\to F$. We recall that the horizontal space of this connection is $\mathbb{H}_{FM\to F}:=\mathbb{H}_{FM\to SM}\oplus \Re(\mathfrak{n}^+_{SO(3)}\oplus \mathfrak{n}^-_{SO(3)})$. 
    
    From (\ref{eq:curvaturelinelift}), for any $U\in C^{\infty}(F,E^{u}_F)$ and $S\in C^{\infty}(F,E^{s}_F)$, we have 
    \begin{align}\label{eq:curvature_lift_f_to_fm}
        i^{-1}\mathbf{F}_{\overline{\nabla}_1}(U,S)&=d\tilde{\Theta}_{\mathrm{dyn}}(U^{\mathbb{H}_{FM\to F}},S^{\mathbb{H}_{FM\to F}})\\
        &=-\tilde{\Theta}_{\mathrm{dyn}}([U^{\mathbb{H}_{FM\to F}},S^{\mathbb{H}_{FM\to F}}]),\notag
    \end{align}
    where $U^{\mathbb{H}_{FM\to F}}, S^{\mathbb{H}_{FM\to F}}$ are the horizontal lifts of $U$ and $S$ with respect to $\tilde{\Theta}_{\mathrm{dyn}}$. The left-invariant vector fields $U_1^{\pm}$ and $U_2^{\pm}$ from (\ref{eq:commutation_algebraic}) verify that $\tilde{\Theta}_{\mathrm{dyn}}([U^+_1,U_2^-])=2$ and $\tilde{\Theta}_{\mathrm{dyn}}([U^+_2,U_1^-])=-2$. Pick a local section $s:W_F(\subset F)\to FM$ for $W_F$ a contractible (and to be chosen sufficiently small) neighborhood of $hT$. The vector fields $(dR_t\bigl(U_{1,2}^{\pm}(s(w))\bigr)_{(w,t)\in W_F\times \mathbb{S}^1}$ on $\pi^{-1}_{FM\to F}(W_F)$ descend to non-vanishing vector fields $\tilde{\mathbf{U}}_{1,2}^{\pm}$ on $W_F\subset F$. Note however that $\tilde{\mathbf{U}}_{1,2}^{\pm}$ have no reason to be basic vector fields on $F\to SM$, this will be fixed below. By the right $T$-equivariance property of $\tilde{\Theta}$ we further obtain
    \begin{align*}
        i^{-1}\mathbf{F}_{\overline{\nabla}_1}(\tilde{\mathbf{{U}}}_{1}^{+},\tilde{\mathbf{{U}}}_{2}^{-})&=-2 \,\text{ on } W_F,\\
        i^{-1}\mathbf{F}_{\overline{\nabla}_1}(\tilde{\mathbf{{U}}}_{2}^{+},\tilde{\mathbf{{U}}}_{1}^{-})&=2 \,\text{ on } W_F,
    \end{align*}
    using (\ref{eq:curvature_lift_f_to_fm}). Let us modify $\tilde{\mathbf{{U}}}_{1,2}^{\pm}$ so that we work with basic vector fields on $F\to SM$. Remark that $\tilde{\mathbf{U}}^{-}_{1,2}$ (resp.$\tilde{\mathbf{U}}^{+}_{1,2}$) is a local section of $E^u_F$ (resp. $E^s_F$) over $W_F$. To see this, it suffices to show that $D\pi_{FM\to F}:E^u_{FM}\to E^u_F$. The flows $\Phi_t^F$ and $\Phi_t$ are isometric extensions of $\varphi_t$ (on $SM$) on $F$ and $FM$ respectively. Hence a consequence of \cite[Lemma 12.1.4]{Lefeuvre2026} is that
    \begin{align*}
        D\pi_{FM\to SM}&:E^{u,s}_{FM}\to E^{u,s}_{SM},\R X^{FM}\to \R X^F,\\
        D\pi_{F\to SM}&:E^{u,s}_{F}\to E^{u,s}_{SM},\R X^{F}\to \R X^{SM},
    \end{align*}
    are isomorphisms. By $\pi_{FM\to SM}=\pi_{FM\to F}\circ \pi_{F\to SM}$ and the splittings of $T(FM)$ and $TF$, we indeed obtain $D\pi_{FM\to F}:E^u_{FM}\to E^u_F$.

    Pick a further local section $s':\pi_{F\to SM}(W_F)\subset SM\to W_F$ and define the vector fields $\mathbf{U}_{1,2}^{\pm}$ on $F$ to be the horizontal lifts of the vector fields on $\pi_{F\to SM}(W_F)$ defined by $D\pi_{F}\bigl(\tilde{\mathbf{U}}^{\pm}_{1,2}(s'(x,v)\bigr)$, for $(x,v)\in\pi_{F\to SM}(W_F)$. We have by construction, for $(x,v)$ as before
    \begin{align}\label{eq:difference_vertical}
        \mathbf{U}_{1,2}^{\pm}(s'(x,v))=\tilde{\mathbf{U}}^{\pm}_{1,2}(s'(x,v)),
    \end{align}
    as we justified that $\tilde{\mathbf{U}}^{\pm}_{1,2}\in C^{\infty}(W_F,E^{u,s}_F)$, hence those sections have no vertical component.
    Thus using (\ref{eq:difference_vertical}) 
    \begin{align*}
        &i^{-1}\mathbf{F}_{\overline{\nabla}_1}(\mathbf{U}_{1}^+,\mathbf{U}_{2}^{-})(s'(x,v))=-2,\\
        &i^{-1}\mathbf{F}_{\overline{\nabla}_1}(\mathbf{U}_{2}^+,\mathbf{U}_{1}^{-})(s'(x,v))=2.
    \end{align*}
    By continuity of the curvature, for any $\epsilon>0$, there exists a small enough open set $W_{\epsilon}\subset W_F$, on which 
    \begin{align}\label{eq:final_proof_curvature_horo}
        &\lvert i^{-1}\mathbf{F}_{\overline{\nabla}_1}(\mathbf{U}_{1}^+,\mathbf{U}_{2}^{-})(s'(w))+2\rvert<\epsilon,\\
        &\lvert i^{-1}\mathbf{F}_{\overline{\nabla}_1}(\mathbf{U}_{2}^+,\mathbf{U}_{1}^{-})(s'(w))-2\rvert<\epsilon. \notag
    \end{align}
  In dimension $4$ the unstable bundle are of dimension $3$, we also have to take into account the vector fields $U_3^{-}$ on $FM$. This vector field descend to $F$ because $[R_{2,3},U_3^{-}]=0$, see (\ref{eq:commutation_algebraic}).
  Remark that $\iota_{U_3^{-}}F_{\overline{\nabla}_1}=0$ using the relation $[U_i^+,U_3^-]=2R_{i+1,4}$ for $i\leq 2$ and $[U_3^{+},U_3^{-}]=2\mathfrak{X}$  together with (\ref{eq:curvaturelinelift}). We set $\tilde{\mathbf{U}}_3^{-}:=U_3^{-}+\tilde{\mathbf{U}}_2^{-}$ and can perform the procedure as above (by considering the section $s':\pi_{F\to SM}(W_F)\to F$) to recover the horizontal vector field $\mathbf{U}_3^{\pm}$. We obtain, using $i_{U_3^{-}}F_{\overline{\nabla}_1}=0$ and the previous relations (\ref{eq:final_proof_curvature_horo}), on $W_{\epsilon}$
\begin{align*}
    |i^{-1}\mathbf{F}_{\overline{\nabla}_1}(\mathbf{U}^+_{1},\mathbf{U}^-_{3})+2|<\epsilon,
\end{align*}
     proving the claim concerning the curvature relations.
    We are left with showing that $\lvert \pi_{F\to SM}^*\alpha([\mathbf{U}_1^+,\mathbf{U}^-_{i}])\rvert<\epsilon$ if $i\geq 2$ and $\lvert \pi_{F\to SM}^*\alpha([\mathbf{U}_2^+,\mathbf{U}^-_{1}])\rvert<\epsilon$ on $W_{\epsilon}$. Notice that the relations in (\ref{eq:commutation_algebraic}) give $\pi^*_{FM\to SM}\alpha\,([U_i^+,U_j^-])=0$ for $i\neq j$ (recall that $\alpha$ vanishes on $E^u_{SM}\oplus E^s_{SM}$). Using the notation of the previous construction, we have on $W_F$ for $i\in\{2,3\}$
    \begin{align*}
        \pi_{F\to SM}^*\alpha([\tilde{\mathbf{U}}^{+}_{1},\tilde{\mathbf{U}}_{j}^-](w))&=\pi_{FM\to SM}^*\alpha([U_1^+,{U}_{j}^-](s(w)))\\
        &=\pi_{FM\to SM}^*\alpha(R_{2,j+1})\\
        &=0
    \end{align*}
    Using the continuity of $\alpha$, we obtain on $W_{\epsilon}$ (up to eventually reducing it's size) that $\lvert \pi_{F\to SM}^*\alpha([\mathbf{U}_1^+,\mathbf{U}^-_{i}])\rvert<\epsilon$. The same argument work identically when considering $\pi_{F\to SM}^*\alpha([\mathbf{U}_2^+,\mathbf{U}^-_{1}])$, giving the desired claim for the $n=4$.

    To finish, let us expand on the higher rank case. Recall that the (lie algebra of the) chosen maximal torus is spanned by $(R_{i,i+1})_{2\leq i\leq n-1\text{ even}}$. As done in the $n=4$ case, we construct vector fields $(\tilde{\mathbf{U}}_j^-)_{1\leq j \leq n-1}$ (resp. $(\hat{\mathbf{U}}_j^+)_{1\leq j \leq 2}$) on an open neighborhood $W_F$ of $hT$ from the left-invariant vector fields $(U_j^-)$ (resp. $({U}_j^+)_{1\leq j \leq 2}$). For $j\geq 3$, set $\hat{\mathbf{U}}_j^-=\tilde{\mathbf{U}}_1^-+\tilde{\mathbf{U}}_j$. Using relation (\ref{eq:curvaturelinelift}), we compute for $w\in W_F$ and $k\in\llbracket 1,d\rrbracket$
    \begin{align*}
        \mathbf{F}_{\overline{\nabla}_k}(\tilde{\mathbf{U}}_1^+,\tilde{\mathbf{U}}_{2}^{-}+\tilde{\mathbf{U}}_j^-)(w)
        &=\mathbf{F}_{\overline{\nabla}_k}(\tilde{\mathbf{U}}_1^+,\tilde{\mathbf{U}}_{2}^{-})(w)
        + \mathbf{F}_{\overline{\nabla}_k}(\tilde{\mathbf{U}}_1^+,\tilde{\mathbf{U}}_{j}^{-})(w)\\
        &=
        \mathbf{F}_{\overline{\nabla}_k}(\tilde{\mathbf{U}}_1^+,\tilde{\mathbf{U}}_{2}^{-})(w)
        +
        \omega_{k}
        \Bigl(
        \Pi_{\mathfrak t}\Theta_{\mathrm{dyn}}
        \bigl([U_1^+,U_j^-](s(w))\bigr)
        \Bigr)
        \\
        &=
        \mathbf{F}_{\overline{\nabla}_k}(\tilde{\mathbf{U}}_1^+,\tilde{\mathbf{U}}_{2}^{-})(s'(x,v))
        +
        \omega_{k}
        \Bigl(
        \Pi_{\mathfrak t}\Theta_{\mathrm{dyn}}
        \bigl(R_{2,j+1}(s(w))\bigr)
        \Bigr)
        \\
        &=
        \mathbf{F}_{\overline{\nabla}_k}(\tilde{\mathbf{U}}_1^+,\tilde{\mathbf{U}}_{2}^{-})(w)
\end{align*}
    where we used in the last line that $R_{2,j+1}\in \ker(\Pi_\mathfrak{t})$ since $j\geq 3$. Finally
    \begin{align*}
        \mathbf{F}_{\overline{\nabla}_k}(\tilde{\mathbf{U}}_1^+,\tilde{\mathbf{U}}_{2}^{-})(w)&=\omega_{k}
        \Bigl(
        \Pi_{\mathfrak t}\Theta_{\mathrm{dyn}}
        \bigl(R_{2,3}(s(w))\bigr)
        \Bigr)\\
        &=-2\omega_k(R_{2,3})\\
        &=-2i\text{ by (\ref{eq:fundament_weight_even}) or \ref{eq:fundament_weight_odd}) }.
    \end{align*}
    Using a continuity argument as in the $n=4$ case, we further modify those vector fields to horizontal vector fields $(\mathbf{U}_j^-)_{1\leq j \leq n-1}$ and $(\mathbf{U}_j^+)_{1\leq j \leq 2}$ . The justification of the relations involving the $1$-form $\pi_{F\to SM}^*\alpha$ is identical to the rank 1 case.
  
\end{proof}

\begin{remark}
     We could be tempted to use the vector fields $\tilde{\mathbf{U}}_j^{-}:W_{\epsilon}\to TF$ from the previous proof to obtain our transversal vector fields. However, by doing so, we cannot guarantee that their Hamiltonian flow preserves the (co)horizontal space $\mathbb{H}_F^*$: we cannot apply the propagation estimate of Proposition \ref{prop:propagation_measure}.
\end{remark}

Inspecting the previous proof and using the compactness of $F$, we may construct a finite cover $(W_{\epsilon,k})$ of $F$ with corresponding vector fields $\mathbf{U}_{i}^{\pm,k}$ satisfying the relations of Lemma \ref{lemma:construction_particular_vf} on $W_{\epsilon,k}$. For simplicity, we will omit most of the time the $\epsilon$ index and write $W_k$ the previous open sets.

\textbf{Horocyclic invariance.} In this paragraph, we prove the following invariance property.
\begin{proposition}\label{prop:horo_invariance_quasimode_re_larger_minus1}
    Let $U\in C^{\infty}(SM,E^u)$, then $(u_h)$ from (\ref{eq:semiclassical_pb_large_k}) satisfies 
    \begin{align*}
-ih\,\nabla_{U^{\mathbb{H}_F}}^{\mathbf{k}(h)}u_h=o_{{\mathcal{H}'}_h^N(F,\El^{\K(h)})}(h).
    \end{align*}
\end{proposition}
\begin{remark}
    In order to prove Lemma \ref{lemma:modified_sobolev_estimate}, one also has to take in consideration the case of quasimodes in $\mathcal{H}'^N_{h,\mathrm{hol}}$. When applying the horocyclic derivative, their anisotropic regularity is $ \mathcal{H}^{[-N-2,N-2]}_h$: the order function is $\tilde{m}=m-2/N$. It is thus needed that $N>2$ to apply Proposition \ref{prop:support_prop_measure}. We refer the reader to Section \ref{Section:resolvent_estim_re_bigger_1} for further details.
\end{remark}
In view of the previous remark, we perform the proof in the case of quasimodes in $\mathcal{H}^N$. The case of quasimodes in $\mathcal{H}'$ is done exactly in the same fashion.
\begin{proof}
    The proof is an adaptation of \cite[Lemma 4.5.4]{CL24}. We may complexify $\mathbb{H}_F^*$ (see Proposition \ref{prop:holpreserved}) to $\mathbb{H}_F^*(\mathbb{C})$ and $(E^u_F)^*$ to $(E^u_F)^*(\mathbb{C})$. By restricting the partial connection $\nabla^{\K(h)}$ to elements of $E^u_F(\mathbb{C})$, we obtain the operator \begin{align}\label{eq:horocyclic_invariance_def_op}
        \nabla^{\K(h),u}:C^{\infty}_{\mathrm{hol}}(F,\El^{\K(h)})\to C^{\infty}_{\mathrm{hol}}(F,\El^{\K(h)}\otimes E_{s,F}^*(\mathbb{C})),
    \end{align}
    where $E_{s,F}^*(\mathbb{C})\hookrightarrow \mathbb{H}_F^*(\mathbb{C})$ is identified with $(E^u_F)^*(\mathbb{C})$. From here, we will omit to mention the complexification for notational convenience. Similarly, we can construct the operators $\X_\K^u$ (resp. $\mathbf{P}_h^u$) from the previous operator $\X_\K$ (resp. $\mathbf{P}_h$) by twisting $\nabla^\K_X$ with the Lie derivative $\mathcal{L}_X$ on $E_s^*$. By the fact that $\iota_X\mathbf{F}_{\nabla}(=\iota_X\mathbf{F}_{\overline{\nabla}})=0$, one can easily see that
    \begin{align}\label{eq:commutation_algebraic_unstable}
        \mathbf{P}_h^u \,h\nabla^{\mathbf{k}(h),u}=h\nabla^{\mathbf{k}(h),u}\mathbf{P}_h.
    \end{align}
The interested reader may refer to the proof of \cite[Lemma 4.5.4]{CL24} for the justification of this relation. The operator $h\nabla^{\mathbf{k}(h),u}$ is continuous from $\mathcal{H}^N_{h,\mathrm{hol}}(F,\El^{\K(h)})$ to $\mathcal{H}'^N_{h,\mathrm{hol}}(F,\El^{\K(h)}\otimes E_{s,F}^*(\mathbb{C}))$. Indeed, the algebra property from Proposition \ref{prop:bwproperties1} implies that $\mathbf{A}_{h,\K(h)}'(E_{s,F}^*) (h\nabla^{\mathbf{k}(h),u})(\mathbf{A}^N_{h,\K(h)})^{-1}$ is of order $0$ by the algebra property, thus it is bounded from $L_{(\mathrm{hol})}^2(F,\El^\K)$ to $L^2_{(\mathrm{hol})}(F,\El^\K\otimes E_{s,F}^*)$.

The continuity gives $h\nabla^{\mathbf{k}(h),u}\mathbf{P}_hu_h=o_{\mathcal{H}'_h(F,\El^{\K(h)}\otimes E_{s,F}^*)}(h^2)$ by (\ref{eq:semiclassical_pb_large_k}). In view of (\ref{eq:commutation_algebraic_unstable}), it thus suffices to prove that $\mathbf{P}_h^u$ is invertible with anisotropic norm bounded by $h^{-1}$ to conclude.

The propagator $e^{-t{\mathcal{L}_X}_{|E_{s,F}^*}}$ acts on $C^{\infty}(F,E_{s,F}^*)$ as $d\varphi_{t}$ on $(E^{u}_F)^*$ (and this action is defined by duality). By the sink and source structure of $p$ for $\omega_0$, we have $\|e^{-t{\mathcal{L}_X}_{|E_{s,F}^*}}\|_{L^2}\lesssim e^{-t}$. The following integral thus makes sense in $L^2$ for $\{\Re(z)>-1\}$ 
\begin{align}\label{eq:resolvent_formula_L2}
    (-\X^{\K(h),u}-z)^{-1}=\int^{+\infty}_0 e^{-zt}e^{-t\X^{u}_{\K(h)}}\, dt\,.
\end{align}
To evaluate the norm of the propagator in $\mathcal{H}'^N(F,\El^{\K(h)}\otimes E_{s,F}^*)$, we write 
\begin{align*}
    &\|e^{-t\X^{u}_{\K(h)}}e^{t\X^{u}_{\K(h)}}(\mathbf{A}_{h,\mathbf{k}(h)}'^{N}(E_{s,F}^*))^{-1}\,e^{-t\X^{u}_{\K(h)}}(\mathbf{A}_{h,\mathbf{k}(h)}'^{N}(E_{s,F}^*))^{-1}\|_{L^2}\\ &\leq e^{-t}\|e^{t\X^{u}_{\K(h)}}(\mathbf{A}_{h,\mathbf{k}(h)}'^{N}(E_{s,F}^*))^{-1}\,e^{-t\X^{u}_{\K(h)}}(\mathbf{A}_{h,\mathbf{k}(h)}'^{N}(E_{s,F}^*))^{-1}\|_{L^2}.
\end{align*}
We work here with principally diagonal operators, the usual semiclassical results exposed in Section \ref{section:BWcalculus} apply. By the Egorov theorem in the BW calculus, the principal symbol of the operator conjugated by the propagator in the previous line is $\exp{(\mathbf{G}'\circ \Phi^{\omega_0,p}_t-\mathbf{G}')}\mathrm{Id}_{E_{s,F}^*}$, see Section \ref{section:anisotropicspace} for the notations. By construction of the escape functions, we have $\exp(\mathbf{G}'\circ \Phi^{\omega_0,p}_t-\mathbf{G}')\leq 1$. The Calderon-Vaillancourt theorem for small enough $h>0$ and $t\in [0,1]$ implies that there exists $C>0$ such that 
\begin{align*}
    \|e^{-t\X^{u}_{\K(h)}}e^{t\X^{u}_{\K(h)}}(\mathbf{A}_{h,\mathbf{k}(h)}'^{N}(E_{s,F}^*))^{-1}\,e^{-t\X^{u}_{\K(h)}}(\mathbf{A}_{h,\mathbf{k}(h)}'^{N}(E_{s,F}^*))^{-1}\|_{L^2}&\leq e^{-t} \,(1+Ch),
\end{align*}
The semigroup property gives the following bound for all $t\in \mathbb{R}_+$ 
\begin{align*}
    \|e^{-t\X^{\K(h),u}}e^{t\X^{\K(h),u}}\mathbf{A}_{h,\K(h)}'^{N,E_{s,F}^*}e^{-t\X^{\K(h),u}}(\mathbf{A}_{h,\K(h)}'^{N,E_{s,F}^*})^{-1}\|_{L^2}&\leq e^{-t} \,(1+Ch)^{t+1}.
\end{align*}
We inject this bound in (\ref{eq:resolvent_formula_L2}) 
\begin{align*}
    \|(-\X^{\K(h),u}-z)^{-1}\|_{\mathcal{H}'^N}\leq (1+Ch)\int^{+\infty}_0 e^{-(1+\Re(z)-\log(1+Ch))t}\, dt.
\end{align*}
Since $\Re(z)>-1+\varepsilon$, taking $h$ small enough gives that $\|\mathbf{P}_h^u\|\lesssim h^{-1}$. We deduce that $h\nabla^{\mathbf{k}(h),u}u_h=o_{\mathcal{H}'^N_h(E_{s,F}^*)}(h)$. For $U$ a section of $T(SM)$, its horizontal lift $U^{\mathbb{H}_F}$ is holomorphic. In particular, $\iota_{U^{\mathbb{H}_F}}:\mathcal{H}'^N_{h,\mathrm{hol}}(F,\El^{\K(h)}\otimes E_{s,F}^*(\mathbb{C}))\to \mathcal{H}'^N_{h,\mathrm{hol}}(F,\El^{\K(h)})$ is well-defined and continuous. We obtain $-ih\,\nabla_{U^{\mathbb{H}_F}}^{\mathbf{k}(h)}u_h=o_{{\mathcal{H}'}_h^N(F,\El^{\K(h)})}(h)$, the desired claim.
\end{proof}
The main consequence of the horocyclic invariance is the support and propagation properties that it implies on the semiclassical measure $\mu$.
\begin{definition}\label{def:hamitlonian_horocylic}
    On each open set $W_k$ we can consider the Hamiltonian function $(x,\xi)\mapsto \varphi^{\pm,k}_j(x,\xi):=\xi(\mathbf{U}^{\pm,k}_j)(x)$ associated with the vector field $\mathbf{U}^{\pm,k}_j$ for all $1\leq j \leq n-1$. It is defined on $T^*W_{k}$. The corresponding Hamiltonian vector fields will be written $H^{\omega_{\mathbf{l}}}_{\varphi^{\pm,k}_j}$.\\
    Also, as in the three dimensional case \cite[Lemma 4.9]{DGH2021} (in dimension $3$, recall that $F=SM$ and $\mathbb{H}_F^*=T^*(SM)$), we define $\mathbf{G}_j^{\pm,k}:=\mathrm{div}(\mathbf{U}_j^{\pm,k})\circ \pi_{\mathbb{H}_F^*\to F}\in C^{\infty}(T^*W_{\epsilon,k})$.
    
\end{definition}
\begin{proposition}\label{prop:propagation_measure_horocyclic}
    For any $a\in C^{\infty}_c(T^*{W}_{k}\cap \mathbb{H}_F^*)$ and $1\leq j \leq n-1$, one has the propagation property
    \begin{align}\label{eq:propagation_measure_horocyclic}
        \int_{\mathbb{H}_F^*} (H^{\omega_{\mathbf{l}}}_{\varphi^{-,k}_j}a+\mathbf{G}_j^{-,k}a)d\mu=0.
    \end{align}
\end{proposition}
\begin{proof}
    We can extend $\mathbf{U}^{-,k}_j$  to a horizontal vector field $\overline{\mathbf{U}}^{-,k}_j$ which is smooth on $F$. Let us write $\overline{\varphi}_j^{-,k}$ the associated Hamiltonian function. By Proposition \ref{prop:horo_invariance_quasimode_re_larger_minus1} 
    \begin{align*}
        \|h\nabla^{\K(h)}_{\overline{\mathbf{U}}^{-,k}_j}u_h\|_{\mathcal{H}'^N_{h,\mathrm{hol}}}=o(h).
    \end{align*}
    For all $a\in C^{\infty}_c(\mathbb{H}_F^*)$, we combine previous estimate combined with Proposition \ref{prop:propagation_measure}
    \begin{align*}
        \int _{\mathbb{H}_F^*}\,(H^{\omega_{\mathfrak{l}}}_{\overline{\varphi}^{-,k}_j}a+\mathbf{G}_j^{\pm,k}a)d\mu=0.
    \end{align*}
    Finally, taking $a\in C^{\infty}_c(\mathbb{H}_F^*\cap T^*W_{k})$ leads to (\ref{eq:propagation_measure_horocyclic}).
    \end{proof}
    
    \textbf{Deriving the contradiction  } We now gather all the previous results to derive the final contradiction. The overall strategy uses crucially the transversality of the unstable Hamiltonian vector fields to the trapped set (see \cite{Dy16}, \cite{CG21} or \cite{DGH2021} where this strategy is used).

    \begin{lemma}\label{lemma:transversality_horocyclic}
        Let $\epsilon>0$ be small enough. For every $j\in \llbracket 1,n-1\rrbracket$, the Hamiltonian vector field $H^{\omega_\mathbf{l}}_{\varphi^{-,k}_j}$ is transverse to $T_{\kappa}\Gamma_-$ for all $\kappa\in K_{-\eta}\cap T^*W_{\epsilon,k}$. More precisely, for any $j$ there exists $j'\in\llbracket 1,n-1\rrbracket$ such that
        \begin{align}\label{eq:transversality_horo}
            d\varphi^{+,k}_{j'}(H^{\omega_{\mathbf{l}}}_{\varphi^{-,k}_j}(\kappa))\neq0\, \text{ for $\kappa\in K_{-\eta}\cap T^*W_{\epsilon,k}$}.
        \end{align}
    \end{lemma}
\begin{proof}
     For any $(x,\kappa)\in K_{-\eta}$ and $j\geq 2$
    \begin{align*}
        d\varphi^{+,k}_{1}(H^{\omega_{\mathbf{l}}}_{\varphi^{-,k}_j}(x,\kappa))&=\omega_{\mathbf{l}}(H^{\omega_{\mathbf{l}}}_{\varphi^{+,k}_{j'}},H^{\omega_{\mathbf{l}}}_{\varphi^{-,k}_{j}})(x,\kappa)\\
        &=\kappa([\mathbf{U}^{+,k}_1,\mathbf{U}^{-,k}_j](x)\,+\,i\el\cdot\mathbf{F}_{\overline{\nabla}}(\mathbf{U}^{+,k}_1,\mathbf{U}^{-,k}_j)(x)\\
        &=-\eta\pi^{*}_{F\to SM}\alpha\,([\mathbf{U}^{+,k}_1,\mathbf{U}^{-,k}_j](x))\,-2\sum_{i=1}^dl_i,
    \end{align*}
    where we used relation (\ref{eq:vfield_symplectic}) in the second line and that $(x,\kappa)=-\eta\pi^{*}_{F\to SM}\alpha$ by definition of $K_{-\eta}$ (see (\ref{eq:subset_TstarF})). By Lemma \ref{lemma:construction_particular_vf} and the fact that each component of $\mathbf{l}$ is positive
    \begin{align*}
        \big\lvert\mathbf{l}\cdot\mathbf{F}_{\overline{\nabla}}(\mathbf{U}_1^{+,k},\mathbf{U}_j^{-,k})+2\sum_{i=1}^dl_i\big\rvert\leq \epsilon\sum_{i=1}^dl_i.
    \end{align*}
    Another consequence of Lemma \ref{lemma:construction_particular_vf} is that $\lvert \eta\pi^{*}_{F\to SM}\alpha\,([\mathbf{U}^{+,k}_1,\mathbf{U}^{-,k}_j](x))\rvert \leq \epsilon\eta$.
    Recall that $l_1>0$ by assumption, we may thus take $\epsilon$ small enough to obtain (\ref{eq:transversality_horo}). In the case of $\varphi_1^-$, using $\varphi_2^+$ and Lemma \ref{lemma:construction_particular_vf} again will give the same conclusion.
    
\end{proof}
\begin{remark}
    Notice that in the course of the proof, we took $\epsilon$ small enough to absorb the contributions of $\eta$ and the components of $\mathbf{l}$. Thus $W_k$ also depends implicitly on $\eta$ and $\mathbf{l}$. We will omit to mention this dependency from now on. 
\end{remark}
We remark that (\ref{eq:transversality_horo}) implies $H^{\omega_{\el}}_{\varphi^{-,k}_j}(\kappa)\notin T_{\kappa}K_{\eta}$, for $\kappa\in K_{\eta}\cap T^*W_{k}$. 
\begin{lemma}\label{lemma:property_hamiltonian_flow}
    $\mathrm{(i)}$ $H^{\omega_{\mathbf{l}}}_{\varphi^{-,k}_j}p=0, \, \forall(x,\xi)\in \Gamma_+\cap T^*W_k$.\\
    $\mathrm{(ii)}$ $H^{\omega_{\mathbf{l}}}_{\varphi^{-,k}_j}\varphi_{j'}^{-,k}(x,\xi)=0, \, \forall(x,\xi)\in \Gamma_+\cap T^*W_k$ for $j,j'\leq n-1$, implying that $H^{\omega_{\mathfrak{l}}}_{\varphi^{-,k}_j}(x,\xi)\in T_{(x,\xi)}\Gamma_+$ for all $(x,\xi)\in \Gamma_+\cap T^*W_k$
\end{lemma}
\begin{proof}
    Notice that $[X,\mathbf{U}^{-,k}_j]\in E^u_F$, thus by (\ref{eq:vfield_symplectic}) one has for all $(x,\xi)\in \Gamma_+\cap T^*W_k$ 
    \begin{align*}
        H^{\omega_{\mathbf{l}}}_{\varphi^{-,k}_j}p\,(x,\xi)&=\xi([X,\mathbf{U}^{-,k}_j](x))+i\mathbf{l}\cdot \mathbf{F}_{\overline{\nabla}}(X,\mathbf{U}_j^{-,k})(x)\\
        &=i\mathbf{l}\cdot \iota_X\mathbf{F}_{\overline{\nabla}}\,(\mathbf{U}_j^{-,k})(x)\\
        &\underset{\text{Lemma }(\ref{lemma:curvature})}{=}0,
    \end{align*}
    in the second line we used that $\xi(\mathbf{U}_j^{-,k})=0$ by definition of $\Gamma_+$.
    For (ii), first notice that $(\mathbf{F}_{\overline{\nabla}_i})_{|E^u\times E^u}=0$ for all $i\in \llbracket 1,d\rrbracket$. Indeed, for all $\mathbf{U}^-$ and $\tilde{\mathbf{U}}^-$ sections of $E^u_F\to F$, one writes for all $x\in F$ (using (\ref{eq:curvaturelinelift}))
    \begin{align*}
        \mathbf{F}_{\overline{\nabla}_i}(\mathbf{U}^-,\tilde{\mathbf{U}}^-)(x)&=-\omega_i(\Pi_\mathfrak{t}\Theta_{\mathrm{dyn}}([\mathbf{U}^-,\tilde{\mathbf{U}}^-]))\\
        &=0,
    \end{align*}
    since $E^u_F\subset \ker(\Theta_{\mathrm{dyn}})$ and $E^u_F$ is integrable. Finally, the claim again follows from (\ref{eq:vfield_symplectic}): $\xi([\mathbf{U}^-_j,\mathbf{U}_{j'}^-](x))=0$ if $(x,\xi)\in \Gamma_+$ thus removing the $\omega_0$ term, the curvature part will vanish as $(\mathbf{F}_{\overline{\nabla}_i})_{|E^u\times E^u}=0$.
    
\end{proof}
Since we will integrate each $H^{\omega_{\mathbf{l}}}_{\varphi^{-,k}_j}$ around $K_{-\eta}\cap T^*W_k$ for small times (in a sense that will be clearer below), we just need to be cautious with the fact that the Hamiltonian flow may not be defined for any time: the flow of $\mathbf{U}^{-,k}_{j}$, which is the projection of the flow $\bigl(\Phi^{\omega_{\mathbf{l}},\varphi_j^{-,k}}_t\bigr)$ on $F$, could leave $W_k$. Up to shrinking the open sets $W_k$ of the cover, we can always assume that each flow $\bigl(\Phi^{\omega_{\mathbf{l}},\varphi_j^{-,k}}_t\bigr)$ is defined for a uniform time $t_0>0$ small enough, thanks to the compactness of $F$. For simplicity, we still write $(W_k)$ the desired covering.

Gathering Lemma \ref{lemma:transversality_horocyclic} and Lemma \ref{lemma:property_hamiltonian_flow}, the inverse function theorem shows the existence of a strictly positive time $t_1<t_0$ such that the following map
\begin{align}\label{eq:parametrisation_horocyclic}
\begin{aligned}
\psi_k : (-t_1,t_1)^{n-1}\times (K_{-\eta}\cap T^*W_k) &\to \Gamma_+(\eta) \\
(s_1,\dotsc,s_{n-1},\kappa) \quad\quad\quad&\longmapsto \Phi_{s_1}^{\omega_\mathbf{l},\varphi_1^{-,k}}\circ\cdots\circ\Phi_{s_{n-1}}^{\omega_\mathbf{l},\varphi_{n-1}^{-,k}}(\kappa),
\end{aligned}
\end{align}
is a diffeomorphism onto its image $\mathcal{O}_k({\eta})\subset\Gamma_+(\eta)\cap T^*W_k$. That this map is well-defined is a consequence of point (i) and (ii) in Lemma \ref{lemma:property_hamiltonian_flow}. This allows us to prove the crucial Lipschitz bounds on the semiclassical measure $\mu$. 
\begin{lemma}\label{lemma:lipschitz_bound}
    Let $B^k_{\delta}:=B_{\mathbb{R}^{n-1}}(0,\delta)\times (K_{-\eta}\cap T^*W_k)$, where the radius $\delta>0$ of the Euclidean ball is taken to be strictly less than $t_1/2$. There exists a constant $c>0$ such that
    \begin{align}\label{eq:lipschitz_mu}
    c^{-1}\delta^{n-1}\leq\psi_k^*\mu(B^k_{\delta})\leq c\delta^{n-1}.
    \end{align}
\end{lemma}
\begin{proof}
    The differential of the diffeomorphism $\psi_k$ maps $\partial_{s_j}$ to $H^{\omega_{\mathbf{l}}}_{\varphi_j^{-,k}}$ for all $j\leq n-1$. Proposition \ref{prop:propagation_measure_horocyclic} thus implies that the measure $\psi_k^*\mu$ satisfies the transport equation
    \begin{align}\label{eq:transport_eq_coord}
        \partial_{s_j}\psi_k^*\mu&=-(\psi_k^*\mathbf{G}_j^{-,k})\psi_k^*\mu\\
        &=:\mathbf{g}^{-,k}_{j}\psi_k^*\mu\notag
    \end{align}
    where each $\mathbf{g}_j^{-,k}$ is smooth. We may form the Radon measure $\mu'_k$ on $B^k_{\delta}$ by setting
    \begin{align}\label{eq:radon_invariant_new}
        \mu'_k=\exp\biggl(-\sum_{j=1}^{n-1}s_j\int_0^{1}\mathbf{g}^{-,k}_j(ts,\kappa)dt\biggr)\psi_k^*\mu=:f(s,\kappa)\psi_{k}^*\mu
    \end{align}
    which satisfies
    \begin{align*}
        \partial_{s_j}\mu'_k=0.
    \end{align*}
    Since $B_{\delta}^k$ is connected, there exists $C>0$ such that 
    \begin{align*}
        \mu'_k=\mathrm{Leb}\otimes \nu,
    \end{align*}
    where $\mathrm{Leb}$ is the Lebesgue measure on the first component of $B_{\delta}^k$ and $\nu$ is a non-zero Radon measure (by Proposition \ref{prop:support_prop_measure}) on the second component. Finally 
    \begin{align*}
        \psi_k^*\mu=f(s,\kappa)^{-1}\mathrm{Leb}\otimes \nu
    \end{align*}
    Setting $c_1=\inf_{(s,\kappa)\in B^k_{\delta}}f(s,\kappa)^{-1}>0$ and $c_2=\sup_{(s,\kappa)\in B^k_{\delta}}f(s,\kappa)^{-1}>0$
    \begin{align*}
        \mathrm{Leb}(B_{\R^n}(0,1))c_1\nu(K_{-\eta}\cap T^*W_k)\delta^{n-1}\leq \psi_k^*\mu(B_\delta^k)\leq \mathrm{Leb}(B_{\R^n}(0,1))c_2\nu(K_{-\eta}\cap T^*W_k)\delta^{n-1},
    \end{align*}
    defining $c=\min\bigl(c_1\nu(K_{-\eta}\cap T^*W_k),c_2\nu(K_{-\eta}\cap T^*W_k)^{-1}\bigr)$ yields the bound (\ref{eq:lipschitz_mu}).
    \end{proof}

\end{proof}
Recall that there exists $C>0$ such that  $\lvert \Phi^{\omega_0,p}_{-t}(x,\xi)\rvert\leq Ce^{-t}$ for all $(x,\xi)\in E_{u,F}^*$, see the beginning of Section \ref{subsection:anisofourierwise} concerning the sink and source structure of $(\Phi_t^{\omega_0,p})$. This contraction property cannot co-exists with the propagation property of $\mu$ along the flow $\Phi_t^{\omega_0,p}$ as we shall explain now.

Let $\delta_0>0$ and set $U_{\delta_0}:=\bigl\{\kappa\in T^*F,\, d_{T^*F}(\kappa,K_{-\eta})<\delta_0\bigr\}\cap \Gamma_+(-\eta)$ to be an open neighborhood of $K_{-\eta}$ in $\Gamma_+(-\eta)$. By the previous contraction property
\begin{align}\label{eq:contraction_propert_flow}
    \Phi_{-t}^{\omega_0,p}(U_{\delta_0})\subset U_{C\delta_0 e^{-t}}.
\end{align}
For $t$ large enough, we will have for some $C'>0$ (using the notation of Lemma \ref{lemma:lipschitz_bound}) 
\begin{align}\label{eq:inclusion_psi_k}
    U_{C\delta_0 e^{-t}}\subset \bigcup_{k} \psi_k(B^k_{C'\delta_0 e^{-t}}).
\end{align}
Relations (\ref{eq:contraction_propert_flow}) and (\ref{eq:inclusion_psi_k}) together with the Lipschitz upper bound in Lemma \ref{lemma:lipschitz_bound} implies for $t$ large enough
\begin{align*}
    \mu\bigl(\Phi_{-t}^{\omega_0,p}(U_{\delta_0})\bigr)\leq (C')^{n-1}\delta_0^{n-1}e^{-(n-1)t},
\end{align*}
Finally, the propagation property for the measure $\mu$ in Proposition \ref{prop:semiclassicalmeasure_first_properties} allows us to write, up to some time independant constants 
\begin{align}\label{eq:exp_bound_contradiction}
    e^{2\nu t}\lesssim e^{-(n-1)t},
\end{align}
Since $\nu>-1(\geq-(n-1)/2)$, we get the desired contradiction.
\subsubsection{The large $\langle \Im(z)\rangle$ case }\label{subsubsec:large_im} Without loss in generalities, we prove as before our estimates in the domain $\Im(z)\geq 0$. Recall that this regime correspond to the case $\langle \Im(z)\rangle\geq \langle \K \rangle$. Assuming that (\ref{eq:case-im}) does not hold, we obtain sequences $(u_j)$, $(\K_j)$ and $(z_j)$ such that, after an eventual extraction 

\begin{align}
\begin{cases}\label{eq:semiclassical_pb_large_im1}
(-i\langle \Im(z_j) \rangle^{-1}\X_{\mathbf{k}_j}-i\langle \Im(z_j) \rangle^{-1} z_j)u_j = o(\langle \Im(z_j) \rangle^{-2}), 
& \text{with } \Re(z_j)\to \nu,  \\[6pt]
\|u_j\|_{\mathcal{H}_{\mathrm{hol},\langle \Im(z_j) \rangle^{-1}}^{N}} = 1.
\end{cases}
\end{align}

Here we still have by assumption that $\nu\in (-1,1)$. After a further extraction 
\begin{equation}\label{eq:large_im_asymptotics}
    \lim_{j\to +\infty}\mathbf{k}_j\langle \Im(z_j) \rangle^{-1}=\mathbf{l},\,\,\lvert\mathbf{l}\rvert\leq 1.
\end{equation}
Notice the fundamental difference with the previous case: $\el$ may be equal to $\mathbf{0}$ (for instance if $(\K_j)$ has bounded norm). 
We would like to use $h_j:=\langle \Im(z_j) \rangle^{-1}$ as a semiclassical parameter, that is it needs to go to $0$ as $j\to +\infty$.
\begin{lemma}\label{lemma:im_goes_infinity}
    The sequence $(\Im(z_j))$ goes to $+\infty$ as $j\to+\infty$.
\end{lemma}
\begin{proof}
    If it was not the case, up to an extraction, we can assume that $z_j$ converges to $z_0\in \mathcal{B}_0'(\varepsilon)$ and that $\mathbf{k}_j=:\K_0\in \Lambda$ for large enough $j$. Then $\langle \Im(z_j)\rangle\to \langle \Im(z_0)\rangle$, and so 
    \begin{align*}
        \|(-i\langle \Im(z_j) \rangle^{-1}\X_{\mathbf{k}_0}-i\langle \Im(z_j) \rangle^{-1} z_j)u_j\|_{\mathcal{H}^N_{h_j,\mathrm{hol}}}=o(1).
    \end{align*}
    Thus $z_0$ is in the spectrum of $\mathbf{X}_{\K_0}$, this is impossible because $\Im(z_0)>\lambda_{\varepsilon}+1$ by assumption (recall the notation in the beginning of Section \ref{subsection:semiclassicalformul}). In order for the proof to be complete, we need to justify that there is indeed a finite number of resonances in $\mathcal{B}_0(\varepsilon)$, see the end of Section \ref{subsec:proof_re_minus_1}. 
\end{proof}
Thus
\begin{align*}
    \lim_{j\to +\infty} \Im(z_j)\langle \Im(z_j) \rangle^{-1}\to 1
\end{align*}

Seeing $z_j$ and $\K_j$ as  functions of $h_j$, we may write $z(h)$ and $\K(h)$. As in (\ref{eq:semiclassical_pb_large_k}), we end up with the semiclassical formulation
\begin{align}\label{eq:semiclassical_pb_large_im2}
\begin{cases}
(-ih\X_{\mathbf{k}(h)}+h\Im(z(h))-ih\Re(z(h)))u_h = o(h^2), 
\\[6pt]
\|u_h\|_{\mathcal{H}_{\mathrm{hol},h}^{N}} = 1.
\end{cases}
\end{align}
Thus (\ref{eq:semiclassical_pb_large_im2}) still gives us that $u_h\rightharpoonup \mu$ with $\mu$ satisfying Proposition \ref{prop:semiclassicalmeasure_first_properties} and \ref{prop:support_prop_measure}. The horocyclic invariance of Proposition \ref{prop:horo_invariance_quasimode_re_larger_minus1} holds as well. We just have to distinguish between two cases.

\textbf{Case 1 : $\el \neq 0$ } We can still use the vector fields of Lemma \ref{lemma:construction_particular_vf} to get the transversality property of the unstable Hamiltonian vector fields. The previous construction indeed works because the curvature part of $\omega_{\mathbf{l}}$ does not vanish. Indeed, the construction of $(\mathbf{U}_j^-)$ is purely geometrical and holds regardless of the semiclassical regime. Since there exists a component of $\mathbf{l}$ (say $l_1$ again) which is non-zero, the proof of Lemma \ref{lemma:transversality_horocyclic} still works. One can conclude by following \textit{verbatim} the arguments at the end of Section \ref{subsec:proof_re_minus_1_large_k}.

\textbf{Case 2 : $\el = 0$ } This case is interesting because one could say that it is the opposite of the case where $\mathbf{l}\neq 0$ and $\eta=0$ (which was only possible in the large $\langle \mathbf{k} \rangle$ regime). Indeed, to treat the latter case we crucially used the existence of local unstable $\mathbf{U}^-$ and stable $\mathbf{U}^+$ local vector fields on $F$ such that $\|\mathbf{F}_{\nabla}(\mathbf{U}^-,\mathbf{U}^+)\|>1$; There was no hope in relying on the contact property of the $1$-form $\alpha$ as the measure was localized on the unstable tail $\Gamma_+(0)$. Let $\mathbf{U}^{\pm}$ be two sections of $E^{u/s}_F$, $\varphi^{\pm}:=\xi(\mathbf{U}^{\pm})$ their associated Hamiltonian functions. Taking $(x,-\eta\pi_{F\to SM}^*\alpha)\in K_{-\eta}$, we have 
\begin{align}
d\varphi^{+}(H^{\omega_{\mathbf{l}}}_{\varphi^{-}}(x,-\eta\pi_{F\to SM}^*\alpha))
&=\bigl\{ \varphi^{+},\varphi^{-}\bigr\}_{\omega_{\el}}(x,-\eta\pi_{F\to SM}^*\alpha)\notag\\
&=-\eta\pi_{F\to SM}^*\alpha([\mathbf{U}^{+},\mathbf{U}^{-}](x))
+i\el\cdot\mathbf{F}_{\overline{\nabla}}(\mathbf{U}^{+}(x),\mathbf{U}^{-}(x))\notag\label{eq:contact_case_im}\\
&\underset{\el=0}{=}
-\eta\pi_{F\to SM}^*\alpha([\mathbf{U}^{+},\mathbf{U}^{-}](x))
.
\end{align}
To construct transversal vector fields, we may rely only on the \textit{contact property of $\alpha$} this time. 
\begin{lemma}\label{lemma:contact_property_transversality}
    Let $hT\in F$ and $(\mathbf{U}^-_i)_{1\leq i \leq n-1}$ be a local horizontal frame for $E^u_F$ on a small enough neighborhood $W\subset F$ of $hT$.
    
    There exists a local horizontal frame $(\mathbf{U}^+_i)_{1\leq i \leq n-1}$ of $E^s_F$ over $W$ such that :
    \begin{align}\label{eq:contact_property_transversality}
        \bigl(-\eta\pi_{F\to SM}^*\alpha ([\mathbf{U}^-_i,\mathbf{U}^+_j]\bigr)_{1\leq i,j\leq n-1}=\mathbf{I}_{n-1}.
    \end{align}
\end{lemma}
\begin{proof}
Fix $(U_i^-)$ a local frame of $E^u_{SM}$ on a contractible open set $W'\subset SM$ and write $(U_i^-)^*$ the corresponding dual frame of the dual bundle $(E^u_{SM})^*$. $d\alpha$ is a non-degenerate $2$-form on $E_{SM}^u\oplus E_{SM}^s$. By non-degeneracy, the contraction map $\iota: E^s_{SM}\to (E^u_{SM})^*$ is an isomorphism. It suffices to set $U_i^+:=-\frac{1}{\eta}\iota^{-1}\bigl((U_i^-)^*\bigr)$ to obtain that $(-\eta d\alpha(U_i^+,U_j^-))_{1\leq i,j\leq n-1}$ is the identity matrix. Setting $\mathbf{U}_i^{\pm}=(U_i^{\pm})^{\mathbb{H}_F}$ (local vector fields over $\pi^{-1}_{F\to SM}(W'))$ gives the claim.
\end{proof}
The previous lemma yields an open covering $(W_k)$ of $F$ and vector fields $\mathbf{U}^{-,k}_i$ such that their respective Hamiltonian vector fields $H^{\omega_0}_{\varphi^{-,k}_{i}}$ are transversal to $TK$ on $K_{-\eta}$ since $\eta\neq 0$ and thanks to the relation (\ref{eq:contact_property_transversality}). Also, in the course of the proof of Lemma \ref{lemma:transversality_horocyclic} in this case, one has to use the fact that $d\alpha_{|E^u_{SM}\times E^u_{SM}}=0$ (instead of $(\mathbf{F}_{\overline{\nabla}})_{|E^u_{F}\times E^u_{F}}=0$) to obtain the same conclusion. It suffices to apply the final arguments from the large $\langle \K \rangle$ case (beginning at Lemma \ref{lemma:lipschitz_bound}) to obtain the same contradiction.
\\

\textbf{There is only a finite number of resonances in $\mathcal{B}_0(\varepsilon)$} We used in Lemma \ref{lemma:k_goes_infinity} and Lemma \ref{lemma:im_goes_infinity} the fact that $\cup_\K \sigma_{\mathrm{PR}}(\X_\K)$ is finite in $\mathcal{B}_0(\varepsilon)=\{0\geq\Re(z)>-1+\varepsilon\}$. Let us justify this claim.

The line of reasoning is usual and goes as follows. The elements of $\sigma_{\mathrm{PR}}(\X_\K)$ are written $(z_{j,\K})$ and the corresponding resonant states are $(u_{j,\K})$ (the reader should bear in mind that one should also take into account the multiplicity of the poles, but it is finite and the argument is unchanged).
Assume for the sake of contradiction that $\cup_\K \sigma_{\mathrm{PR}}(\X_{\K})\cap\mathcal{B}'_0(\varepsilon)$ is infinite. Then after extraction there are two possible cases.

\textit{The case $\langle \K_j \rangle\geq \langle \Im(z_j)\rangle $.} Then $\langle \K_j\rangle$ must go to infinity. If it did not this would imply that $|\Im(z_j)|$ is also bounded, and we can assume that $\K_j=:\K_0$. This implies that $(-\X_{\K_0}-z_j)u_{j,\K_0}=0$ for all $j$, contradicting the meromorphic property of $z\mapsto (-\X_{\K_0}-z)^{-1}$. We fall exactly in the setting of the semiclassical problem in  (\ref{eq:semiclassical_pb_large_k}), the only difference being that the sequence of resonant states $(u_h)$ are exact (meaning $\mathbf{P}_hu_h=0$) and not quasimodes. Horocyclic invariance is even easier to obtain. Inspecting the proof of Proposition \ref{prop:horo_invariance_quasimode_re_larger_minus1}, we see that the relation (\ref{eq:commutation_algebraic_unstable}) implies that $h\nabla^{\K(h),u}u_h$ is a resonant state of $\mathbf{P}^u_h$. But (\ref{eq:resolvent_formula_L2}) exactly shows that $\mathbf{P}^u_h$ does not have resonant states in $\{\Re(z)>-1\}$, hence horocyclic invariance. The result of Lemma \ref{lemma:lipschitz_bound} still applies for the semiclassical measure associated to this sequence of resonant state and the end of the proof for the large $\langle \K \rangle$ applies.

\textit{The case $\langle \Im(z_j)\rangle \geq \langle \K_j\rangle$.} Again, if $\langle \Im(z_j)\rangle$ did not tend to infinity, then $\K_j=\K_0'$ after extraction and we would contradict the meromorphic property of the resolvent on $\mathcal{B'}_0(\varepsilon)$. Thus we are in the setting of (\ref{eq:semiclassical_pb_large_im1}). The above justification for horoyclic invariance applies and the arguments of Section \ref{subsubsec:large_im} carry as well.

\section{Resolvent estimates in deeper bands}\label{Section:resolvent_estim_re_bigger_1}
Take $\varepsilon\in (0,1)$, $k\in \mathbb{N}_{>0}$ and write $\{\beta_j(k)\}_{j\leq l(\varepsilon)}$ to be the finite set of all the poles of $z\mapsto (-X_{FM}-z)^{-1}$ in the band $\mathcal{B}_k(\varepsilon)$. The justification of this finiteness property is given at the end of this Section. We define 
\begin{align*}
    \lambda_{\varepsilon}(k):=\max_{j\leq l(\varepsilon)} |\Im(\beta_j(k))|,
\end{align*}
allowing us to define the region where we will perform the resolvent estimates, namely $\mathcal{B}_k(\varepsilon)':=\mathcal{B}_k(\varepsilon)\,\cap\, \{|\Im(z)|\geq \lambda_{\varepsilon}(k)+1\}$. Again, this restriction is only a matter of notational convenience: for $\langle \K \rangle$ large enough the estimate can be proved on the whole strip $\mathcal{B}_k(\varepsilon)$.
We aim at proving the following estimates, to compare with Proposition \ref{prop:two_asymptotics_estimates}. 
\begin{proposition}\label{prop:two_asymptotics_estimates2}
    Let $N>(n-1)/2+k$. There exists $C_{\varepsilon}$ such that for all $z\in \mathcal{B}_k(\varepsilon)$ 
    \begin{align}
\|(-\X_{\K}-z)^{-1}\|_{\mathcal{H}^{N}_{\langle \K \rangle^{-1},\mathrm{hol}}}
&\le C_{\varepsilon}\langle \mathbf{k}\rangle^{\lfloor \Re(z) \rfloor},
&\qquad&
\text{if }\langle \mathbf{k}\rangle \ge \langle \Im(z)\rangle,
\label{eq:case-k_deep}
\\[0.5em]
\|(-\X_{\K}-z)^{-1}\|_{\mathcal{H}^{N}_{\langle \Im(z)\rangle^{-1},\mathrm{hol}}}
&\le C_{\varepsilon}\langle \Im(z)\rangle^{\lfloor \Re(z) \rfloor},
&\qquad&
\text{if }\langle \Im(z)\rangle \ge \langle \mathbf{k}\rangle.
\label{eq:case-im_deep}
\end{align}
\end{proposition}
Exactly as in Lemma \ref{lemma:modified_sobolev_estimate}, the resolvent estimates can be proved on $\mathcal{H}'^N$ up to taking the Sobolev regularity to satisfy $N>(n-1)/2+k+1$. 

\subsection{Estimates in $\Re(z)>-(n-1)/2$}\label{subection:resolvent_estim_re_bigger_vertical1}
So far, we have proved the resolvent estimates on the region $\{0\geq\Re(z)\geq -1+\varepsilon\}$ of the complex plane, for any $\varepsilon>0$. If the dimension is $3$ then the proof cannot go further: the contradiction in Section \ref{Section:resolvent_estim_re_bigger_1} came from the bound $e^{2\nu t}\lesssim e^{-(n-1)t}$, if $n=3$ there is no contradiction when $\nu=-(3-1)/2=-1$. This dimension will be considered again in Section \ref{section:estimates_deeper_band}. In dimensions greater than $3$, we aim at extending the resolvent estimates to $\mathcal{B}_0(\varepsilon)=\{0\geq\Re(z)>-(n-1)/2+\varepsilon\}$. We assume for now that $n=2j+1, j\in \mathbb{N}^*$, the even case will be considered at the end of this section and is almost identical. We will proceed by induction: resolvent estimates in $\{-(m-1)\geq \Re(z)>-m+\varepsilon\}$ imply estimates in $\{-m\geq\Re(z)>-m-1+\varepsilon\}$, where $m$ is an interger such that $m<j$.

We will only show how to iterate the large $\langle \mathbf{k} \rangle$ estimate (see Section \ref{subsec:proof_re_minus_1_large_k}), the large imaginary part estimate will have a strictly identical line of proof. Assume we have for $m<j$ that there exists $C_{\varepsilon,m}>0$ such that 
\begin{align}\label{eq:induction_hypothesis}
    \|(-\X_{\K}-z)^{-1}\|_{\mathcal{H}^N_{\langle \K \rangle ^{-1},\mathrm{hol}}}\leq C_{\varepsilon,m}\langle \K \rangle^{m} \text{ in }\Re(z)\in[-m+\varepsilon,-m+1],\, \langle \K\rangle\geq \langle\Im(z)\rangle.
\end{align}
From this we would like to deduce an estimate in the left-adjacent band
\begin{align}\label{eq:conclusion_induction_adjacent_band}
    \|(-\X_{\K}-z)^{-1}\|_{\mathcal{H}^N_{\langle \K \rangle ^{-1},\mathrm{hol}}}\leq C_{\varepsilon,m+1}\langle \K \rangle^{m+1}, \text{ in }\Re(z)\in[-m-1+\varepsilon,-m].
\end{align}
As before, assume that the latter estimate does not hold. We end up with the following semiclassical problem
\begin{align}
\begin{cases}
(-ih\X_{\mathbf{k}(h)}+h\Im(z(h))-ih\Re(z(h))u_h = o(h^{m+2}), 
\\[6pt]
\|u_h\|_{\mathcal{H}_{hol,\langle \mathbf{k}_j\rangle^{-1}}^{N}} = 1.
\end{cases}
\end{align}
where we recall that $h_j:=\langle \mathbf{k}_j\rangle^{-1}$ (which we write $h$ for simplicity). Notice that in contrast with (\ref{eq:semiclassical_pb_large_k}), our quasimodes are sharper as $m\geq 1$. In turn we can assume that 
\begin{align}
    \Re(z_h)\to \nu \in [-m-1-\varepsilon,-m]\, \text{ , } \dfrac{\Im(z_h)}{\langle \mathbf{k}(h)\rangle }\to \eta \in [-1,1]\,\text{ and } h\mathbf{k}(h)\to_{h\to 0} \mathbf{l}\neq 0.
\end{align}
Again, $u_h\rightharpoonup \mu$ after an eventual extraction and $\mu$ satisfies Propositions \ref{prop:semiclassicalmeasure_first_properties} and \ref{prop:support_prop_measure}. 

The only non-trivial point to perform the induction procedure is the part concerning (quasi) horocyclic invariance (see Proposition \ref{prop:horo_invariance_quasimode_re_larger_minus1}), which allowed us to deduce the Lipschitz bounds on $\mu$ in Lemma \ref{lemma:lipschitz_bound}: we crucially used that $\Re(z_h)/h>-1$ in the proof of Proposition \ref{prop:horo_invariance_quasimode_re_larger_minus1}. In order to prove this invariance here, recall from (\ref{eq:identification}) that $C^{\infty}_{\mathrm{hol}}(F,\El^{\K})$ may be identified with sections $C^{\infty}\bigl(SM,H^0(F,\El^{\K(h)})\bigr)$. Using Lemma \ref{eq:intertwinconnection}, the quasimodes $u_h$ for $\mathbf{P}_h(z_h)$ can be seen as quasimodes $\tilde{u}_h:=i_{h}(u_h)\in \mathcal{H}_{\langle \K(h) \rangle^{-1}}^N(SM,H^0(F,\El^{\K(h)})) $ for $(-ih\nabla^{H^0(F,\El^{\K})}_X-ihz_h)$. This is a mere consequence of the continuity of the map $i_{\K(h)}:\mathcal{H}^N_{h,\mathrm{hol}}(F,\El^{\K(h)})\to \mathcal{H}_{h}^N(SM,H^0(F,\El^{\K(h)}))$, induced from the one defined on smooth sections in Lemma \ref{lemma:intertwinconnection}. At this point, it may seem artificial to work on anisotropic spaces over $SM$, but this will appear to be convenient in the next lemma: we can nicely decompose $\mathcal{H}_{h}^N(SM,H^0(F,\El^{\K(h)})\otimes (E_s^*\otimes \mathbb{C}))$ into several anisotropic components using decomposition of tensor products of irreducible representations. 

Most of the time, we will omit the dependency of $\mathbf{k}$ in $h$ in the notation.
We thus have 
\begin{align}\label{eq:quasimode_on_sm}
    \|(-ih\nabla^{H^0(F,\El^{\K})}_X+h\Im(z(h))-ih\Re(z(h))u_h\|_{\mathcal{H}^N_{h}} = o(h^{m+2}).
\end{align}
The operator in the latter norm will be written $\mathbf{P}^{H^0(F,\El^\K)}_h(z_h)$.
We consider the complexified horocyclic operator $\mathcal{U}^-_{\mathbb{C}} : C^{\infty}(SM,H^0(F,\El^\K))\to C^\infty(SM,H^0(F,\El^\K)\otimes_{\mathbb{C}}E_{s}^*(\mathbb{C}))$, whose real version is defined in (\ref{eq:horocyclic_op}). We will omit the mention of $\mathbb{C}$ for a lighter notation, but the reader should keep in mind that we work over $\mathbb{C}$. We recall that, since we work with a connection $\mathbb{H}_{\mathrm{dyn}}$ on $FM\to SM$, any associated vector bundle $\mathcal{V}_\rho$ on $SM$ comes equipped with an induced connection $\nabla^{\mathcal{V}_\rho}$. Here, recall from (\ref{eq:unstableassociated}) that $E^u_{SM}(\mathbb{C})$ is an associated vector bundle (the adjoint representation needs to be complexified), and so is its dual $E_s^*(\C)$. We write the natural induced connection $\nabla^{E_s^*}$ and $\nabla^{H^0(F,\El^\K)\otimes E_s^*}$ the connection induced on the tensor bundle $H^0(F,\El^\K)\otimes E_s^*$. 
From Lemma \ref{lemma:commutationhorocyclic}, the horocyclic operator satisfies the commutation relation 
\begin{align}\label{eq:commutation_on_sm_esstar}
[\nabla^{H^0(F,\El^{\K})}_X,\mathcal{U}^-]&:=\nabla^{H^0(F,\El^{\K})\otimes E_s^*}_X\mathcal{U}^- - \mathcal{U}^-\nabla_X^{H^0(F,\El^{\K})}.\\
    &= -\mathcal{U}^-\notag.
\end{align}
By boundedness of $h\mathcal{U}^-:{\mathcal{H}^N_{h}(SM,H^0(F,\El^\K)))}\to \mathcal{H}'^N_{h}(SM,H^0(F,\El^\K)\otimes E_s^*)$ (see Remark \ref{rk:modified_N} for the definition of $\mathcal{H}'^\bullet$)
\begin{align}\label{eq:horocyclic_applied_deeper_band}
    \|h\mathcal{U}^-\mathbf{P}^{H^0(F,\El^\K)}_h(z_h)\|_{\mathcal{H}'^N_{h}(SM,H^0(F,\El^\K\otimes E_s^*))}=o(h^{m+2}),
\end{align}
and the commutation relation (\ref{eq:commutation_on_sm_esstar}) gives 
\begin{align}\label{eq:commutation_on_sm2}
    h\mathcal{U}^-\mathbf{P}^{H^0(F,\El^\K)}_h(z_h)u_h&=\mathbf{P}_h^{H^0(F,\El^\K)\otimes E_s^*}(z_h+ih)\,h\mathcal{U}^-u_h\\
    &=\bigl(-ih\nabla^{H^0(F,\El^\K)\otimes E_s^*}+h\Im(z(h))+1-ih\Re(z(h))\bigr)h\mathcal{U}^-u_h\notag.
\end{align}
In view of (\ref{eq:commutation_on_sm2}), a semiclassical estimate on the inverse of $\mathbf{P}_h^{H^0(F,\El^\K)\otimes E_s^*}(z_h+ih)$ remains to be proved.
\begin{lemma}\label{lemma:bound_inverse_sm}
    There exists $C>0$ (independant of $h$) such that 
    \begin{align}\label{eq:bound_inverse_sm}
        \|\mathbf{P}_h^{H^0(F,\El^\K)\otimes E_s^*}(z_h+ih)^{-1}\|_{\mathcal{H}'^N_h(H^0(F,\El^\K)\otimes E_s^*)}\leq C h^{-m-1}.
    \end{align}
\end{lemma}
\begin{proof}
    In this proof, let us write for simplicity $\|\cdot\|_1$ the anisotropic norm on $\mathcal{H}_1:=\mathcal{H}'^N_h(SM,H^0(F,\El^\K)\otimes E_s^*(\C))$. Recall that $(E^u_{SM}(\mathbb{C}))^*\cong FM\times_{Ad(G)^{\vee}}(\mathfrak{n}^-_{\mathbb{C}})^*$ where $(Ad(G)^{\vee},(\mathfrak{n}^-_{\mathbb{C}})^*)$ is the dual representation of  $(Ad(G),\mathfrak{n}^-_{\mathbb{C}})$. The representation is $(Ad(G),\mathfrak{n}^-)$ is isomorphic to the vector representation of $G=SO(n-1)$ on $\mathbb{C}^{n-1}$. Indeed, an element $g\in SO(n-1)$ is canonically seen in $PSO(1,n)$ using the diagonal embedding to a block-diagonal matrix $\diag(1,1,g)$. Then, by (\ref{eq:rootspacehyp}), we may canonically identify $\mathbf{v}\in \mathbb{R}^{n-1}$ to a matrix $A(\mathbf{v})\in \mathfrak{n}^-$. For all $g\in G$, $Ad(g)(A(\mathbf{v}))=gA(\mathbf{v})g^{-1}=A(g\mathbf{v})$: $A$ intertwins the standard vector representation and the adjoint representation on $\mathfrak{n}^-$. In turn, this implies that $(Ad(G)^{\vee},(\mathfrak{n}^-_{\mathbb{C}})^*)$ is isomorphic to the (dual of the) standard vector representation of $G$ on $\C^{n-1}$, which is irreducible since $n\geq 4$. We note that this is simply a manifestation of the general fact that the normal bundle $\mathcal{N}$ is isomorphic to $E^u_{SM}$ (this isomorphism of representation descends to an isomorphism of associated vector bundle).
    
    The tensor representation $H^0(G/T,\mathbf{J}^{\mathbf{k}})\otimes_{\mathbb{C}} \mathfrak{n}^-_{\mathbb{C}}$ may be decomposed in a, bounded uniformly in $\K$, number of factors by \cite[Corollary 9.2.4]{GoodmanWallach2009}
    \begin{align}
        H^0(G/T,\mathbf{J}^{\mathbf{k}})\otimes (\mathfrak{n}^-_{\mathbb{C}})^*\cong \bigoplus_{j=1}^d H^0(G/T,\mathbf{J}^{\K + \mathbf{e}_j})\,\,\text{if $n$ is odd}\label{eq:decomposition_odd},\\
        H^0(G/T,\mathbf{J}^{\mathbf{k}})\otimes (\mathfrak{n}^-_{\mathbb{C}})^*\cong H^0(G/T,\mathbf{J}^\K)\oplus \bigoplus_{j=1}^d H^0(G/T,\mathbf{J}^{\K + \mathbf{e}_j})\,\,\text{if $n$ is even}\label{eq:decomposition_even}.
    \end{align}
    For simpler notations, we recall that we assume that $n$ is odd (again, the even case is dealt with the same procedure). Thanks to (\ref{eq: fiber_id}), the decomposition descends to the level of the corresponding associated vector bundles, that is 
    \begin{align}\label{eq:decmposition_descended_bundle}
        \Phi^\K :
H^0(F,\El^\K)\otimes E_s^*(\mathbb C)
\overset{\cong}{\longrightarrow}
\bigoplus_{j=1}^d H^0(F,\mathbf L^{\K+\mathbf e_j}).
    \end{align}
    Let us write the previous isomorphism of vector bundle $\Phi_{\mathbf{k}}$ (from the left-hand side to the right-hand side). The latter induces a map 
    \begin{align*}
        \Phi^\K_*:\mathcal{D}'(SM,H^0(F,\El^\K)\otimes E_s^*(\mathbb{C}))\to \bigl(\mathcal{D}'(SM,H^0(F,\El^{\K+\mathbf{e}_j}))\bigr)_{1\leq j \leq d}
    \end{align*}

    Each component of a (distributional) section $s$ in the previous direct sum will be written $s_j$, for $j\leq d$. Through this isomorphism, we have for any $s\in \mathcal{D}'(SM,H^0(F,\El^\K)\otimes E_s^*)$ 
    \begin{align}\label{eq:phi_k_intertwins}
        \Phi^{\K}_*(\nabla^{H^0(F,\El^\K)\otimes E_s^*}_X s)=\bigl(\nabla^{H^0(F,\El^{\K+\mathbf{e}_j})}_X\Phi^\K_*(s)_j)_{1\leq j \leq d}.
    \end{align}
    The anisotropic norm on $\mathcal{H}_2:=\mathcal{H}'^N_h(SM,\bigoplus_{j=1}^d H^0(F,\mathbf{L}^{\K + \mathbf{e}_j}))$ is equivalent to the natural product norm induced by the ones from each anisotropic block. We write this product norm $\|\cdot\|_{2}$. Notice that $\Phi^{\K}_*:(\mathcal{H}_1,\|\cdot\|_1)\to (\mathcal{H}_2,\|\cdot \|_2)$ is continuous as follows from usual $L^2$-boundedness since $\Phi_{\mathbf{k}}$ is a semiclassical pseudodifferential operator of order $0$ (between vector bundles). This yields that there exists $C>0$ independent of $h$ such that 
    \begin{align}\label{eq:phi_k_c0}
    ||\Phi^{\K}_*\|_{\mathcal{H}_1\to \mathcal{H}_2}\leq C.
    \end{align}
    For all $j\in \llbracket 1,d\rrbracket$, the induction hypothesis applies to $P_h^{H^0(F,\El^{\K+\mathbf{e}_j})}(z_h+ih)$ since $\Re(z_h)+1\in [-m+\varepsilon,-m+1]$. More precisely
    \begin{align*}        \|P_h^{H^0(F,\El^{\K+\mathbf{e}_j})}(z_h+ih)^{-1}\|_{\mathcal{H}'^N_h(H^0(F,\El^\K)\otimes E_s^*)}\lesssim h^{-1}h^{\lfloor \Re(z_h)\rfloor +1}=h^{-m-1}.
    \end{align*}
    The continuity property (\ref{eq:phi_k_c0}) together with (\ref{eq:phi_k_intertwins}) gives the desired estimate.
\end{proof}
\begin{remark}
    Notice that it is precisely in this proof that we use the estimate in anisotropic spaces with order function $m'=m-1/N$, see Lemma \ref{lemma:modified_sobolev_estimate}.
\end{remark}
Using the previous commutation relation (\ref{eq:commutation_on_sm2}), we finally get using Lemma \ref{lemma:bound_inverse_sm} and the assumption (\ref{eq:quasimode_on_sm}) that $h\mathcal{U}^-u_h$ is of order $h^{-m-1}o(h^{m+2})=o(h)$ in the $\mathcal{H}'^N(SM,H^0(F,\El^\K)\otimes E_s^*)$ norm. We are back in the more familiar situation
\begin{align*}
&\|P_h(z_h)u_h\|_{\mathcal{H}^N_h}=o(h^{m+2}),\\
    &\|h\mathcal{U}^-u_h\|_{\mathcal{H}'^N_h}=o(h).
\end{align*}
By continuity of the contraction of tensors in (vector valued) anisotropic spaces, the latter horocyclic estimate implies that $\|h\nabla^{H^0(F,\El^{\K(h)})}_{U}u_h\|_{\mathcal{H}'^N_h}=o(h)$ for any $U\in C^{\infty}(SM,E^u)$.  We thus obtain that $h\,\nabla_{U^{\mathbb{H}_F}}^{\mathbf{k}(h)}u_h=o_{\mathcal{H}'^N_h(F,\El^{\K(h)})}(h)$ by the intertwining property of $\iota_\K$, see (\ref{eq:intertwinconnection}). Thus we are once again in the setting where Proposition \ref{prop:horo_invariance_quasimode_re_larger_minus1} holds, giving the bound of Lemma \ref{lemma:lipschitz_bound}. Since we still have that $\Re(z_h)\to \nu$ with $\nu>-(n-1)/2$, the bound (\ref{eq:exp_bound_contradiction}) yields once again a contradiction. 
If $n$ was even, the proof above applies: one just needs to be cautious as the last resolvent estimate before the critical line $\Re(z)=-(n-1)/2$ will take place on the band $[-(n-1)/2+\varepsilon,-(n-1)/2+1/2]$.

\subsection{Resolvent estimates in deeper bands}\label{section:estimates_deeper_band}
From now on, the dimension of $M$ may once again be equal to $3$.

We proceed again by induction to prove the estimate (\ref{eq:case-k_deep}), the large $\langle \Im(z)\rangle$ case is dealt with similarly. Let $k\in \mathbb{N}_0$ and assume that we have the following estimates for $k\in \mathbb{N}_0$ and $|\Im(z)|\ge \lambda_{\varepsilon}(k)+1$, if $n$ is odd
\begin{align}\label{eq:estimate_large_band_odd}
    &\|(-\X_{\K}-z)^{-1}\|_{\mathcal{H}^{N}_{\langle \K \rangle^{-1},\mathrm{hol}}}\leq C_{\varepsilon,k}\langle \mathbf{k}\rangle^{-\tfrac{n-1}{2}-k-1},  
    \\
    &\text{for }\Re(z)\in \bigl[-\tfrac{n-1}{2}-k-1+\varepsilon,-\tfrac{n-1}{2}-k-\varepsilon\bigl]\notag,
\end{align}
if $n$ is even
\begin{equation}\label{eq:estimate_large_band_even}
\begin{split}
    & \|(-\X_{\K}-z)^{-1}\|_{\mathcal{H}^{N}_{\langle \K \rangle^{-1},\mathrm{hol}}} \leq C_{\varepsilon,k}\langle \mathbf{k}\rangle^{-\tfrac{n-1}{2}-k+\tfrac{1}{2}}, \\
    & \Re(z) \in [-\tfrac{n-1}{2}-k+\tfrac{1}{2},-\tfrac{n-1}{2}-k+1-\varepsilon] \cup[-\tfrac{n-1}{2}-k+1+\varepsilon,-\tfrac{n-1}{2}-k+\tfrac{3}{2}],
\end{split}
\end{equation}
To keep notations simpler, we only deal with the odd $n$ case. The proof for the even $n$ case will follow similarly. We aim at proving (\ref{eq:estimate_large_band_odd}) with $k$ replaced by $k+1$ in order to conclude by induction. Assuming (\ref{eq:estimate_large_band_odd}) does not hold in $\mathcal{B}_{k+1}(\varepsilon)$
    \begin{align}\label{eq:deeper_band_semiclassical_pb}
\begin{cases}
(-ih\X_{\mathbf{k}(h)}+h\Im(z(h))-ih\Re(z(h))u_h = o(h^{-\tfrac{n-1}{2}-k-3}), 
\\[6pt]
\|u_h\|_{\mathcal{H}_{\langle \mathbf{k}(h)\rangle^{-1},\mathrm{hol}}^{N}} = 1.
\end{cases}
\end{align}
where 
\begin{align*}
    \Re(z_h)\to \nu \in [-\dfrac{n-1}{2}-k-2+\varepsilon,-\dfrac{n-1}{2}-k-1-\varepsilon]\, \text{ , } \dfrac{\Im(z_h)}{\langle \mathbf{k}(h)\rangle }\to \eta \in [-1,1]\,\text{ and } h\mathbf{k}(h)\to \mathbf{l}.
\end{align*}
The order of the anisotropic space $N$ is assumed to be strictly larger than $\frac{n-1}{2}+k+2$. 
\begin{proposition}\label{prop:measure_first_property_deep}
    The sequence $(u_h)$ converges in the sense of semiclassical measures to a Radon measure $\mu$ on $\mathbb{H}_F^*$. This measure satisfies 
    \begin{itemize}
        \item $\supp(\mu)\subset p^{-1}\bigl(\{-\eta\}\bigr)$
        \item $(\Phi_t^{\omega_0,p})_*\mu \,=\, e^{2\nu t}\mu$
        \item $\supp(\mu)\subset \Gamma_+(-\eta)$ \text{ and } $\mu(V)>0$ \text{ for any open neighborhood $V$ of $K_{-\eta}$}.
    \end{itemize}
Moreover, the sequence $(u_h)$ satisfy for any $U\in C^{\infty}(SM,E^u)$ 
\begin{align}
    h\,\nabla_{U^{\mathbb{H}_F}}^{\mathbf{k}(h)}u_h=o_{{\mathcal{H}}_h^N(F,\El^{\K(h)})}(h).
\end{align}
\end{proposition}
\begin{proof}
    The first items are proved in Proposition \ref{prop:semiclassicalmeasure_first_properties}. Concerning the horocyclic invariance property, we notice that the proof of Lemma \ref{eq:bound_inverse_sm} did not use the assumption $\nu>-\tfrac{n-1}{2}$. In fact, the proof of the horocyclic invariance property of Section \ref{subection:resolvent_estim_re_bigger_vertical1} carries directly for our sequence $(u_h)$.
\end{proof}
In the case $\nu>-\tfrac{n-1}{2}$, we used the Lipschitz upper bound for $\mu$ from Lemma \ref{lemma:lipschitz_bound} to deduce the contradiction in (\ref{eq:exp_bound_contradiction}). Here, since $\nu<-\tfrac{n-1}{2}$, we will use instead the Lipschitz lower bound on $\mu$ from Lemma \ref{lemma:lipschitz_bound}.
We use the notation from the end of Section \ref{section:resolventestimatesfirstband}. In particular, there exists a covering $(W_i)$ of $F$ satisfying Lemma \ref{lemma:transversality_horocyclic}, giving in turn the mapping $\psi_i$ defined in (\ref{eq:parametrisation_horocyclic}). There exist constants $C_1,C_2>0$ such that 
\begin{align}\label{eq:inclusion_propagation}
    \bigcup_i \psi_i(B^i_{C_2\delta e^{-t}}) \subset U_{C_1\delta_0 e^{-t}} \subset \Phi_{-t}^{\omega_0,p}(U_{\delta_0}), \forall t>0,
\end{align}
which gives for any $i$ that $\mu(\Phi_{-t}^{\omega_0,p}(U_{\delta_0}))\geq \mu(\psi_i(B^i_{C_2\delta e^{-t}}))$. From (\ref{eq:inclusion_propagation}), one uses the Lipschitz lower bound from Proposition \ref{lemma:lipschitz_bound} together with the propagation property of Proposition  \ref{prop:measure_first_property_deep} to get
\begin{align*}
    e^{2\nu t}\gtrsim e^{-(n-1)t},
\end{align*}
with a time independant implicit constant. But since $2\nu<-(n-1)$, this is impossible. The large imaginary part regime follows exactly for the same reasons.\\ 

\textbf{There is only a finite number of resonances in $\mathcal{B}_k(\varepsilon)$ } This is done by induction: assume that the set of resonances is finite in $\mathcal{B}_k(\varepsilon)$. Once again the elements of $\sigma_{\mathrm{PR}}(\X_\K)$ are written $(z_{j,\K})$ and the corresponding resonant states are $(u_{j,\K})$.
Assume for the sake of contradiction that $\cup_\K \sigma_{\mathrm{PR}}(\X_{\K})\cap\mathcal{B}_{k+1}(\varepsilon)$ is infinite and write the corresponding sequence of resonant states $(u_h)_h$. We only deal with the large $\langle \K \rangle$ case as the large $\langle \Im(z) \rangle$ case follows similarly.

 That $\langle \mathbf{k}_j\rangle$ must go to infinity is justified at the end of Section \ref{section:resolventestimatesfirstband}. Once again, we obtain a sequence of resonant states $(u_h)$ satisfying $(-\X_{\K(h)}-z_h)u_h=0$. The only point that will differ with the case of the resolvent estimates concerns horocyclic invariance. The relation (\ref{eq:commutation_on_sm2}) yields that $h\mathcal{U}^-u_h$ satisfies $\mathbf{P}^{H^0(F,\El^\K)\otimes E_s^*}_h(z_h+ih)(h\mathcal{U}^-u_h)=0$. To apply the induction hypothesis, apply decomposition (\ref{eq:decomposition_odd}) or (\ref{eq:decomposition_even}) depending on the parity of $n$ to retrieve that the components of $h\mathcal{U}^-u_h$ are all eventually zero for $h$ small enough. Indeed, the components of $h\mathcal{U}^-u_h$ with respect to decomposition (\ref{eq:decomposition_odd}) are resonant states for $\mathbf{P}_h$ in $\mathcal{B}_k(\varepsilon)$. Now that we have (exact) horocyclic, the rest of the arguments for quasimodes apply in this setting.

\bibliographystyle{alpha}
\bibliography{Biblio}

@article{Dol98,
  author  = {Dmitry Dolgopyat},
  title   = {On decay of correlations in Anosov flows},
  journal = {Annals of mathematics},
  year    = {1998},
  volume  = {147},
  pages   = {357-390},
}

@article{Liv,
    author  = {Carlangelo Liverani},
  title   = {On contact Anosov flows},
  journal = {Annals of mathematics},
  year    = {2004},
  volume  = {159},
  pages   = {1275–1312},
}

@article{NZ14,
  author  = {Stéphane Nonnenmacher and Maciej Zworski},
  title   = {Decay of correlations for normally hyperbolic trapping},
  journal = {Inventiones mathematicae},
  volume  = {200},
  number  = {2},
  pages   = {345-438},
  year    = {2014},

}

@article{FT,
  author  = {Frédéric Faure and Masato Tsujii},
  title   = {Micro-local analysis of contact Anosov flows and band structure of the Ruelle spectrum},
  journal = {Communications of the American mathematical society},
  volume  = {4},
  pages   = {641-745},
  year    = {2024},
}

@article{FS11,
  author  = {Frédéric Faure and Johannes Sjöstrand},
  title   = {Upper Bound on the Density of Ruelle Resonances for Anosov Flows},
  journal = {Communications in Mathematical Physics},
  volume  = {308},
  pages   = {325--364},
  year    = {2011},
}

@article{Dy16,
  author  = {Semyon Dyatlov},
  title   = {Spectral gaps for normally hyperbolic trapping},
  journal = {Annales de l'Institut Fourier},
  volume  = {66},
  pages   = {55--82},
  year    = {2016},
}

@article{HK90,
  author  = {Steve Hurder and Anatole Katok},
  title   = {Differentiability, rigidity and Godbillon-Vey classes for Anosov flows},
  journal = {Publications Mathématiques de l'IHÉS},
  volume  = {72},
  pages   = {5--61},
  year    = {1990},
}

@article{BFL90,
  author  = {Yves Benoist and Patrick Foulon and François Labourie},
  title   = {Flots d'Anosov à distributions de Liapounov différentiables. I},
  journal = {Annales de l'Institut Henri Poincaré, Physique théorique},
  volume  = {53},
  number  = {4},
  pages   = {395--412},
  year    = {1990},
}

@article{CG21,
  author  = {Mihajlo Ceki\'{c} and Colin Guillarmou},
  title   = {First Band of Ruelle Resonances for Contact Anosov Flows in Dimension 3},
  journal = {Communications in Mathematical Physics},
  volume  = {386},
  number  = {2},
  pages   = {1289--1318},
  year    = {2021},
}

@misc{CL24,
  author       = {Mihajlo Ceki\'{c} and Thibault Lefeuvre},
  title        = {Semiclassical analysis on principal bundles},
  note         = {arXiv:2405.14846},
  year         = {2024},
}

@article{PZ25,
  author  = {Mark Pollicott and Daofei Zhang},
  title   = {Rapid mixing for compact group extensions of hyperbolic flows},
  journal = {Transactions of the American Mathematical Society},
  year    = {2025},
}

@article{BG80,
  author  = {Matthew Brin and Mikhail Gromov},
  title   = {On the Ergodicity of Frame Flows},
  journal = {Inventiones Mathematicae},
  volume  = {60},
  pages   = {1-7},
  year    = {1980},
}

@article{CLMS24a,
  author  = {Mihajlo Ceki\'{c} and Thibault Lefeuvre and Andrei Moroianu and Uwe Semmelmann},
  title   = {On the ergodicity of the frame flow on even‑dimensional manifolds},
  journal = {Inventiones Mathematicae},
  volume  = {238},
  pages   = {1067--1110},
  year    = {2024},
}

@article{CLMS24b,
  author  = {Mihajlo Ceki\'{c} and Thibault Lefeuvre and Andrei Moroianu and Uwe Semmelmann},
  title   = {On the ergodicity of unitary frame flows on Kähler manifolds},
  journal = {Ergodic Theory and Dynamical Systems},
  volume  = {44},
  number  = {8},
  pages   = {2143--2172},
  year    = {2024},
  doi     = {10.1017/etds.2023.72},
}

@incollection{Moore1987,
  author    = {Calvin C. Moore},
  title     = {Exponential decay of correlation coefficients for geodesic flows},
  booktitle = {Group representations, ergodic theory, operator algebras, and mathematical physics},
  editor    = {Calvin C. Moore},
  series    = {Mathematical Sciences Research Institute Publications},
  volume    = {6},
  pages     = {163--181},
}

@article{DGH2021,
  author  = {Benjamin Delarue and Colin Guillarmou and Charles Hadfield},
  title   = {Spectral theory of the frame flow on hyperbolic 3-manifolds},
  journal = {Annales Henri Poincaré},
  volume  = {22},
  pages   = {3565--3617},
  year    = {2021},
  note = { - B. Delarue was formerly known as B. Küster},
}

@article{SW21,
  author  = {Pratyush Sarkar and Dale Winter},
  title   = {Exponential mixing of frame flows for convex cocompact hyperbolic manifolds},
  journal = {Compositio Mathematica},
  volume  = {157},
  number  = {12},
  pages   = {2585--2634},
  year    = {2021},

}

@article{LP23,
  author  = {Jialun Li and Wenyu Pan},
  title   = {Exponential mixing of geodesic flows for geometrically finite hyperbolic manifolds with cusps},
  journal = {Inventiones Mathematicae},
  volume  = {231},
  number  = {3},
  pages   = {931--1021},
  year    = {2023},

}

@article{LiPanSarkar2026,
  author  = {Jialun Li and Wenyu Pan and Pratyush Sarkar},
  title   = {Exponential mixing of frame flows for geometrically finite hyperbolic manifolds},
  journal = {Journal of the European Mathematical Society},
  year    = {2026},
}

@article{BCG95,
  author  = {Gérard Besson and Gilles Courtois and Sylvain Gallot},
  title   = {Entropies et rigidit{\'e}s des espaces localement sym{\'e}triques de courbure strictement n{\'e}gative},
  journal = {Geometric and Functional Analysis},
  volume  = {5},
  pages   = {731--799},
  year    = {1995},
}

@article{B25,
  author       = {Louis-Brahim Beaufort},  
  title        = {On {K}anai's conjecture for frame flows over negatively curved manifolds},
  note         = {arXiv:2509.09500},
  year         = {2025},
}

@book{Sepanski2007,
  author    = {Mark R. Sepanski},
  title     = {Compact Lie Groups},
  series    = {Graduate Texts in Mathematics},
  volume    = {235},
  publisher = {Springer, New York, NY},
  year      = {2007},
  doi       = {10.1007/978-0-387-49158-5},
  isbn      = {978-0-387-49158-5},
}

@article{GuillarmouDePoyferre2022,
  author  = {Colin Guillarmou and Thibault de Poyferr\'{e}},
  title   = {A paradifferential approach for hyperbolic dynamical systems and applications},
  journal = {Tunisian Journal of Mathematics},
  volume  = {4},
  number  = {4},
  pages   = {673--718},
  year    = {2022},
}

@book{DZ19,
  author    = {Dyatlov, Semyon and Zworski, Maciej},
  title     = {Mathematical Theory of Scattering Resonances},
  series    = {Graduate Studies in Mathematics},
  volume    = {200},
  publisher = {American Mathematical Society},
  year      = {2019},
  address   = {Providence, Rhode Island},
  isbn      = {978-1-4704-4366-5},
  doi       = {10.1090/gsm/200}
}

@article{KusterWeich2021,
  author  = {K{\"u}ster, Benjamin and Weich, Tobias},
  title   = {Quantum-Classical Correspondence on Associated Vector Bundles Over Locally Symmetric Spaces},
  journal = {International Mathematics Research Notices},
  year    = {2021},
  volume  = {2021},
  number  = {11},
  pages   = {8225--8296},
  doi     = {10.1093/imrn/rnz068}
}

@book{Lefeuvre2026,
  author    = {Lefeuvre, Thibault},
  title     = {Microlocal Analysis in Hyperbolic Dynamics and Geometry},
  series    = {Cours Spécialisés},
  volume    = {32},
  publisher = {Société Mathématique de France},
  year      = {2026},
  isbn      = {978-2-37905-221-7},
  note      = {With a contributed chapter by Yann Chaubet},
  url       = {https://bookstore.ams.org/view?ProductCode=COSP/32}
}

@book{Anosov1967,
  author  = {Anosov, Dmitri V.},
  title   = {Geodesic Flows on Closed Riemannian Manifolds with Negative Curvature},
  journal = {Proceedings of the Steklov Institute of Mathematics},
  year    = {1967},
  volume  = {90},
  publisher = {American Mathematical Society},
  address = {Providence},
  note    = {Originally published in Trudy Matematicheskogo Instituta Imeni V.A. Steklova}
}

@article{Humbert2025,
  title   = {First Ruelle resonance for an Anosov flow with smooth potential},
  volume  = {45},
  DOI     = {10.1017/etds.2024.131},
  number  = {8},
  journal = {Ergodic Theory and Dynamical Systems},
  publisher={Cambridge University Press},
  author  = {Humbert, Tristan},
  year    = {2025},
  pages   = {2439–2475}
}

@article{DyatlovFaureGuillarmou2015,
  author  = {Dyatlov, Semyon and Faure, Frédéric and Guillarmou, Colin},
  title   = {Power spectrum of the geodesic flow on hyperbolic manifolds},
  journal = {Analysis \& PDE},
  volume  = {8},
  number  = {4},
  pages   = {923--1000},
  year    = {2015},
  doi     = {10.2140/apde.2015.8.923}
}

@misc{CekicMicrolocalMethods,
  author       = {Mihajlo Ceki{\'c}},
  title        = {Microlocal Methods in Dynamical Systems},
  note         = {Lecture notes, University of Zurich},
  year         = {2022},
  url          = {https://user.math.uzh.ch/cekic/Microlocal-Methods.pdf},

}

@book{FultonHarris1991,
  author    = {William Fulton and Joe Harris},
  title     = {Representation Theory: A First Course},
  series     = {Graduate Texts in Mathematics},
  volume     = {129},
  publisher  = {Springer},
  address    = {New York},
  year        = {1991}
}

@book{GoodmanWallach2009,
  author    = {Roe Goodman and Nolan R. Wallach},
  title     = {Symmetry, Representations, and Invariants},
  series     = {Graduate Texts in Mathematics},
  volume     = {255},
  publisher  = {Springer},
  address    = {New York},
  year       = {2009},
  isbn       = {978-0-387-79851-6},
  doi        = {10.1007/978-0-387-79852-3}
}

@article{BurqSchlenker1998,
  author    = {Burq, Nicolas and Schlenker, Jean-Marc},
  title     = {Contrôle de l'équation des ondes dans des ouverts comportant des coins},
  journal   = {Bulletin de la Société Mathématique de France},
  volume    = {126},
  number    = {4},
  pages     = {601--637},
  year      = {1998},
  publisher = {Société mathématique de France},
  doi       = {10.24033/bsmf.2338},
}

@article{Burq1997,
  author    = {Burq, Nicolas},
  title     = {Contr{\^o}labilit{\'e} exacte des ondes dans des ouverts peu r{\'e}guliers},
  journal   = {Asymptotic Analysis},
  volume    = {14},
  number    = {2},
  pages     = {157--191},
  year      = {1997},
  doi       = {10.3233/ASY-1997-14203},
}

\end{document}